\documentclass[hidelinks,onefignum,onetabnum]{siamart250211}

\usepackage{lipsum}
\usepackage{amsfonts}
\usepackage{graphicx}
\usepackage{epstopdf}
\usepackage{algorithmic}
\ifpdf
  \DeclareGraphicsExtensions{.eps,.pdf,.png,.jpg}
\else
  \DeclareGraphicsExtensions{.eps}
\fi

\newsiamremark{remark}{Remark}
\newsiamremark{hypothesis}{Hypothesis}
\crefname{hypothesis}{Hypothesis}{Hypotheses}
\newsiamthm{claim}{Claim}
\newsiamremark{fact}{Fact}
\crefname{fact}{Fact}{Facts}

\headers{An Example Article}{D. Doe, P. T. Frank, and J. E. Smith}

\title{An Example Article\thanks{Submitted to the editors DATE.
\funding{This work was funded by the Fog Research Institute under contract no.~FRI-454.}}}

\author{Dianne Doe\thanks{Imagination Corp., Chicago, IL 
  (\email{ddoe@imag.com}, \url{http://www.imag.com/\string~ddoe/}).}
\and Paul T. Frank\thanks{Department of Applied Mathematics, Fictional University, Boise, ID 
  (\email{ptfrank@fictional.edu}, \email{jesmith@fictional.edu}).}
\and Jane E. Smith\footnotemark[3]}

\usepackage{amsopn}

\usepackage{adjustbox}
\usepackage{sidecap}
\usepackage{subcaption}
\usepackage[utf8]{inputenc} % allow utf-8 input
\usepackage[T1]{fontenc}    % use 8-bit T1 fonts
\usepackage{booktabs}       % professional-quality tables
\usepackage{amsfonts}       % blackboard math symbols
\usepackage{nicefrac}       % compact symbols for 1/2, etc.
\usepackage{microtype}      % microtypography
\usepackage{xcolor}         % colors

\usepackage{framed}
\usepackage{MnSymbol}
\usepackage{comment}
\usepackage{xargs}                      % Use more than one optional parameter in a new commands
\usepackage{graphicx}
\usepackage{pgf, tikz}
\usepackage{pgfplots}
\usepackage{diagbox}

\newtheorem{assumption}{Assumption}

\expandafter\def\expandafter\normalsize\expandafter{%
    \normalsize%
    \setlength\abovedisplayskip{2pt}%
    \setlength\belowdisplayskip{2pt}%
    \setlength\abovedisplayshortskip{0pt}%
    \setlength\belowdisplayshortskip{-10pt}%
}
\newcommand{\Value}[2]{V_{#1}^{#2}} % #1 is the policy, #2 is the state
\newcommand{\Q}[2]{Q_{#1}^{#2}} % #1 is the policy, #2 is the state
\newcommand{\Qsp}[2]{G_{#1}^{#2}} % #1 is the policy, #2 is the state
\newcommand{\A}[2]{A_{#1}^{#2}}
\newcommand{\freq}[2]{\mu_{#1}^{#2}} % #1 is the policy, #2 is the state
\newcommand{\freqS}[2]{d_{#1}^{#2}} % #1 is the policy, #2 is the state

\newcommand{\ValRob}[1]{V_{#1}^\star} % #1 is the policy

\newcommand{\mb}{\mathbb}
\newcommand{\mc}{\mathcal}

\DeclareMathOperator*{\argmax}{argmax}
\DeclareMathOperator*{\argmin}{argmin}

\usepackage[british]{babel}
\usepackage{hyperref}  
\usepackage{enumerate}
\usepackage{enumitem}

\usepackage{algorithm}
\usepackage{algorithmic}
\newcommand{\pistar}{\pi^\star}
\newcommand{\pik}{\pi^{(k)}}
\newcommand{\pikp}{\pi^{(k+1)}}

\newcommand{\Pt}{\xi^{(m)}}
\newcommand{\Pk}{P^{(k)}}
\newcommand{\Ptp}{\xi^{(m+1)}}

\allowdisplaybreaks

\headers{Policy Gradient Algorithms for Robust MDPs}{M. Li, D. Kuhn, and T. Sutter}

\title{Policy Gradient Algorithms for Robust MDPs with Non-Rectangular Uncertainty Sets}
\author{Mengmeng Li\thanks{Risk Analytics and Optimization Chair, EPFL, 1015 Lausanne, Switzerland 
  (\email{mengmeng.li@epfl.ch}, \email{daniel.kuhn@epfl.ch}).}
\and Daniel Kuhn\footnotemark[1]
\and Tobias Sutter\thanks{Department of Economics, University of St.Gallen, 9000 St.\ Gallen, Switzerland 
  (\email{tobias.sutter@unisg.ch}).}
}
\begin{document}

\maketitle

\begin{abstract}
We propose policy gradient algorithms for robust infinite-horizon Markov decision processes (MDPs) with non-rectangular uncertainty sets, thereby addressing an open challenge in the robust MDP literature. Indeed, uncertainty sets that display statistical optimality properties and make optimal use of limited data often fail to be rectangular. Unfortunately, the corresponding robust MDPs cannot be solved with dynamic programming techniques and are in fact provably intractable. We first present a randomized projected Langevin dynamics algorithm that solves the robust policy evaluation problem to global optimality but is inefficient. We also propose a deterministic policy gradient method that is efficient but solves the robust policy evaluation problem only approximately, and we prove that the approximation error scales with a new measure of non-rectangularity of the uncertainty set.
Finally, we describe an actor-critic algorithm that finds an $\epsilon$-optimal solution for the robust policy improvement problem in~$\mc O(1/\epsilon^4)$ iterations. We thus present the first complete solution scheme for robust MDPs with non-rectangular uncertainty sets offering global optimality guarantees. Numerical experiments show that our algorithms compare favorably against state-of-the-art methods. 
\end{abstract}

% REQUIRED
\vspace{-5pt}


% REQUIRED
\vspace{-5pt}

% \section{Related Work}
%\input{text/relatedWork.tex}
% \begin{itemize}
%     \item motivation for non-rectangularity: small deviations, large deviations, moderate deviations
%     \item time-inconsistency
%     \item This research was supported by the Swiss National
% Science Foundation under the NCCR Automation, grant agreement 51NF40 180545.
% \end{itemize}
\vspace{-10pt}
\section{Introduction}

Markov decision processes (MDPs) form the backbone of reinforcement learning and dynamic decision-making \cite{ref:BerTsi-96, puterman2005markov,sutton2018reinforcement,ref:book:meyn-22}. Classical MDPs operate in a time-invariant stochastic environment represented by a known constant transition kernel. 
%This assumption significantly simplifies the analysis since the test and training environments are identical.
%In the context of (offline) reinforcement learning, this suggests that the test and training environments are identical.
In most applications, however, the transition kernel is only indirectly observable through a state-action trajectory generated under a fixed policy. In addition, it may even change over time. Uncertain and non-stationary transition kernels are routinely encountered, for example, in finance, healthcare or robotics etc.\ \cite{goh2018data,sun2023reinforcement,wang2023offline}. In these applications it is thus expedient to work with \textit{robust} MDPs \cite{ref:Nilim-05,ref:White-94,wiesemann2013robust},
%In addition, even when assuming that training and test data are coming from the same transition kernel, using a point estimate such as maximum likelihood estimator is susceptible to statistical estimation errors and suffers from poor out-of-sample performance when the sample size is limited.
%To overcome these obstacles, robust MDPs have emerged as an effective and promising approach. Robust MDPs 
which assume that the unknown true transition kernel falls within a known uncertainty set and aim to identify a policy that exhibits the best performance under the worst-case transition kernel in this uncertainty set. Optimal policies of robust MDPs display a favorable out-of-sample performance when the transition kernel must be estimated from scarce data or changes over time~\cite{mannor2016robust,wang2022convergence}.
% {\color{blue} 
Robust MDPs are also popular in machine learning---particularly in inverse reinforcement learning with expert demonstrations or in offline reinforcement learning with time-varying environments~\cite{chae2022robust,viano2021robust,viano2022robust,blanchet2023double}.
% }

The literature on robust MDPs distinguishes rectangular and non-rectangular uncertainty sets. An uncertainty set is called $(s)$-rectangular (or $(s,a)$-rectangular) if it is representable as a Cartesian product of separate uncertainty sets for the transition probabilities associated with the different current states~$s$ (or current state-action pairs $(s,a)$). Otherwise, the uncertainty set is called non-rectangular. Rectangularity is intimately related to computational tractability. Indeed, robust MDPs with rectangular polyhedral uncertainty sets can be solved in polynomial time, whereas robust MDPs with non-rectangular polyhedral uncertainty sets are NP-hard~\cite{wiesemann2013robust}. Most existing papers on robust MDPs focus on rectangular uncertainty sets. However, statistically optimal uncertainty sets often fail to be rectangular. Indeed, classical Cram\'er-Rao bounds imply that non-rectangular ellipsoidal uncertainty sets around the maximum likelihood estimator of the transition kernel constitute---in an asymptotic sense---the smallest possible confidence sets for the ground truth transition kernel (see~\cite[\S~5]{wiesemann2013robust} and Appendix~\ref{example:MLE:detail}). Results from large deviations theory further imply that non-rectangular conditional relative entropy uncertainty sets lead to polices that display an optimal trade-off between in-sample performance and out-of-sample disappointment~\cite{ref:Sutter-19,li2021distributionally}.

%In an average-reward setting, the uncertainty sets studied in~\cite{ref:Sutter-19,li2021distributionally} have been shown to enjoy statistical optimality properties using ideas from large deviations theory but are also non-rectangular. Using such uncertainty sets achieves a statistically optimal tradeoff between in-sample cost and out-of-sample disappointment.

% \begin{example}[Confidence sets from maximum likelihood estimation]\label{example:MLE}
% For a state-action sequence $(s_0,a_0,\ldots,s_{n-1},a_{n-1})\in (\mc S\times \mc A)^{n}$ generated under a known policy $\pi,$ we can use maximum likelihood estimation to construct an uncertainty set~$\mc P$ that contains the unknown, true, transition matrix with confidence~$1-\alpha.$ Specifically, we can set $\mc P$ to be the family of all transition matrices whose log-likelihood exceeds a threshold that depends on~$\alpha$ and~$n$. It is shown in~\cite[Section 5]{wiesemann2013robust} that this uncertainty set reduces to an ellipsoid inside the set of all transition matrices asymptotically for large~$T$.
% % and is thus induced by a parameter set of the form~\eqref{expr:ellipsoid}. 
% Clearly, this ellipsoid uncertainty set fails to be rectangular.
% \end{example}

Robust MDPs with rectangular uncertainty sets are usually addressed with value iteration, policy iteration, convex reformulation, or policy gradient methods. Value iteration constructs a sequence of increasingly accurate estimates for the value function of the optimal policy by iterating the robust Bellman operator~\cite{iyengar2005robust,ref:Nilim-05,wiesemann2013robust}, whereas
policy iteration computes a sequence of increasingly optimal policies by iteratively computing the value function of the current policy and updating it greedily~\cite{iyengar2005robust,wiesemann2013robust}. The convex reformulation method is reminiscent of the linear programming approach for non-robust MDPs~\cite{hernandez2012discrete}. It uses an exponential change of variables to construct a convex optimization problem whose solution coincides with the fixed point of an entropy-regularized robust Bellman operator~\cite{grand2022convex}. Policy gradient methods, finally, construct a sequence of increasingly optimal policies by locally updating the current policy along the policy gradient of the value function~\cite{wang2022convergence}.
Value iteration methods enjoy linear convergence and are thus theoretically faster than most known policy gradient methods, which are only guaranteed to display sublinear convergence. However, evaluating the robust Bellman operator can be costly, and value iteration methods can be 
slower than policy gradient methods for large state and action spaces~\cite{wang2022convergence}. This observation has spurred significant 
interest in gradient-based methods. A policy gradient method tailored to robust MDPs with specially structured $(s,a)$-rectangular uncertainty sets is described in~\cite{wang2022policy}, while a policy mirror descent algorithm that can handle general $(s,a)$-rectangular uncertainty sets is developed in~\cite{li2022first}. In addition, there exists a projected policy gradient method for robust MDPs with $s$-rectangular uncertainty sets~\cite{wang2022convergence,wang2024policy}.
% The finite-horizon online rectangular robust MDPs with exploration-exploitation tradeoff have been studied in~\cite{dong2022online}. 
% Analytical solutions for some $s$ and $(s,a)$-rectangular uncertainty sets in the form of $\ell_p$ balls centered at the nominal transition matrix with additional assumptions have been derived in~\cite{kumar2023efficient}. 
While this paper was under review, it has been discovered that policy gradient methods for the robust policy evaluation problem can in fact achieve linear convergence~\cite{li2023first}. We emphasize that the convergence guarantees of all reviewed solution methods for robust MDPs critically exploit a robust version of Bellman's optimality principle, which ceases to hold for non-rectangular uncertainty sets~\cite{goyal2022robust}.

To make things worse, the solution methods described above become inefficient or converge to strictly suboptimal solutions of the robust MDP if the uncertainty set fails to be rectangular. For example, value iteration outputs the optimal value function corresponding to the $s$-rectangular hull of the uncertainty set. This function provides only an upper bound on the sought value function if the uncertainty set is non-rectangular~\cite[Proposition~3.6]{wiesemann2013robust}. The corresponding optimal policy is therefore over-conservative and may perform poorly in out-of-sample tests~\cite[\S~6]{wiesemann2013robust}. Policy iteration, on the other hand, is computationally excruciating because the robust policy evaluation subroutine is already NP-hard~\cite[Theorem~1]{wiesemann2013robust}. However, there exists an efficient {\em approximate} policy iteration scheme based on ideas from robust optimization~\cite{wiesemann2013robust}. This scheme characterizes the value function of any given policy as the solution of an adjustable robust optimization problem, which can be solved approximately but efficiently in linear decision rules. However, the decision rule approximation is accurate only for small uncertainty sets. A Frank-Wolfe policy gradient method for robust policy evaluation with a non-rectangular conditional relative entropy uncertainty set is described in~\cite{li2021distributionally}. However, this method is only guaranteed to find a stationary point. A projected policy gradient method for robust MDPs with generic convex uncertainty sets is proposed in~\cite{wang2022convergence}. However, its convergence proof %relies on the assumption that the set of worst-case kernels is finite, which is difficult to check in practice. The proof also 
assumes access to a robust policy evaluation oracle. Yet no such oracle is provided.
%A similar projected gradient descent algorithm with an approximate robust policy evaluation oracle is described in~\cite{wang2022convergence}. 
In addition, its convergence proof differs methodologically from ours in that it expresses the subgradients of the worst-case net present cost as convex combinations of value function gradients evaluated at finitely many worst-case transition kernels. In contrast, we approximate the subgradients at non-differentiable points by sequences of gradients at nearby differentiable points. By focusing on a single worst-case kernel rather than relying on convex combinations, our proof is thus arguably more transparent and more directly reveals the essential gradient dominance property.

The main contributions of our paper can be summarized as follows.
\begin{enumerate}
    \item We show that robust policy evaluation problems with non-rectangular uncertainty sets can be solved to global optimality with a projected Langevin dynamics algorithm. Numerical results suggest that if the uncertainty set happens to be rectangular, then this randomized algorithm is competitive with state-of-the-art deterministic first-order methods in terms of runtime.
    \item We present a Frank-Wolfe algorithm that solves robust policy evaluation problems approximately. The approximation error is shown to scale with a new measure of non-rectangularity of the uncertainty set. We prove that the same method solves robust policy evaluation problems with rectangular uncertainty sets to any accuracy $\epsilon>0$ in $\mc O(S^2/\epsilon^2)$ iterations, where~$S$ denotes the number of states. In contrast, the iteration complexity of the state-of-the-art policy gradient method for this problem class developed in~\cite{wang2022convergence} includes an extra factor~$S^3 A$, where $A$ denotes the number of actions.
    \item We present an actor-critic method that solves robust policy improvement problems with non-rectangular uncertainty sets to any accuracy $\epsilon>0$ in~$\mc O(1/\epsilon^4)$ iterations.
    % {\color{blue} under constant policy evaluation error}. 
    This is the first complete solution scheme for robust MDPs with non-rectangular uncertainty sets offering global optimality guarantees. A similar projected gradient descent algorithm with an approximate robust policy evaluation oracle is described in~\cite{wang2022convergence}. However, the policy evaluation oracle is not made explicit for general non-rectangular uncertainty sets. We also analyze the convergence properties of our actor-critic method when the robust policy evaluation oracle can only be solved up to a fixed accuracy.
    
%    {\color{blue}A similar projected gradient descent algorithm with an approximate robust policy evaluation oracle is described in~\cite{wang2022convergence}. However, their approach differs methodologically from ours. They express subgradients of the primal function as convex combinations of value function gradients evaluated at finitely many worst-case kernels. In contrast, we approximate subgradients at non-differentiable points by starting from arbitrary differentiable points, where the subdifferential is a singleton, and using sequences of such differentiable points.}
    
%  {\color{blue}The proof strategy leading up to the gradient dominance condition of \cite{wang2022convergence} is methodologically different from ours. \cite{wang2022convergence} express any subgradient of the primal function using a convex combination of the gradient of the value function evaluated at a finite number of gradients of worst-case kernels. On the other hand, our analysis starts from an arbitrary point at which the primal function is differentiable (and thus the subdifferential of the primal function is a singleton instead of a convex combination of different points). Then, we approximate the subgradient at which the primal function is non-differentiable using sequences of differentiable points.
  % Thus, our analysis leading up to a gradient dominance of the Moreau envelope is distinct from~\cite{wang2022convergence}. }
\end{enumerate}

Our theoretical contributions critically rely on celebrated results in approximate dynamic programming and multi-agent reinforcement learning. Specifically, we adapt a policy iteration algorithm for non-robust MDPs described in~\cite{kakade2002approximately} to solve robust policy evaluation problems. In addition, the convergence analysis of our actor-critic algorithm for robust policy improvement exploits a gradient dominance result originally developed for multi-agent reinforcement learning problems with a fixed transition kernel and adapts it to single-agent MDPs with an uncertain transition kernel.

We remark that if the uncertainty set of the transition kernel is non-rectangular, then the corresponding robust MDP fails to be time consistent~\cite{loomes1986disappointment, shapiro2012time,shapiro2016rectangular}. Thus, it satisfies no Bellman-type equation and cannot be addressed with dynamic programming. Even though alternative optimality criteria are discussed in~\cite{bjork2014theory,lesmanareinventing,zhou2023time}, robust MDPs with general non-rectangular ambiguity sets remain unsolved to date.

\paragraph{Notation} We use $\Delta(\mc S)=\{p\in\mb R_+^{|\mc S|}: \sum_{s\in\mc S} p(s)=1\}$ as a shorthand for the probability simplex over a finite set~$\mc S$. Random variables are denoted by capital letters (\textit{e.g.},~$X$) and their realizations by the corresponding lowercase letters (\textit{e.g.,} $x$). For any $i,j\in \mb N$, the Kronecker delta is defined through~$\delta_{ij}=1$ if $i=j,$ and $\delta_{ij}=0$ otherwise. 
We say that a function~$f:\mc X\to \mb R$ is $\ell$-weakly convex for some~$\ell\ge 0$ if $\tilde f(x)=f(x)+\ell\|x\|^2_2/2$ is convex. 
In this case, the subdifferential $\partial f$ of $f$ is defined as $\partial f(x)=\partial \tilde f(x)-\{\ell x\}.$
In addition, we say that~$f$ is $\ell$-smooth for some $\ell\ge 0$ if it is continuously differentiable and if $\|\nabla f(x)-\nabla f(x')\|_2 \le \ell\|x-x'\|_2$ for all $x,x'\in\mc X$. The Frobenious norm of a matrix $M\in\mathbb R^{m\times n}$ is defined as $\|M\|_{\mathbf F}=(\sum_{i=1}^m\sum_{j=1}^n M_{ij}^2)^{1/2}$.

\vspace{-6pt}
\section{Rectangular and Non-rectangular Uncertainty Sets}

Consider an MDP given by a five-tuple $\left(\mathcal{S}, \mathcal{A}, P, c, \rho\right)$ comprising a finite state space $\mathcal{S}=\{1,\ldots,S\}$, a finite action space $\mathcal{A}=\{1,\ldots,A\}$, a transition kernel $P: \mathcal{S} \times \mathcal{A} \rightarrow \Delta(\mathcal{S})$, a cost-per-stage function $c: \mathcal{S} \times \mathcal{A} \rightarrow \mathbb{R}$,
and an initial distribution $\rho\in\Delta(\mathcal{S})$. 
Note that $\left(\mathcal{S}, \mathcal{A}, P, c,  \rho\right)$ describes a controlled discrete-time stochastic system, where the state at time~$t$ and the action applied at time~$t$ are denoted as random variables $S_t$ and $A_t$, respectively. 
If the system is in state $s_t\in\mc S$ at time $t$ and action $a_t\in\mc A$ is applied, then an immediate cost $c(s_t,a_t)$ is incurred, and the system moves to state $s_{t+1}$ at time $t+1$ with probability $P(s_{t+1}|s_t,a_t)$. 
% Thus, $P(\cdot|s_t,a_t)$ represents the distribution of $S_{t+1}$ conditional on $S_t=s_t$ and $A_t=a_t$. 
Actions are chosen according to a policy that prescribes a random action at time $t$ depending on the state history up to time $t$ and the action history up to time $t-1$. Throughout the rest of the paper, we restrict attention to stationary policies, which are described by a stochastic kernel $\pi\in\Pi=\Delta(\mc A)^{S}$, that is, $\pi(a_{t}|s_{t})$ denotes the probability of choosing action $a_{t}$ if the current state is~$s_t$. Unless otherwise stated, we assume without loss of generality that $c(s,a)\in[0,1]$ for all $s\in\mc S$ and $a\in\mc A.$

% Thus, $\pi(\cdot|s_{t})\in \Delta(\mc A)$ represents the distribution of~$A_{t}$ conditional on~$S_{t}=s_{t}$. 
Given a stationary policy $\pi$, there exists a unique probability measure $\mb P^{P}_\pi $ defined on the canonical sample space $\Omega=(\mathcal{S}\times\mathcal{A})^\infty$ equipped with its power set $\sigma$-algebra $\mc F=2^\Omega$
such that $\mb P^{P}_\pi (S_0=s_0)=\rho(s_0)$ for every $s_0\in\mc S$, while 
\begin{subequations}
\label{eq:def:MDP}
\begin{align}
\label{def:Q:pi}
&\mb P^{P}_\pi (S_{t}=s_{t}|S_{t-1} = s_{t-1}, A_{t-1}=a_{t-1},\ldots,S_0=s_0,A_0=a_0) = P(s_{t}|s_{t-1},a_{t-1}) 
\end{align}
and
\begin{align}
\label{def:pi:pi}
&\mb P^{P}_\pi (A_{t}=a_{t}|S_{t}=s_{t},\ldots,S_0=s_0,A_0=a_0) = \pi(a_{t}|s_{t})
\end{align} 
\end{subequations}
for all $s_1,\ldots,s_{t}\in\mc S$, $a_0,\ldots,a_t\in\mc A$, and $t\in\mb N$.
% Further details about the construction of $\mb P^{P}_\pi$ are provided in~\cite[Section 2.2]{hernandez2012discrete}. 
% For simplicity we fix the initial state-aciton pair to be $s_0$.
% Note that $\mb P^{P}_\pi$ depends not only on~$\pi$ but also on~$P$ and~$\rho$, but we suppress this dependence notationally to avoid clutter. 
% To simplify notation, we will sometimes use $X_t$ as a shorthand for the state-action pair~$(S_t,A_t)$. Note that $X_t$ ranges over $\mc X = \mathcal{S}\times\mathcal{A}$. 
We denote the expectation operator with respect to $\mb P^{P}_\pi $ by $\mb E_\pi^P[\cdot]$. One readily verifies that the stochastic process $\{S_t\}^\infty_{t=0}$ represents a time-homogeneous Markov chain under $\mb P^{P}_\pi$ with transition probabilities $\mb P^{P}_\pi(S_{t+1}=s'|S_t=s)=\sum_{a\in\mc A} P(s'|s,a)\pi(a|s)$. 
Throughout this paper, we assess the desirability of a policy by its expected net present cost with respect to a prescribed discount factor $\gamma\in(0,1)$.
\begin{definition}[Value function]
    The value function $\Value{\pi}{P}\in\mb R^{S}$ corresponding to a transition kernel $P$ and a stationary policy $\pi$ is defined through
    $$\Value{\pi}{P}(s)=\mb E_\pi^P\left[\sum_{t=0}^{\infty} \gamma^t c(S_t, A_t) \mid S_0=s\right].$$
\end{definition}
One can show that $\Value{\pi}{P}(s)$ constitutes a continuous, rational function of~$P$ and~$\pi$~\cite[Appendix A]{puterman2005markov}.
The policy evaluation problem consists in evaluating the value function~$\Value{\pi}{P}(s)$ for a fixed policy~$\pi$ and initial state~$s$, whereas the policy improvement problem seeks a policy that solves $\min_{\pi\in\Pi} \Value{\pi}{P}(s)$.

In this paper, we are interested in robust MDPs. We thus assume that the transition kernel~$P$ is only known to belong to an uncertainty set $\mc P\subseteq \Delta(\mc S)^{S\times A}$, and we assess the desirability of a policy by its worst-case expected net present cost.
% In addition, we denote the value function corresponding to a transition kernel~$P$ and a policy~$\pi$ by~$\Value{\pi}{P}$.
% where $O_j \in \mathbb{S}^q$ satisfies $O_j \preceq 0$.
% We denote a parametrized uncertainty set $\mc P$ via $\mc P^\Xi$.

\begin{definition}[Worst-case value function]
The worst-case value function $\Value{\pi}{\star}\in\mb R^S$ associated with a given policy $\pi$ and an uncertainty set $\mc P$ is defined through
\begin{equation}\label{expr:DMDP}
    \Value{\pi}{\star}(s)=\max_{P\in\mc P} \Value{\pi}{P}(s).
\end{equation}
\end{definition}
The \textit{robust policy evaluation problem} then consists in evaluating the worst-case value function~$\ValRob{\pi}(s)$ for a fixed policy~$\pi$ and initial state~$s$, and the \textit{robust policy improvement problem} aims to solve
\begin{align}\label{def:policy:learning:objective}
    \min_{\pi\in\Pi}\max_{P\in\mc P}\Value{\pi}{P}(s).
\end{align}
% We use $\pistar\in\Pi$ to denote an optimal policy  in~\eqref{def:policy:learning:objective}.
The structure of the uncertainty set $\mc P$ largely determines the difficulty of solving the robust policy evaluation and improvement problems. These problems become relatively easy if the uncertainty set is rectangular.

\begin{definition}[Rectangular uncertainty sets]
    A set $\mc P\subseteq \Delta(\mc S)^{S\times A}$ of transition matrices is called $\ $
    \begin{enumerate}[label = (\roman*), leftmargin=15pt]
        \item \textit{$(s,a)$-rectangular~\cite{iyengar2005robust}}
            if $\mc P=\prod_{(s,a) \in \mc S\times\mathcal{A}} \mathcal{P}_{s,a}$ for some $\mathcal{P}_{s,a} \subseteq \Delta(\mc S)$, $(s,a) \in \mc S\times\mathcal{A}$;
            % ; equivalently,
            % $$
            % \mathcal{P}=\left\{P \in \Delta(\mc S)^{\mc S\times\mathcal{A}} \mid P=(p(\cdot ; s,a))_{(s,a) \in \mc S\mathcal{A}} \text { and } p(\cdot ; s,a) \in \mathcal{P}_{(s,a)}, \forall(s,a) \in \mc S\mathcal{A}\right\}
            % $$
        \item \textit{$s$-rectangular~\cite{le2007robust}} if $\mc P=\prod_{s \in \mathcal{S}} \mathcal{P}_s $ for some $\mathcal{P}_s \subseteq \Delta(\mc S)^{A}, s \in \mathcal{S}$.
        %  equivalently
        % $$
        % \mathcal{P}=\left\{P \in \Delta^{\mc S\mathcal{A}} \mid P=(p(\cdot ; s,\cdot))_{x \in \mathcal{X}} \text { and } p(\cdot ; s,\cdot) \in \mathcal{P}_s,\forall x \in \mathcal{X}\right\}
        % $$
    \end{enumerate}
\end{definition}
There is also an alternative notion of rectangularity, known as $r$-rectangularity~\cite{goh2018data}, which models the transition kernel as a linear function of an uncertain factor matrix. We will not study $r$-rectangular uncertainty sets in the remainder. From now on, we 
% An uncertainty set $\mc P\subseteq \Delta(\mc S)^{S\times A}$ is called 
% if there are sets $\mathcal{P}_r \subseteq \Delta(\mc S), r\in\mc R=\{1,\ldots,R\},$ and a known factor $u \in \Delta(\mc R)^{S\times A}$ such that
        % $$
        % \mc P\!=\!\left\{P\!\in\!\Delta(\mc S)^{S\times A}  :  P(s'|s,a)\!=\!\sum_{r\in\mc R} \!p(s'|r) u(r|s,a)  ,\, p(\cdot|r)\in\mc P_r \forall s,s'\in\mc S, a\in\mc A, r\in\mc R\right\}.$$
% We 
call
an uncertainty set $\mc P\subseteq \Delta(\mc S)^{S\times A}$ \textit{non-rectangular} if it is neither $(s,a)$-rectangular nor $s$-rectangular (nor $r$-rectangular).    
As the probability simplex and thus also~$\mc P$ has an empty interior, we employ a reparametrization to represent~$\mc P$ as the image of a solid parameter set~$\Xi$.
 Specifically, we assume that there exists an affine function~$P^\xi$ that maps a solid parameter set $\Xi \subseteq \mb R^q$ to $\Delta(\mc S)^{S\times A}$ such that $\mc P=\{P^\xi : \xi\in\Xi\}. $ 
 This reparametrization may lead to a dimensionality reduction as it allows us to account for structural knowledge about the uncertainty set (\textit{e.g.}, it may be known that certain transitions are impossible or that some transitions have the same probabilities).
 % In the following we sometimes assume that  
 This reparametrization will also help us to establish algorithmic guarantees in Section~\ref{sec:PE}. 

If $\mc P$ is rectangular, then the robust policy improvement problem~\eqref{def:policy:learning:objective} can be solved in polynomial time using robust value iteration.

\begin{theorem}[Complexity of robust policy improvement with $s$-rectangular uncertainty sets~{\cite[Corollary 3]{wiesemann2013robust}}]
    \label{thm:rectangular:DP}
    If the parameter set $\Xi\subseteq\mb R^q$  is representable through~$J$ linear and convex quadratic constraints, and if~$\Xi$ induces an $s$-rectangular uncertainty set $\mc P$, then an $\epsilon$-optimal solution to the robust policy improvement problem~\eqref{def:policy:learning:objective} can be computed in polynomial time $\mc O((q+A+J)^{1 / 2}(q J+A)^3 S \log ^2 (1/\epsilon)+q A S^2 \log (1/\epsilon))$.
\end{theorem}

If the uncertainty set~$\mc P$ fails to be rectangular, on the other hand, then the robust policy evaluation problem is strongly NP-hard even if~$\mc P$ is a convex polyhedron.

\begin{theorem}[Hardness of robust policy evaluation with non-rectangular uncertainty sets~{\cite[Theorem 1]{wiesemann2013robust}}]\label{thm:hardness:robust:PE}
Deciding whether the worst-case value function~\eqref{expr:DMDP} over a non-rectangular polyhedral uncertainty set $\mc P$ exceeds a given value $\alpha$ is strongly NP-hard for any stationary policy $\pi$.
\end{theorem}

Theorem~\ref{thm:hardness:robust:PE} implies that, unless P=NP, there exists no algorithm for computing an $\epsilon$-optimal solution of the robust policy evaluation problem~\eqref{expr:DMDP} with a non-rectangular uncertainty set in time polynomial in the input size and~$\log(1/\epsilon).$ Thus, the best we can realistically hope for is to develop methods that have a runtime polynomial in~$1/\epsilon$.

\vspace{-5pt}
\section{Robust Policy Evaluation} \label{sec:PE}

Throughout this section we fix a policy $\pi\in\Pi$ and a convex and compact parameter set $\Xi$ that induces an uncertainty set $\mc P=\{P^\xi:\xi\in\Xi\}$. Our aim is to solve the robust policy evaluation problem~\eqref{expr:DMDP} to global optimality. 
The following definitions are needed throughout the paper.

\begin{definition}[Action-value function] \label{def:Q-function}
    % We overload the notation in this section by letting $r\left(S_t, A_t\right)\leftarrow r(S_t, A_t)$.
The action-value function $\Q{\pi}{P}\in\mb R^{S\times A}$ corresponding to a transition kernel $P$ and a stationary policy $\pi$ is defined through 
$$\Q{\pi}{P}(s,a)=\mb E^{P}_\pi\left[\sum_{t=0}^{\infty} \gamma^t c(S_t, A_t) | S_0=s,A_0=a\right].$$ 
\vspace{-10pt}
% and the form depending on the initial state-action distribution $\gamma$ as 
% $$\Value{\pi}{P}(\gamma)=\mb E_\pi^P\left[\sum_{t=0}^{\infty} \gamma^t r\left(S_t, A_t\right) | \pie,P, s_0\sim \gamma \right].$$ 
\end{definition}
% We further denote the worst-case action value function for any $a_0\in\mc A$ as $\Q{\pi}{\star}(s_0,a_0)=\Q{\pi}{\xi^\star_\pi}(s_0,a_0)$, where $\xi^\star_\pi$ is the worst-case uncertain parameter according to \eqref{expr:DMDP}.

\begin{definition}[Action-next-state value function]
The action-next-state value function $\Qsp{\pi}{P}\in\mb R^{S\times A\times S}$ corresponding to a transition kernel $P$ and a stationary policy~$\pi$ is defined through 
$$\Qsp{\pi}{P}(s, a,s')=\mb E_\pi^P\left[\sum_{t=0}^{\infty} \gamma^t c(S_t, A_t) | S_0=s, A_0=a,S_1=s'\right].$$ 
% and the advantage function is $ A^P(s, a,s')=- \Qsp{\pi}{P}(s, a,s')+\Q{\pi}{P}(s,a)$. 
\end{definition}

% \begin{mylem}[Relations between value functions]\label{lem:recursion:PE}
% For any $s,s'\in\mc S$ and $a\in\mc A$ we have
% \begin{enumerate}[label = (\roman*)]
%     \item \label{item:valueinQ} $\displaystyle \Value{\pi}{P}(s)=\sum_{a\in\mc A} \pi(a|s)\Q{\pi}{P}(s, a)$,
%     \item \label{item:QinG}$\displaystyle \Q{\pi}{P}(s, a)= \sum_{s'\in\mc S} P(s'|s,a) \Qsp{\pi}{P}(s,a,s')$,
%     \item \label{item:GinQ}$\displaystyle \Qsp{\pi}{P} (s, a,s')=c(s,a)+ \gamma \sum_{a'\in\mc A} \pi(a'|s') \Q{\pi}{P}(s',a').  $
% \end{enumerate}
% \end{mylem}

\begin{definition}[Discounted state visitation distribution]\label{def:visitation:distribution}
   \hfill \begin{enumerate}[label = (\roman*), leftmargin=15pt]
        \item The discounted state visitation distribution $\freqS{\pi}{P}\in\Delta(\mc S)^{S}$ corresponding to a transition kernel~$P$, a stationary policy $\pi$, and an initial state~$s_0$ is defined through
            $$%\label{def:visitfreq:S}
                \freqS{\pi}{P}(s|s_0)=(1-\gamma) \sum_{t=0}^{\infty} \gamma^t \mb P_\pi^P\left(S_t=s| S_0=s_0\right). 
            $$
        \item The discounted state-action visitation distribution 
        %(also known as occupation measure~\cite{lasserre2009moments})
        $\freq{\pi}{P}\in\Delta(\mc S\times \mc A)^{S\times A}$ corresponding to a transition kernel~$P$ and an initial state-action pair~$(s_0,a_0)$ is defined through
        $$%\label{def:visitfreq}
                \freq{\pi}{P}(s,a|s_0,a_0)=(1-\gamma) \sum_{t=0}^{\infty} \gamma^t \mb P_\pi^P\left(S_t=s,A_t=a | S_0=s_0,A_0=a_0\right). 
        $$
    \end{enumerate}
\end{definition}
Lemma~\ref{lem:recursion:PE} in the appendix shows that the value functions $\Value{\pi}{P}, \Q{\pi}{P},$ and $\Qsp{\pi}{P}$ are related through several linear equations, which imply the Bellman equation for~$\Value{\pi}{P}.$ One can use these equations to express~$\Value{\pi}{P}, \Q{\pi}{P}$, and $\Qsp{\pi}{P}$ as explicit rational functions of~$\pi$ and~$P.$ These functions are well-defined on dense subsets of $\mb R^{A\times S}$ and $\mb R^{S\times S\times A}$ and, in particular, on open neighborhoods of the physically meaningful domains $\Delta(\mc A)^S$ and $\Delta(\mc S)^{S\times A}$. In the following, we can thus assume that the functions~$\Value{\pi}{P}, \Q{\pi}{P}$, and $\Qsp{\pi}{P}$ extend to open sets containing~$\Delta(\mc A)^{S}$ and $\Delta(\mc S)^{S\times A}$. This implies in particular that the gradients of these functions with respect to~$\pi$ and~$P$ are well-defined.

\begin{remark}[Analytical formula for $V_\pi^P$]
  \label{rem:V-formula}
  One can show that $\freqS{\pi}{P}(s|s_0)$ is the $(s_0,s)$-th entry of the matrix $(1-\gamma)(I-\gamma P_\pi)^{-1}$, where $P_\pi(s,s')=\sum_{a\in\mc A}\pi(a|s) P(s'|s,a).$ If we set $r_\pi(s)=\sum_{a\in\mc A} \pi(a|s)r(s,a),$ then $\Value{\pi}{P}=(I-\gamma P_\pi)^{-1}r_\pi$ by
% We consider the  is a continuous rational function of~$P$~\cite[Appendix A]{puterman2005markov}. 
\cite[Theorem~6.1.1]{puterman2005markov}.
% Given this expression of $\Value{\pi}{P}$, the value functions $V_\pi^P, Q_\pi^P$, and $G_\pi^P$ admit an analytical extension to a $\delta$-neighborhood of the simplex $\Delta(\mc A)^{S}\times \Delta(\mc S)^{S\times A}$ for sufficiently small $\delta>0$ thanks to Lemma~\ref{lem:recursion:PE}. We will work with these analytical extensions of value functions when taking partial derivatives.  
\end{remark}

% \begin{proposition}
%     \begin{align}\label{rel:expr:occ:state}
%         \freqS{\pi}{P}(s|s_0)=\sum_{a\in\mc A}\freq{\pi}{P}(s,a|s_0,a_0)\pi(a_0|s_0).
% \end{align}
% \end{proposition}
% Equipped with these definitions, we state one auxiliary lemma, which plays a central role in the convergence proof of our proposed Algorithm. 
% First, inspired by {\cite[Lemma 6.1]{kakade2002approximately}}, we can express the sensitivity of the value function with respect to the uncertain parameter as follows.
% \begin{lemma}[Performance difference across uncertain parameters] \label{lem:pdl:Q}
%     For any $\xi,\xi'\in\Xi$, we have 
%     \begin{align*}
%         &\Value{\pi}{P}(s_0)-\Value{\pi}{\xi'}(s_0)
%         \\&= \frac{1}{1-\gamma} \sum_{s\in\mc S,a,a_0\in\mc A} \pi(a_0|s_0)\freq{\pi}{P}(s,a|s_0,a_0)  \sum_{s'\in\mc S} (P^\xi(s'|s,a)-P^{\xi'}(s'|s,a)) \Qsp{\pi}{\xi'}(s,a,s').
%     \end{align*}
% \end{lemma}
% \begin{corollary} \label{lem:pdl:Q:opt}
%     For any $P,P'\in\mc P$, $s_0,s,s'\in\mc S$ and $a\in\mc A$, we have 
%     \begin{align*}
%         &\Value{\pi}{P^\star_\pi}(s_0)-\Value{\pi}{P}(s_0)
%         \\&= \frac{1}{1-\gamma} \sum_{s\in\mc S,a,a_0\in\mc A} \pi(a_0|s_0)\freq{\pi}{P}(s,a|s_0,a_0)  \sum_{s'\in\mc S} (P^\star_\pi(s'|s,a)-P(s'|s,a)) \Qsp{\pi}{P^\star_\pi}(s,a,s').
%     \end{align*}
% \end{corollary}

A robust MDP can be viewed as a zero-sum game between the decision maker, who selects the policy~$\pi,$ and an adversary, who chooses the transition kernel $P^\xi$. In this view, the parameter $\xi$ encodes the adversary's policy. 
% {\color{blue}An explicit formula for the gradient of the value function is available with respect to the adversary's policy parameter $\xi$.}
Adopting a similar reasoning as in {\cite[Theorem 1]{sutton1999policy}}, we can thus derive an explicit formula for the gradient of the value function with respect to the adversary's policy parameter $\xi$. 
\begin{lemma}[Adversary's policy gradient] \label{lemma:policy:gradient:OPE}
For any $\xi\in\Xi$ and $s_0\in\mc S$, we have 
$$
\nabla_\xi\Value{\pi}{P^\xi}(s_0)=\frac{1}{1-\gamma}\sum_{s,s'\in\mc S,a,a_0\in\mc A}\pi(a_0|s_0) \freq{\pi}{P^\xi}(s,a|s_0,a_0) \Qsp{\pi}{P^\xi}(s, a,s')\nabla_\xi P^\xi(s'|s,a).$$
% $$
% \frac{\partial\Value{\pi}{P}(s_0)}{\partial\xi_{g(s,a,s')}}=\frac{1}{1-\gamma}\sum_{a_0\in\mc A}\pi(a_0|s_0) \freq{\pi}{P}(s,a|s_0,a_0) \Qsp{\pi}{P}(s, a,s')\frac{\partial P(s'|s,a)}{\partial \xi_{g(s,a,s')}}.$$
\end{lemma}

\begin{proof}[Proof of Lemma~\ref{lemma:policy:gradient:OPE}]
% We first derive the policy gradient for any fixed initial state-action pair $(s_0,a_0)\in\mc X$.% That is, we have
% Introduce an indicator variable 
% \begin{align*}
%     \delta_t=\begin{cases}
%     1 & \text{if } s_t=s, a_t=a, \\
%     0 & \text{otherwise}
%     \end{cases}
% \end{align*}
% for every $t\in\mb Z_+.$
By Lemma~\ref{lem:recursion:PE}\ref{item:valueinQ} and the chain rule we have
\begin{align}\label{expr:chain:rule}
    \nabla_\xi\Value{\pi}{P^\xi}(s_0)=\sum_{s,s'\in\mc S,a,a_0\in\mc A}\pi(a_0|s_0) \left.\frac{\partial \Q{\pi}{P}(s_0,a_0)}{\partial P(s'|s,a)}\right|_{P=P^\xi}\nabla_\xi P^\xi(s'|s,a).
\end{align}
Thus, it remains to find an explicit formula for the derivative of the action-value function~$\Q{\pi}{P}$ with respect to the transition kernel~$P$. A direct calculation reveals that
\begin{align}
  \nonumber  &\frac{\partial \Q{\pi}{P}(s_0,a_0)}{\partial P(s'|s,a)} 
 = \frac{\partial}{\partial P(s'|s,a)} \sum_{s_1\in\mc S} P(s_1|s_0,a_0) \Qsp{\pi}{P}(s_0,a_0,s_1) \\
  \nonumber& =\sum_{s_1\in\mc S}\left[\frac{\partial P(s_1|s_0,a_0)}{\partial P(s'|s,a)}  \Qsp{\pi}{P}(s_0,a_0,s_1)+  P(s_1|s_0,a_0) \frac{\partial \Qsp{\pi}{P}(s_0,a_0,s_1)}{\partial P(s'|s,a)} \right] \\
  \nonumber& =  \delta_{ss_0}\delta_{aa_0}\Qsp{\pi}{P}(s,a,s')+\sum_{s_1\in\mc S} P(s_1|s_0,a_0)
\\&\qquad \frac{\partial }{\partial P(s'|s,a)} \nonumber
    \left[ c(s_0,a_0)+\gamma\sum_{a_1\in\mc A} \pi(a_1|s_1)\Q{\pi}{P}(s_1,a_1)\right]
    % \right] 
    \\
    \nonumber& = \delta_{ss_0}\delta_{aa_0}\Qsp{\pi}{P}(s,a,s')+\gamma\sum_{s_1\in\mc S,a_1\in\mc A}  P(s_1|s_0,a_0)\pi(a_1|s_1)\frac{\partial \Q{\pi}{P}(s_1,a_1)}{\partial P(s'|s,a)} \\
\nonumber& = \delta_{ss_0}\delta_{aa_0}\Qsp{\pi}{P}(s,a,s')
\\\label{eq:policygradient:unroll0}&\qquad+\gamma\sum_{s_1\in\mc S,a_1\in\mc A} \mb P^P_\pi(S_1=s_1,A_1=a_1|S_0=s_0,A_0=a_0)\frac{\partial \Q{\pi}{P}(s_1,a_1)}{\partial P(s'|s,a)},
\end{align}
where the first and third equalities use Lemmas~\ref{lem:recursion:PE}\ref{item:QinG} and \ref{lem:recursion:PE}\ref{item:GinQ}, respectively. The last equality follows from the defining properties~\eqref{def:Q:pi} and~\eqref{def:pi:pi} of $\mb P^P_\pi$.  
Repeating the above reasoning for the state-action pair $(s_t,a_t)$ instead of $(s_0,a_0)$ yields
% \begin{equation}
%     \begin{split}
\begin{align*}%\label{eq:policygradient:unroll1}
    &\frac{\partial \Q{\pi}{P}(s_t,a_t)}{\partial P(s'|s,a)} =\delta_{ss_t}\delta_{aa_t} \Qsp{\pi}{P}(s,a,s')
    \\&\qquad\qquad+\gamma\!\!\! \sum_{s_{t+1}\in\mc S,a_{t+1}\in\mc A}\!\!\! \mb P^P_\pi(S_{t+1}=s_{t+1},A_{t+1}=a_{t+1}|S_t=s_t,A_t=a_t)\frac{\partial \Q{\pi}{P}(s_{t+1},a_{t+1})}{\partial P(s'|s,a)} .  
\end{align*}
Substituting the above expression for $t=1$ into~\eqref{eq:policygradient:unroll0} and recalling that $\{(S_t,A_t)\}^{\infty}_{t=0}$ constitutes a Markov chain under $\mb P^P_\pi$ yields
%     \end{split}
% \end{equation}
\begin{align*}
    \frac{\partial \Q{\pi}{P}(s_0,a_0)}{\partial P(s'|s,a)}
    &= \delta_{ss_0}\delta_{aa_0}\Qsp{\pi}{P}(s,a,s')+\gamma \mb P^P_\pi(S_1=s,A_1=a|S_0=s_0,A_0=a_0) \Qsp{\pi}{P}(s,a,s')\\&\qquad\qquad +\gamma^2\sum_{s_2\in\mc S,a_2\in\mc A} \mb P^P_\pi(S_2=s_2,A_2=a_2|S_0=s_0,A_0=a_0)\frac{\partial \Q{\pi}{P}(s_2,a_2)}{\partial P(s'|s,a)} 
\\&=\gamma^{0} \mb P^P_\pi\left(S_0=s,A_0=a \mid S_0=s_0,A_0=a_0\right) \Qsp{\pi}{P}(s,a,s')
\\&\qquad +\gamma^{1} \mb P^P_\pi\left(S_1=s,A_1=a \mid S_0=s_0,A_0=a_0\right) \Qsp{\pi}{P}(s,a,s')
\\&\qquad +\gamma^2\sum_{s_2\in\mc S,a_2\in\mc A} \mb P^P_\pi(S_2=s_2,A_2=a_2|S_0=s_0,A_0=a_0)\frac{\partial \Q{\pi}{P}(s_2,a_2)}{\partial P(s'|s,a)}.
\end{align*}
Iteratively reformulating $\partial \Q{\pi}{P}(s_t,a_t)/ \partial P(s'|s,a)$ for $t=2,3,\ldots$, we finally obtain
\begin{align*}
\frac{\partial \Q{\pi}{P}(s_0,a_0)}{\partial P(s'|s,a)}& =\sum_{t=0}^{\infty} \gamma^{t} \mb P^P_\pi\left(S_t=s,A_t=a \mid S_0=s_0,A_0=a_0\right) \Qsp{\pi}{P}(s,a,s')
\\&=\frac{1}{1-\gamma} \freq{\pi}{P}(s,a|s_0,a_0) \Qsp{\pi}{P}(s, a,s'),
\end{align*}
where the last equality exploits the definition of the discounted state visitation distribution.  
% Using Lemma~\ref{lem:recursion:PE}\ref{item:valueinQ} and chain rule concludes the proof.
The claim then follows by substituting the above expression into~\eqref{expr:chain:rule}.
\end{proof}
Lemma~\ref{lemma:policy:gradient:OPE} is a key ingredient for two complementary algorithms for solving the robust policy evaluation problem~\eqref{expr:DMDP} with a non-rectangular uncertainty set. Section~\ref{ssect:PLD} first develops a Markov chain Monte Carlo method for solving~\eqref{expr:DMDP} exactly. Next, Section~\ref{alg:CPI} develops a more efficient Frank-Wolfe method for solving~\eqref{expr:DMDP} approximately. Throughout the two sections we fix an initial state $s_0\in\mc S$.
%{\color{blue} Nonetheless, our theoretical results generalize straightforwardly to the case where the initial state~$s_0$ follows a distribution~$\rho\in\Delta(\mc S)$.}

% \input{text/1.5-PLD.tex}
\vspace{-5pt}
\subsection{Projected Langevin Dynamics}
\label{ssect:PLD}
We now develop a Markov Chain Monte Carlo method to solve the robust policy evaluation problem~\eqref{expr:DMDP} to global optimality and derive its convergence rate in expectation.
% , reminiscent of simulated annealing~\cite{granville1994simulated}.
To this end, we assume throughout this section that $\Xi\subseteq\mb R^q$ is a compact convex body, and we consider the problem of sampling from the Gibbs distribution
$$
\nu_\beta(\mathrm{d}\xi) =\frac{\displaystyle\exp \big(\beta \Value{\pi}{P^\xi}(s_0)\big)}{\displaystyle\int_\Xi \exp\big(\beta \Value{\pi}{P^{\xi'}}(s_0)\big) \mathrm{d} \xi'}\mathrm{d}\xi,
$$
where $\beta>1$ represents the inverse temperature.
Note that the denominator is finite because~$\Xi$ is compact and $\Value{\pi}{P^{\xi}}(s_0)$ is continuous in $\xi$. Indeed, $\Value{\pi}{P}(s_0)$ is continuous in~$P$~\cite[Appendix A]{puterman2005markov}, and~$P^{\xi}$ is affine in~$\xi.$
Sampling from the Gibbs distribution $\nu_\beta$ is of interest because the robust policy evaluation problem~\eqref{expr:DMDP} is equivalent to 
\begin{align}\label{PLD:obj}
    \Value{\pi}{\star}(s_0)=\max_{\xi\in\Xi}V_\pi^{P^\xi}(s_0),
\end{align}
and because~$\nu_\beta$ converges weakly to the uniform distribution on the set of global maximizers of~\eqref{PLD:obj}
as $\beta$ tends to infinity~\cite[Section 2]{hwang1980laplace}. We use the discrete-time counterpart of the Langevin diffusion~\cite{roberts1996exponential} to generate samples that are (approximately) governed by the Gibbs distribution $\nu_\beta,$ see Algorithm~\ref{alg:PLD}. In each iteration $m\in\mb Z_+$, Algorithm~\ref{alg:PLD} first uses Lemma~\ref{lemma:policy:gradient:OPE} to compute the adversary's policy gradient $\nabla_\xi \Value{\pi}{P^\xi}(s_0)$ at the current iterate $\xi=\xi^{(m)},$ perturbs it by adding Gaussian noise, and then applies a projected gradient step to find the next iterate~$\xi^{(m+1)}$. After~$M$ iterations, Algorithm~\ref{alg:PLD} outputs a random iterate $\xi^{(M)}$ whose distribution $\nu_M$ approximates $\nu_\beta$ in the $1$-Wasserstein distance~\cite[Theorem 1]{lamperski2021projected}.
\begin{algorithm}[ht!]  
\caption{Projected Langevin dynamics for solving the robust policy evaluation problem~\eqref{expr:DMDP}}
  \label{alg:PLD}
  \begin{algorithmic}[1]
  \REQUIRE Initial iterate $\xi^{(0)}\in\Xi$, Gibbs parameter $\beta>1$, stepsize $\eta >0$, \ iteration number $M\in\mb N$
  \FOR{$m=0,\ldots, M-1$}
  \STATE Sample $w_{m+1}\sim \mathcal{N}\left(0, I_q\right)$ 
  % Gaussian random variables.
    \STATE Find $\Ptp=\mathrm{Proj}_{\Xi}\left(\Pt+\eta \nabla_{\xi}\left. \Value{\pi}{P^{\xi}}(s_0)\right|_{\xi=\Pt}+\sqrt{2 \eta/\beta} w_{m+1}\right)$ \label{step:pld}
  \ENDFOR
  \end{algorithmic}
  \end{algorithm}

% (see e.g. Dalalyan and Riou-Durand $(2020))$ where , and for any fixed $k$, the random vectors $\left(\left(\xi_k\right)_2,\left(\xi_k^{\prime}\right)_2\right),\left(\left(\xi_k\right)_2,\left(\xi_k^{\prime}\right)_2\right), \ldots\left(\left(\xi_k\right)_d,\left(\xi_k^{\prime}\right)_d\right)$ are i.i.d. with the covariance matrix
% $$
% C(\eta):=\int_0^\eta\left[\psi_0(t), \psi_1(t)\right]^T\left[\psi_0(t), \psi_1(t)\right] d t
% $$
\begin{theorem}[Convergence of Algorithm~\ref{alg:PLD}]
\label{thm:convergence:PLD}
    If $\epsilon>0$, $\eta<1/2$, and $\lambda\in (0,1),$ then there exist universal constants $a>4$, $b>1$, and $c_1,c_2,c_3>0$ such that for $\beta \geq c_1^{-1}\left(2 q/(c_1(1-\lambda) \epsilon e)\right)^{1 / \lambda}$ and
% \eta=  \frac{\log M}{4 a M} $, $\beta =c_5^{-1}\left(\frac{2 n}{c_5(1-\lambda) \epsilon e}\right)^{1 / \lambda}$, 
$M\ge \max\{4,c_2\exp (c_3q^b)/\epsilon^a\}$, the distribution~$\nu_M$ of the output  $\xi^{(M)} $ of Algorithm~\ref{alg:PLD} satisfies
$\mb E_{\xi\sim\nu_M} [\Value{
\pi}{P^\xi}(s_0)]\ge \ValRob{\pi}(s_0)-\epsilon.$
% where  are constants that only depends on $a$ and $b$.
\end{theorem}   
\begin{proof}
%[Proof of Theorem~\ref{thm:convergence:PLD}]
By~\cite[Lemma 4]{wang2022convergence}, there exists a constant $L>0$ such that the objective function~$\Value{\pi}{P^\xi}(s_0) $ of problem~\eqref{PLD:obj} is $L$-smooth in~$\xi$.
% , that is, there exists $L\ge0$ such that $\|\nabla_\xi \Value{\pi}{P^\xi}(s_0)-\nabla_\xi \Value{\pi}{P^{\xi'}}(s_0)\| \leq L\|\xi-\xi'\|$ for all $\xi,\xi' \in \Xi$. 
In addition,~$\Xi$ is a convex body. The claim thus follows from~\cite[Proposition 3]{lamperski2021projected}.
\end{proof}
Theorem~\ref{thm:convergence:PLD} shows that the number of iterations~$M$ needed by Algorithm~\ref{alg:PLD} to compute an~$\epsilon$-optimal solution for the robust policy evaluation problem~\eqref{expr:DMDP} scales exponentially with the dimension~$q$ of the uncertain parameter~$\xi$ and with the number of desired accuracy digits~$\log(1/\epsilon).$ 
This is consistent with the hardness result of Theorem~\ref{thm:hardness:robust:PE}.
Nonetheless, Algorithm~\ref{alg:PLD} solves the robust policy evaluation problem via a simple gradient-based approach and enjoys global optimality guarantees even if the uncertainty set fails to be rectangular.

\begin{remark}[Implementation of Algorithm~\ref{alg:PLD}]
The following modifications can improve the scalability of Algorithm~\ref{alg:PLD} in practice. First, Algorithm~\ref{alg:PLD} computes an exact policy gradient in every iteration, which can be costly when the state and action spaces are large. Stochastic or approximate policy gradients may be cheaper to evaluate. Fortunately, Theorem~\ref{thm:convergence:PLD} continues to hold when stochastic instead of exact policy gradients are used provided that they are affected by sub-Gaussian noise~\cite[Proposition~3]{lamperski2021projected}. In addition, the projection onto the parameter space $\Xi$ is computed in every iteration, which can be costly. As~$\Xi$ is convex, however, the Euclidean projection subroutine solves a convex program and is thus amenable to efficient general-purpose solvers that scale to high dimensions~\cite{usmanova2021fast}. For specific non-rectangular polyhedral uncertainty sets, Euclidean balls, or $\ell_1$-balls, projections are available in closed form or can be computed highly efficiently with specialized methods~\cite{duchi2008efficient,liu2009efficient}.
\end{remark}

The concentration behavior of the discrete-time counterpart of the Langevin diffusion is generally open despite some recent results for convex objective functions~\cite{altschuler2022concentration}. We leave the study of strong concentration bounds complementing Theorem~\ref{thm:convergence:PLD} for future research. However, 
by applying Markov's inequality, we directly obtain the following probabilistic guarantee. 
\begin{corollary}[Probabilistic suboptimality guarantee]
Under the assumptions of Theorem~\ref{thm:convergence:PLD} we have
$\mb P_{\xi\sim\nu_M}[\Value{
\pi}{P^{\xi}}(s_0)>\Value{\pi}{\star}(s_0)-\epsilon/\delta]\ge 1-\delta$ for all $\delta\in(0,1).$
\end{corollary}

\vspace{-5pt}

\subsection{Frank-Wolfe Algorithm}
\label{ssec:FW}

The robust policy evaluation problem~\eqref{expr:DMDP} is challenging because the objective function $\Value{\pi}{P}(s_0)$ is non-concave in $P.$ Accordingly, it is not surprising that the runtime of the Markov Chain Monte Carlo method developed in Section~\ref{ssect:PLD} scales exponentially with the dimension~$q$ of $\xi$. In this section we show that a stationary point of~\eqref{expr:DMDP} can be found in time polynomial in~$q$. We will also show that the suboptimality of this stationary point \textit{vis-\`a-vis} the global maximum of~\eqref{expr:DMDP} admits a tight computable estimate that depends on the degree of non-rectangularity of the uncertainty set~$\mc P.$ 
% {\color{blue}Note that, in this section, we assume the convexity of~$\mc P$ instead of convexity of $\Xi.$}
To this end, we first note that problem~\eqref{expr:DMDP} is susceptible to a Frank-Wolfe (FW) algorithm~\cite{frank1956algorithm}, see Algorithm~\ref{alg:CPI}. 
A similar FW method has been proposed to solve the policy improvement problem associated with non-robust MDPs~\cite{bhandari2021linear}. 
This method is often referred to as conservative policy iteration \cite{kakade2002approximately}.
Algorithm~\ref{alg:CPI} can thus be also viewed as a conservative policy iteration method for robust policy evaluation problems with non-rectangular uncertainty sets.

\begin{algorithm}[htb!] 
\caption{FW algorithm for solving the robust policy evaluation problem~\eqref{expr:DMDP}}
    \label{alg:CPI}
  \begin{algorithmic}[1]
  \REQUIRE Initial iterate $P^{(0)}\in\mc P$,
  positive stepsizes $\{\alpha_m\}_{m=0}^{\infty}$, 
  tolerence $\epsilon >0$
  \STATE $m \leftarrow 0$
  \REPEAT
\STATE \label{LMO} Find an $\epsilon$-optimal solution $P_\epsilon$ of problem~\eqref{expr:direction:finding} with $P'=P^{(m)}$
% \begin{align*}
%    % &\sum_{s\in\mc S,a\in\mc A} \freqS{\pi}{P^{(m)}}(s|s_0)\pi(a|s)  \sum_{s'\in\mc S} P'(s'|s,a)\A{\pi}{P^{(m)}}(s,a,s')\\&\quad\ge\max_{P\in \mc P} \sum_{s\in\mc S,a\in\mc A} \freqS{\pi}{P^{(m)}}(s|s_0)\pi(a|s)  \sum_{s'\in\mc S} P(s'|s,a)\A{\pi}{P^{(m)}}(s,a,s')-\epsilon 
% \end{align*}
% \sum_{s\in\mc S,a\in\mc A} \freqS{\pi}{P^{(m)}}(s|s_0)\pi(a|s)  \sum_{s'\in\mc S} P'(s'|s,a)\A{\pi}{P^{(m)}}(s,a,s')
\STATE Update $P^{(m+1)}=(1-\alpha_m)P^{(m)} + \alpha_m P_\epsilon $ \label{line:interpolation}
\STATE $m \leftarrow m+1$
\UNTIL{$\langle \nabla_P \Value{\pi}{P^{(m)}}(s_0),P_\epsilon-P^{(m)}\rangle \le \epsilon$}
\RETURN $\widehat P=P^{(m)}$
\end{algorithmic}
\end{algorithm}

In each iteration~$m\in\mb Z_+$, Algorithm~\ref{alg:CPI} computes an $\epsilon$-optimal solution of the direction-finding subproblem
\begin{align}\label{expr:direction:finding}
    \max_{P\in\mc P} \langle \nabla_P \Value{\pi}{P'}(s_0),P-P'\rangle,
\end{align}
which linearizes the objective function of problem~\eqref{expr:DMDP} around the current iterate~$P'=P^{(m)}$. The next iterate~$P^{(m+1)}$ is constructed as a point on the line segment connecting~$P^{(m)}$ and~$P_\epsilon.$
The algorithm terminates as soon as the (approximate) Frank-Wolfe gap $\langle \nabla_P \Value{\pi}{P^{(m)}}(s_0),P_\epsilon-P^{(m)}\rangle$ drops below the prescribed tolerance~$\epsilon$. 
% Otherwise,  
A more explicit reformulation of the direction-finding subproblem~\eqref{expr:direction:finding} suitable for implementation is provided in Proposition~\ref{prop:deg:nonrect} below. For notational convenience, we define the adversary's advantage function as
\begin{align*}
    \A{\pi}{P} (s,a,s')= \Qsp{\pi}{P}(s,a,s') - \Q{\pi}{P}(s,a).
\end{align*}
It quantifies the extent to which the adversary prefers to set the next state to $s'$ instead of sampling it from~$P(\cdot|s,a)$, assuming that the future dynamics of the states and actions are determined by~$\pi$ and~$P.$

It can be shown that the direction-finding subproblem~\eqref{expr:direction:finding} can be equivalently expressed in terms of the adversary's advantage function.
% Note that we have $\sum_{s'\in\mc S}P(s'|s,a)\A{\pi}{P} (s,a,s')=0$ thanks to Lemma~\ref{lem:recursion:PE}~\ref{item:QinG}.
\begin{proposition}[Direction-finding subproblem]
\label{prop:deg:nonrect}
Problem~\eqref{expr:direction:finding} is equivalent to
    \begin{align}\label{expr:explicit:direction:finding}
         \frac{1}{1-\gamma}\max_{P\in\mc P}\sum_{s\in\mc S,a\in\mc A} \freqS{\pi}{P'}(s|s_0)\pi(a|s)  \sum_{s'\in\mc S} P(s'|s,a) \A{\pi}{P'}(s,a,s') .
    \end{align}
\end{proposition}

\begin{proof}
Under the trivial embedding $P^\xi=\xi$ for $\xi=P\in\Xi=\mc P$, nature's policy gradient $\nabla_\xi P^\xi(s'|s,a)\in\mb R^{S\times S\times A}$ constitutes a tensor containing a~$1$ in position $(s',s,a)$ and zeros elsewhere. Thus, Lemma~\ref{lemma:policy:gradient:OPE} implies that
\begin{align*}
 &\langle \nabla_P \Value{\pi}{P'}(s_0),P-P'\rangle
 \\&= \frac{1}{1-\gamma}\sum_{s,s'\in\mc S,a,a_0\in\mc A}\pi(a_0|s_0) \freq{\pi}{P'}(s,a|s_0,a_0)\Qsp{\pi}{P'}(s, a,s')(P(s'|s,a)-P'(s'|s,a))
 \\&=\frac{1}{1-\gamma}\sum_{s\in\mc S,a\in\mc A} \freqS{\pi}{P'}(s|s_0)\pi(a|s) \sum_{s'\in\mc S} (P(s'|s,a)-P'(s'|s,a)) \Qsp{\pi}{P'}(s, a,s')
 \\&=\frac{1}{1-\gamma}\sum_{s\in\mc S,a\in\mc A} \freqS{\pi}{P'}(s|s_0)\pi(a|s)  \sum_{s'\in\mc S} P(s'|s,a) \A{\pi}{P'}(s,a,s'),
\end{align*}
where the second equality follows from~\eqref{eq:def:MDP} and the law of total probability, which implies that $\!\sum_{a_0\in\mc A}\pi(a_0|s_0)\freq{\pi}{P'}(s,a|s_0,a_0)\!=\!\freqS{\pi}{P'}(s|s_0)\pi(a|s)$. The last equality follows from Lemma~\ref{lemma:advantage:Qsp} in the appendix. Thus,~\eqref{expr:direction:finding} and~\eqref{expr:explicit:direction:finding} are equivalent.
\end{proof}
\allowdisplaybreaks
The following assumption is instrumental for the main results of this section.

\begin{assumption}[Irreducibility]\label{assume:irreducible}
% The Markov chain $\{S_t\}_{t=0}^\infty$ induced by the given policy $\pi$ for every $P\in\mc P_{\mathsf{s}}$ has fixed initial distribution $\eta>0. $ 
The Markov chain $\{S_t\}_{t=0}^\infty$ induced by the given policy $\pi$ is irreducible for every $P\in\mc P_{\mathsf{s}}$, where $\mc P_{\mathsf{s}}$ denotes the smallest $s$-rectangular uncertainty set that contains~$\mc P$.
\end{assumption}
Assumption~\ref{assume:irreducible} ensures that, for every transition kernel $P\in\mc P_{\mathsf s}$, every state $s'$ can be reached from any other state $s$ within a finite number of transitions. Similar assumptions are frequently adopted in the literature on robust and non-robust MDPs~\cite{wiesemann2013robust,gong2020duality}.

In the following, we define the distribution mismatch coefficient associated with two transition kernels $P,P'\in\mc P_{\mathsf{s}}$ as $\delta_d(P',P)=\max_{s\in\mc S} \freqS{\pi}{P'}(s|s_0)/\freqS{\pi}{P}(s|s_0)$, and we set the universal distribution mismatch coefficient to $\delta_d=\max_{P',P\in\mc P_{\mathsf{s}}} \delta_d(P',P)$.

In addition, we define the degree of non-rectangularity of the uncertainty set~$\mc P$ with respect to an anchor point $P'\in\mc P$ as 
\begin{align*}%\label{def:deg:nonrectangularity}
    \delta_{\mc P}(P')
    % &= \max_{P\in\mc P_{\mathsf{s}}}\sum_{s\in\mc S,a\in\mc A} \freqS{\pi}{\widehat P}(s|s_0)\pi(a|s)  \sum_{s'\in\mc S} P(s'|s,a) \A{\pi}{\widehat P}(s,a,s') 
    % \\& \qquad -\max _{P \in \mc P} \sum_{s\in\mc S,a\in\mc A} \freqS{\pi}{\widehat P}(s|s_0)\pi(a|s)  \sum_{s'\in\mc S} P(s'|s,a)\A{\pi}{\widehat P}(s,a,s')
    % \\&
    =\max_{P_{\mathsf{s}}\in\mc P_{\mathsf{s}}}\langle \nabla_P \Value{\pi}{P'}(s_0),P_{\mathsf{s}}\rangle -\max_{P\in\mc P} \langle \nabla_P \Value{\pi}{P'}(s_0),P\rangle,
\end{align*}
where $\mc P_{\mathsf{s}}$ denotes again the smallest $s$-rectangular uncertainty set that contains $\mc P$, and we set the absolute degree of non-rectangularity of~$\mc P$ to $\delta_{\mc P}=\max_{P'\in\mc P} \delta_{\mc P}(P')$.
Note that if $\mc P$ is $s$-rectangular, then $\mc P_{\mathsf{s}}=\mc P$, and thus $\delta_{\mc P}(P')$ vanishes for every anchor point $P'\in\mc P,$ implying that $\delta_{\mc P}=0.$ If~$\mc P$ is non-rectangular, however, then $\mc P\subseteq\mc P_{\mathsf{s}},$ and thus $\delta_{\mc P}(P')$ is non-negative for every $P'\in\mc P$. Hence, $\delta_{\mc P}$ is non-negative, too. 
% Lemma~\ref{lemma:finite:C} ensures that the constant $C$ only depends on $\mc P$ and is finite.

\begin{remark}[Finiteness of $\delta_d$ and $\delta_{\mc P}$]\label{remark:distribution:mismatch}
Assumption~\ref{assume:irreducible} ensures that $\delta_d(P',P)>0$ for all $P,P'\in\mc P.$ As $\freqS{\pi}{P'}(s|s_0)$ and $\freqS{\pi}{P}(s|s_0)$ are respectively continuous in~$P'$ and~$P$ for all $s\in\mc S$, and as~$\mc P_{\mathsf{s}}$ is compact, it is clear that~$\delta_d$ is finite and strictly positive. Similarly, as~$\nabla_{P}\Value{\pi}{P'}(s_0)$ is continuous in~$P'$~\cite[Lemma 4]{wang2022convergence} while~$\mc P$ and~$\mc P_{\mathsf{s}}$ are compact, Berge's maximum theorem~\cite[pp.~115-116]{berge1997topological} ensures that~$\delta_{\mc P}(P')$ is continuous in~$P'.$ Thus,~$\delta_{\mc P}$ is finite and non-negative. Note that both~$\delta_d$ and~$\delta_{\mc P}$ depend only on~$\mc P$, $\pi$ and~$s_0.$
\end{remark}

% {\color{blue}
\begin{remark}[Measures of non-rectangularity]
The degree of non-rectangularity of the uncertainty set~$\mc P$ could also be quantified by the Hausdorff distance between~$\mc P$ and its $s$-rectangular hull~$\mc P_{\mathsf{s}}$. However, this alternative non-rectangularity measure is agnostic of the value function~$V_\pi^P(s)$ to be maximized over~$\mc P$. In contrast, the proposed non-rectangularity measure~$\delta_{\mc P}$ accounts for the geometry of~$V_\pi^P(s)$. For example, $\delta_{\mc P}$ vanishes even if~$\mc P$ fails to be $s$-rectangular provided that~$V_\pi^P(s)$ is constant in~$P$, say. Hence, $\delta_{\mc P}$ is small not only when $\mc P$ closely approximates $\mc P_s$ but also when replacing~$\mc P$ with~$\mc P_s$ does not significantly change the solution of the robust policy evaluation problem~\eqref{expr:DMDP}. Note that $\delta_{\mc P}$ is particularly tailored to Algorithm~\ref{alg:CPI} because it estimates the impact of replacing~$\mc P$ with~$\mc P_s$ on the direction-finding subproblem~\eqref{expr:direction:finding}.
\end{remark}
% }

The following theorem uses Assumption~\ref{assume:irreducible} to show that the FW algorithm offers a global performance guarantee.
\begin{theorem}[Global performance guarantee]\label{thm:CPI:PE}
Suppose that Assumption~\ref{assume:irreducible} holds and that $\epsilon>0.$ For every $m\in\mb Z_+$, define the approximate Frank-Wolfe gap
$$\mb G_m= \langle \nabla_P \Value{\pi}{P^{(m)}}(s_0),P_\epsilon-P^{(m)}\rangle,$$ 
where $P_\epsilon$ denotes the $\epsilon$-optimal solution of problem~\eqref{expr:direction:finding} computed in the $m$-th iteration of Algorithm~\ref{alg:CPI}, and let $\alpha_m = \mb G_m (1-\gamma)^3/(4\gamma^2)$. Then, Algorithm~\ref{alg:CPI} terminates within~$\mc O(1/\epsilon^2)$ iterations, and its output $\widehat P$ satisfies
% \begin{align*}
$\ValRob{\pi}(s_0)-\Value{\pi}{\widehat P}(s_0) \leq \delta_d(2\epsilon+\delta_{\mc P}).$
\end{theorem}

% {\color{red}
% [ML: We could also add a remark here saying how $\delta_{\mc P}$ guarantee does not depend on the algorithm used. For example, ~\cite{wang2022convergence} applied to problem~\eqref{expr:direction:finding} would have just the same rate of convergence to the same near-optimal solution, but with a worse constant.]
% }

Theorem~\ref{thm:CPI:PE} implies that if $\mathcal{P}$ is $s$-rectangular, in which case $\delta_{\mc P}=0$, then Algorithm~\ref{alg:CPI} solves the robust policy evaluation problem~\eqref{expr:DMDP} to global optimality. This insight is formalized in the next corollary.

\begin{corollary}[Global optimality guarantee when $\mc P$ is rectangular]
\label{cor:CPI-for-rectangular-sets}
Suppose that all assumptions of Theorem~\ref{thm:CPI:PE} hold and that~$\mc P$ is $s$-rectangular. Then,
the output~$\widehat P$ of Algorithm~\ref{alg:CPI} satisfies
$\ValRob{\pi}(s_0)-\Value{\pi}{\widehat P}(s_0) \leq 2\delta_d\epsilon.$
\end{corollary}

\begin{remark}[Robust policy evaluation with rectangular uncertainty sets]
A policy gradient method that solves robust policy evaluation problems with $s$-rectangular uncertainty sets to global optimality is proposed in~\cite{wang2022convergence}.
While displaying the same dependence on~$\epsilon$, one can show that the iteration complexity of this alternative method exceeds that of our algorithm by a factor~$S^3 A$.
% ; see also the {\color{red}more detailed discussion in Section~\ref{ssec:rec:amb}.}
% {\color{blue}
The theoretical analysis in Section~\ref{ssec:FW} implies that our FW algorithm requires $32\gamma^2\delta_d^2/(\delta^2(1-\gamma)^5)$ iterations to find a~$\delta$-optimal solution for problem~\eqref{expr:DMDP}. Indeed, if we set~$\epsilon=\delta/(2\delta_d)$, then the output~$\widehat P$ of Algorithm~\ref{alg:CPI} satisfies $\ValRob{\pi}(s_0)-\Value{\pi}{\widehat P}(s_0) \leq \delta$ by Corollary~\ref{cor:CPI-for-rectangular-sets}, and the algorithm terminates within~$32\gamma^2\delta_d^2/(\delta^2(1-\gamma)^5)$ iterations by Lemma~\ref{lem:local:opt:CPI}. Similarly, one can show that \cite[Algorithm~2]{wang2022convergence} requires $32\gamma S^3 A \delta_d^2/(\delta^2(1-\gamma)^4)$ iterations to find a~$\delta$-optimal solution for problem~\eqref{expr:DMDP}; see~\cite[Theorem~4.4]{wang2022convergence}. Thus, the iteration complexity of~\cite[Algorithm~2]{wang2022convergence} includes an extra factor~$S^3 A$, which grows polynomially with the numbers of states and actions, but lacks a dimensionless factor~$\gamma/(1-\gamma)$. 

Note that the iteration complexities of both methods scale with the squared distribution mismatch coefficient~$\delta_d^2$. As $s_0$ follows the uniform distribution on~$\mc S$, the discounted state visitation distribution $d^P_\pi(s|s_0)$ must be averaged over~$s_0$. Hence, one can use the trivial bounds $d^P_\pi(s_0|s_0)\geq 1-\gamma$ and $d^P_\pi(s|s_0)\leq 1$ for all $s,s_0\in\mc S$ to show that~$\delta^2_d\leq S^2/(1-\gamma)^2$. The iteration complexities of the FW algorithm and~\cite[Algorithm~2]{wang2022convergence} can thus be expressed as explicit functions of the fundamental parameters~$S$, $A$, $\gamma$ and~$\delta$.
% }
% {\color{blue}This method enjoys a $\mc O(1/\epsilon^4)$ iteration complexity, which is worse than the $\mc O(1/\epsilon^2)$ guarantee of our Frank-Wolfe algorithm (see Theorem~\ref{thm:CPI:PE}) and the $\mc O(\log^2(1/\epsilon))$ guarantee of dynamic programming-based methods (see Theorem~\ref{thm:rectangular:DP}).} 
In addition, the method in~\cite{wang2022convergence} requires an exact projection oracle onto the uncertainty set, while our FW algorithm only requires approximate solutions of the direction-finding subproblem~\eqref{expr:direction:finding}. Our projection-free FW algorithm is thus preferable for non-elementary uncertainty sets. Numerical experiments suggest that the policy gradient method developed in~\cite{wang2022convergence} converges faster than dynamic programming methods despite its suboptimal theoretical convergence rate.
\end{remark}

The proof of Theorem~\ref{thm:CPI:PE} relies on a few preparatory results. First, we need the following variant of the celebrated performance difference lemma for non-robust MDPs~\cite[Lemma 6.1]{kakade2002approximately}, which compares the performance of different transition kernels under a fixed policy~$\pi$.

\begin{lemma}[Performance difference across transition kernels] \label{lem:pdl:Q}
    For any $P,P'\in\mc P$, $\pi\in\Pi,$ and $s_0\in\mc S$, we have 
    \begin{align*}
        \Value{\pi}{P}(s_0)-\!\Value{\pi}{P'}(s_0)
        =\! \frac{1}{1-\gamma}\! \sum_{s\in\mc S,a\in\mc A} \!\!\freqS{\pi}{P}(s|s_0)\pi(a|s)  \sum_{s'\in\mc S} \!(P(s'|s,a)-\!\!P'(s'|s,a)) \Qsp{\pi}{P'}(s,a,s').
    \end{align*}
\end{lemma}
\begin{proof}[Proof of Lemma~\ref{lem:pdl:Q}]
    % We have
    % \begin{align*}
    %     &\Value{\pi}{P}(s_0)-\Value{\pi}{P'}(s_0)
    %     \\&=\sum_{a_0\in\mc A} \pi(a_0|s_0) \Q{\pi}{P}(s_0,a_0)-\sum_{a_0\in\mc A} \pi(a_0|s_0) \Q{\pi}{P'}(s_0,a_0)
    %     \\&=\sum_{s_1\in\mc S,a_0\in\mc A} P(s_1|s_0,a_0)\pi(a_0|s_0) \Qsp{\pi}{P}(s_0,a_0,s_1)-\sum_{s_1\in\mc S,a_0\in\mc A} P'(s_1|s_0,a_0) \pi(a_0|s_0)\Qsp{\pi}{P'}(s_0,a_0,s_1)
    %     \\&=\sum_{a_0\in\mc A} \pi(a_0|s_0) \sum_{s_1\in\mc S}P(s_1|s_0,a_0)\Value{\pi}{P}(s_1)-\sum_{a_0\in\mc A}\pi(a_0|s_0)\sum_{s_1\in\mc S} P'(s_1|s_0,a_0) \Value{\pi}{P'}(s_1)
    %     \\&=\sum_{a_0\in\mc A} \pi(a_0|s_0) \sum_{s_1\in\mc S}P(s_1|s_0,a_0)\Value{\pi}{P}(s_1)-\sum_{a_0\in\mc A}\pi(a_0|s_0)\sum_{s_1\in\mc S} P'(s_1|s_0,a_0) \Value{\pi}{P'}(s_1)
    % \end{align*}
    For any $t\in\mb Z_+$ we have
\begin{align*}
   \nonumber & \Q{\pi}{P}(s_t,a_t)-\Q{\pi}{P'}(s_t,a_t)
    \\& =\gamma\sum_{s_{t+1}\in\mc S} P(s_{t+1}|s_t,a_t) (\Value{\pi}{P}(s_{t+1})-\!\Value{\pi}{P'}(s_{t+1}))
    \\&\qquad\qquad+\gamma\sum_{s_{t+1}\in\mc S} (P(s_{t+1}|s_t,a_t)-\!P'(s_{t+1}|s_t,a_t)) \Value{\pi}{P'}(s_{t+1})
    \\&=\gamma\sum_{s_{t+1}\in\mc S} P(s_{t+1}|s_t,a_t) \sum_{a_{t+1}\in\mc A} \pi(a_{t+1}|s_{t+1}) (\Q{\pi}{P}(s_{t+1},a_{t+1})-\Q{\pi}{P'}(s_{t+1},a_{t+1}))
    \\&\qquad\qquad+\sum_{s_{t+1}\in\mc S} (P(s_{t+1}|s_t,a_t)-P'(s_{t+1}|s_t,a_t)) \Qsp{\pi}{P'}(s_t,a_t,s_{t+1})
    \\&=\gamma\!\!\sum_{\substack{s_{t+1}\in\mc S\\a_{t+1}\in\mc A}} \mb P^P_\pi(S_{t+1}=s_{t+1},A_{t+1}=a_{t+1}|S_t=s_t,A_t=a_t)  (\Q{\pi}{P}(s_{t+1},a_{t+1})-\Q{\pi}{P'}(s_{t+1},a_{t+1}))
    \\&\qquad\qquad+\sum_{s_{t+1}\in\mc S} (P(s_{t+1}|s_t,a_t)-P'(s_{t+1}|s_t,a_t)) \Qsp{\pi}{P'}(s_t,a_t,s_{t+1}), 
\end{align*}
where the first equality follows from Lemma~\ref{lem:recursion:PE}\ref{item:QinG}, the second equality follows from Lemmas~\ref{lem:recursion:PE}\ref{item:valueinQ} and~\ref{lem:recursion:PE}\ref{item:GinQ}, and the last equality holds because of~\eqref{eq:def:MDP}. Substituting the above equation for $t=1$ into the above equation for $t=0$ yields
\begin{equation*}
\begin{split}
& \Q{\pi}{P}(s_0,a_0)-\Q{\pi}{P'}(s_0,a_0)
\\&=\sum_{s_1\in\mc S} (P(s_1|s_0,a_0)-P'(s_1|s_0,a_0)) \Qsp{\pi}{P'}(s_0,a_0,s_1)
\\&\quad +\!\gamma\!\!\sum_{s_1\in\mc S,a_1\in\mc A}\!\! \mb P^P_\pi(S_1=s_1,A_1=a_1|S_0=s_0,\!A_0=a_0)\!\!\!
\\&\qquad\qquad\quad\sum_{s_2\in\mc S}(P(s_2|s_1,a_1)-\!P'(s_2|s_1,a_1\!)) \Qsp{\pi}{P'}\!\!(s_1,a_1,s_2)
\\&\quad+\gamma^2\sum_{s_2\in\mc S,a_2\in\mc A} \mb P^P_\pi(S_2=s_2,A_2=a_2|S_0=s_0,A_0=a_0)  (\Q{\pi}{P}(s_2,a_2)-\Q{\pi}{P'}(s_2,a_2)).
\end{split}
\end{equation*}
By iteratively expanding $\Q{\pi}{P}(s_t,a_t)-\Q{\pi}{P'}(s_t,a_t)$ for all $t\in\mb N$ and recalling that $\gamma\in (0,1)$, we then find
\begin{equation}\label{expr:pdl:nature}
\begin{split}   
& \Q{\pi}{P}(s_0,a_0)-\Q{\pi}{P'}(s_0,a_0)
% \\&=\gamma^2\sum_{s_2\in\mc S,a_2\in\mc A} \mb P^P_\pi(S_2=s_2,A_2=a_2|S_0=s_0,A_0=a_0)  (\Q{\pi}{P}(s_2,a_2)-\Q{\pi}{P'}(s_2,a_2))
% \\& \;\;\; +\!\gamma\!\!\sum_{s_1\in\mc S,a_1\in\mc A}\!\! \mb P^P_\pi(S_1=s_1,A_1=a_1|S_0=s_0,\!A_0=a_0)\!\!\!\sum_{s_2\in\mc S}\!\! (P(s_2|s_1,a_1)-\!P'(s_2|s_1,a_1\!)) \Qsp{\pi}{P'}\!\!(s_1,a_1,s_2)
% \\& \quad +\sum_{s_1\in\mc S} (P(s_1|s_0,a_0)-P'(s_1|s_0,a_0)) \Qsp{\pi}{P'}(s_0,a_0,s_1)
\\&=\sum_{t=0}^\infty\! \gamma^t\!\!\!\!\sum_{s\in\mc S,a\in\mc A} \!\!\!\!\mb P^P_\pi(S_t=s,A_t=a|S_0=s_0,A_0=a_0) 
\\&\qquad\qquad\sum_{s'\in\mc S} (P(s'|s,a)-P'(s'|s,a)) \Qsp{\pi}{P'}(s,a,s')
\\&= \frac{1}{1-\gamma} \sum_{s\in\mc S,a\in\mc A} \freq{\pi}{P}(s,a|s_0,a_0)  \sum_{s'\in\mc S} (P(s'|s,a)-P'(s'|s,a)) \Qsp{\pi}{P'}(s,a,s'),
\end{split}
\end{equation}
where the second equality follows from the construction of $\freq{\pi}{P}(s,a|s_0,a_0)$ in Definition~\ref{def:visitation:distribution}.  By Lemma~\ref{lem:recursion:PE}\ref{item:valueinQ}, we finally obtain
\begin{align*}
    \Value{\pi}{P}(s_0)-\Value{\pi}{P'}(s_0)&=\sum_{a_0\in\mc A}\pi(a_0|s_0)(\Q{\pi}{P}(s_0,a_0)-\Q{\pi}{P'}(s_0,a_0))
    \\&= \frac{1}{1-\gamma} \!\!\sum_{s\in\mc S,a\in\mc A} \!\!\!\freqS{\pi}{P}(s|s_0)\pi(a|s) \!\! \sum_{s'\in\mc S} (P(s'|s,a)-\!P'(s'|s,a)) \Qsp{\pi}{P'}(s,a,s'),
\end{align*}
where the last equality follows from~\eqref{expr:pdl:nature} and the identity $\sum_{a_0\in\mc A}\freq{\pi}{P}(s,a|s_0,a_0)\pi(a_0|s_0)$ $=\freqS{\pi}{P}(s|s_0)\pi(a|s)$. Thus, the claim follows.
\end{proof}

Step~\ref{line:interpolation} of Algorithm~\ref{alg:CPI} readily implies that
\begin{align}\label{expr:transition:difference}
    \sum_{s'\in\mc S}| P^{(m+1)}(s'|s,a)-P^{(m)}(s'|s,a)|\le 2\alpha_m \quad\forall s\in\mc S, a\in\mc A.
\end{align}
Thus, the difference between any two consecutive iterates of Algorithm~\ref{alg:CPI} is bounded by twice the stepsize.
The next lemma, which is inspired by~\cite[Theorem 4.1]{kakade2002approximately}, translates this bound to one for the difference between the discounted state visitation frequencies corresponding to two consecutive iterates.

\begin{lemma}[Similarity between discounted state visitation frequencies]
    \label{lem:PE:distance}
The iterates of Algorithm~\ref{alg:CPI} satisfy
\begin{align*}
    \sum_{s\in\mc S}\left|\freqS{\pi}{P^{(m+1)}}(s|s_0)-\freqS{\pi}{P^{(m)}}(s|s_0)\right| \le \frac{2\alpha_m \gamma}{1-\gamma}\quad \forall m\in\mb Z_+.
\end{align*}
\end{lemma}

\begin{proof}[Proof of Lemma~\ref{lem:PE:distance}]
We use $\rho_t^P\in \Delta(\mc S)$ to denote the probability mass function of~$S_t$ under~$\mb P^P_\pi$ conditional on $S_0=s_0$. 
Its dependence on~$\pi$ and~$s_0$ is suppressed to avoid clutter. Note first that for any $t\in\mb N$ we have
\begin{align*}
\rho_t^{P^{(m+1)}}(s')-\!\rho_t^{P^{(m)}}\!(s')= & \!\!\!\!\sum_{s\in\mc S, a\in\mc A}\!\!\left(\rho_{t-1}^{P^{(m+1)}}\!(s) P^{(m+1)}(s'|s,a)-\!\rho_{t-1}^{P^{(m)}}\!(s) P^{(m)}(s'|s,a)\right)\!\pi(a|s) \\
% = & \sum_{s, a}\left(\rho_{t-1}^{P^{(m+1)}}(s) P^{(m+1)}(s'|s,a)-\rho_{t-1}^{P^{(m+1)}}(s) P^{(m)}(s'|s,a)+\rho_{t-1}^{P^{(m+1)}}(s) P^{(m)}(s'|s,a)-\rho_{t-1}^{P^{(m)}}(s) P^{(m)}(s'|s,a)\right) \pi(a|s) \\
= & \sum_{s\in\mc S} \rho_{t-1}^{P^{(m+1)}}(s) \sum_{a\in\mc A}\left(P^{(m+1)}(s'|s,a)-P^{(m)}(s'|s,a)\right) \pi(a|s) \\
& +\sum_{s\in\mc S}\left(\rho_{t-1}^{P^{(m+1)}}(s)-\rho_{t-1}^{P^{(m)}}(s)\right) \sum_{a\in\mc A}\!\! P^{(m)}(s'|s,a) \pi(a|s) \quad \forall s'\in\mc S.
\end{align*}
Taking absolute values on both sides, using the triangle inequality and summing over $s'$ then yields
\begin{align*}
\left\|\rho_t^{P^{(m+1)}}-\rho_t^{P^{(m)}}\right\|_1 \leq & \sum_{s\in\mc S} \rho_{t-1}^{P^{(m+1)}}(s) \sum_{a\in\mc A}\pi(a|s) \sum_{s'\in\mc S}\left|P^{(m+1)}(s'|s,a)-P^{(m)}(s'|s,a)\right| \\
& +\sum_{s\in\mc S}\left|\rho_{t-1}^{P^{(m+1)}}(s)-\rho_{t-1}^{P^{(m)}}(s)\right| \sum_{s'\in\mc S} \sum_{a\in\mc A} P^{(m)}(s'|s,a) \pi(a|s) \\
\leq & 2 \alpha_m+\left\|\rho_{t-1}^{P^{(m+1)}}-\rho_{t-1}^{P^{(m)}}\right\|_1 ,
% \leq 4 \alpha_m+\left\|\mathbb{P}_{h-2}^{P^{(m+1)}}-\mathbb{P}_{h-2}^{P^{(m)}}\right\|_1=2 h \alpha_m .
\end{align*}
where the second inequality follows from~\eqref{expr:transition:difference}. 
By unfolding this recursive bound for all time points from~$t$ to~$0$ and noting that $\rho_0^{P^{(m+1)}}=\rho_0^{P^{(m)}},$ we then obtain
\begin{align}\label{finite:vistation:freq:bound}
    \left\|\rho_t^{P^{(m+1)}}-\rho_t^{P^{(m)}}\right\|_1 \le 2t\alpha_m.
\end{align}
Next, from the definition of $\freqS{\pi}{P}$ it is clear that
$$
d_{\pi}^{P^{(m+1)}}(s|s_0)-d_{\pi}^{P^{(m)}}(s|s_0)=(1-\gamma) \sum_{t=0}^{\infty} \gamma^t\left(\rho_t^{P^{(m+1)}}(s)-\rho_t^{P^{(m)}}(s)\right)\qquad\forall s\in\mc S.
$$
By~\eqref{finite:vistation:freq:bound}, we therefore find
$$\sum_{s\in\mc S}\left|d_{\pi}^{P^{(m+1)}}(s|s_0)-d_{\pi}^{P^{(m)}}(s|s_0)\right| \leq(1-\gamma) \sum_{t=0}^{\infty} \gamma^t 2 t \alpha_m \le \frac{2 \alpha_m \gamma}{1-\gamma},$$
where the second inequality holds because $\sum_{t=0}^{\infty} \gamma^t t=\gamma/(1-\gamma)^2$.
% Thus, the claim follows.
\end{proof}

The next lemma shows that, under the adaptive stepsize schedule of Theorem~\ref{thm:CPI:PE}, the objective function values of the transition kernels generated by Algorithm~\ref{alg:CPI} are non-decreasing. It is inspired by~\cite[Corollary 4.2]{kakade2002approximately}. 
\begin{lemma}[Adversary's policy improvement]
\label{lem:PE:policy:improve}
Under the stepsize schedule of Theorem~\ref{thm:CPI:PE}, we have $\Value{\pi}{P^{(m+1)}}(s_0)-\Value{\pi}{P^{(m)}}(s_0) \ge \mb G_m^2(1-\gamma)^4/(8 \gamma^2)$ for all $m\in\mb Z_+$.
% \begin{align*}
% \vspace{-5pt}
%     \Value{\pi}{P^{(m+1)}}(s_0)-\Value{\pi}{P^{(m)}}(s_0) \ge \frac{\mb G_m^2(1-\gamma)^4}{8 \gamma^2}\qquad \forall m\in\mb Z_+.
%      \vspace{-5pt}
% \end{align*}
\end{lemma}

\begin{proof}[Proof of Lemma~\ref{lem:PE:policy:improve}]
Throughout the proof we use $P_\epsilon$ to denote the $\epsilon$-optimal solution of problem~\eqref{expr:direction:finding} that is computed in the $m$-th iteration of Algorithm~\ref{alg:CPI}. By Lemma~\ref{lem:pdl:Q}, we then have
\begin{equation}\label{eq:pdl:advantage}
 \vspace{-5pt}
\begin{split}
&(1-\gamma) (\Value{\pi}{P^{(m+1)}}(s_0)-\Value{\pi}{P^{(m)}}(s_0))
\\&=\!\!\!\!\sum_{s\in\mc S,a\in\mc A} \freqS{\pi}{P^{(m+1)}}(s|s_0)\pi(a|s)  \sum_{s'\in\mc S} (P^{(m+1)}(s'|s,a)-P^{(m)}(s'|s,a)) \Qsp{\pi}{P^{(m)}}(s,a,s')
\\&=\!\!\!\!\sum_{s\in\mc S,a\in\mc A} \freqS{\pi}{P^{(m+1)}}(s|s_0)\pi(a|s)  \sum_{s'\in\mc S}\alpha_m (P_\epsilon(s'|s,a)-P^{(m)}(s'|s,a)) \Qsp{\pi}{P^{(m)}}(s,a,s')
\\&=\!\!\!\!\sum_{s\in\mc S,a\in\mc A} \freqS{\pi}{P^{(m+1)}}(s|s_0)\pi(a|s)  \sum_{s'\in\mc S}\alpha_m P_\epsilon(s'|s,a)\A{\pi}{P^{(m)}}(s,a,s'),
 \end{split}
 \vspace{-5pt}
\end{equation}
where the second equality follows from the construction of~$P^{(m+1)}$ in Algorithm~\ref{alg:CPI}, and the last equality follows from Lemma~\ref{lemma:advantage:Qsp}.
Adding and subtracting $\alpha_m\mb G_m(1-\gamma)$ on the right hand side of~\eqref{eq:pdl:advantage} and using 
a similar reasoning as in the proof of Proposition~\ref{prop:deg:nonrect} to express the approximate Frank-Wolfe gap~$\mb G_m$ in terms of the advantage function~$A_\pi^{P^{(m)}}$, we obtain
\begin{align*}
    &(1-\gamma) (\Value{\pi}{P^{(m+1)}}(s_0)-\Value{\pi}{P^{(m)}}(s_0))
    \\&=\alpha_m \mb G_m(1-\gamma) +\alpha_m \sum_{s\in\mc S,a\in\mc A} \freqS{\pi}{P^{(m+1)}}(s|s_0)\pi(a|s)  \sum_{s'\in\mc S} P_\epsilon(s'|s,a)\A{\pi}{P^{(m)}}(s,a,s')
\\&\quad -\alpha_m \sum_{s\in\mc S,a\in\mc A} \freqS{\pi}{P^{(m)}}(s|s_0)\pi(a|s)  \sum_{s'\in\mc S} P_\epsilon(s'|s,a)\A{\pi}{P^{(m)}}(s,a,s')
\\&\ge \alpha_m \mb G_m(1-\gamma) - \alpha_m \max_{s,s'\in\mc S, a\in\mc A, P\in\mc P} | \A{\pi}{P}(s,a,s')|\sum_{s\in\mc S}\left|\freqS{\pi}{P^{(m+1)}}(s|s_0)-\freqS{\pi}{P^{(m)}}(s|s_0)\right|
\\&\ge \alpha_m \mb G_m(1-\gamma) -  \frac{\alpha_m \gamma}{1-\gamma}\sum_{s\in\mc S}\left|\freqS{\pi}{P^{(m+1)}}(s|s_0)-\freqS{\pi}{P^{(m)}}(s|s_0)\right|
\\&\ge \alpha_m\left(\mb G_m(1-\gamma)-\frac{2 \alpha_m \gamma^2}{(1-\gamma)^2}\right).
\end{align*}
The first inequality in the above expression follows from H\"older's inequality. Recall next that $c(s,a)\in [0,1]$ for all $(s,a)\in\mc S\times\mc A$, which implies that $0\le \Value{\pi}{P}(s_0)\le \sum_{t=0}^\infty \gamma^t= 1/(1-\gamma)$ for all $P\in\mc P$. By the definition of the advantage function and by Lemmas~\ref{lem:recursion:PE}\ref{item:QinG} and~\ref{lem:recursion:PE}\ref{item:GinQ}, we then have $|\A{\pi}{P}(s,a,s')|\le \gamma/(1-\gamma)$ for all $s,s'\in\mc S, a\in\mc A,$ and $ P\in\mc P$. This justifies the second inequality. The last inequality follows from Lemma~\ref{lem:PE:distance}. The stepsize $\alpha_m = \mb G_m (1-\gamma)^3/(4\gamma^2)$ was chosen to maximize the last expression. Replacing~$\alpha_m$ by this formula yields the desired bound.
\end{proof}

We can now show that the proposed FW algorithm terminates within $\mc O(1/\epsilon^2)$ iterations with a Frank-Wolfe gap of at most $2\epsilon$.
\begin{lemma}[Finite termination]
\label{lem:local:opt:CPI}
Under the stepsize schedule of Theorem~\ref{thm:CPI:PE}, Algorithm~\ref{alg:CPI} terminates within $8 \gamma^2/(\epsilon^2(1-\gamma)^5) $  iterations, and its output~$\widehat P$ satisfies $\max_{P\in\mc P} \langle \nabla_P \Value{\pi}{\widehat P}(s_0),P-\widehat P\rangle \le 2\epsilon.$
    % \begin{align*}
    % \vspace{-5pt}
    %     \max_{P\in\mc P} \langle \nabla_P \Value{\pi}{\widehat P}(s_0),P-\widehat P\rangle \le 2\epsilon.
    % \end{align*}
\end{lemma}

Theorem 5 in~\cite{li2021distributionally} also shows that Algorithm~\ref{alg:CPI} converges to an approximate stationary point but does not provide an explicit expression for the iteration complexity. While~\cite{li2021distributionally} focuses on a specific non-rectangular uncertainty set constructed from the conditional relative entropy and uses exact line search to determine the stepsize sequence, which is computationally expensive, Lemma~\ref{lem:local:opt:CPI} applies to general non-rectangular uncertainty sets and leverages an easily computable stepsize schedule. In addition,~\cite{li2021distributionally} assumes to have access to an exact optimizer of the direction-finding subproblem, while Lemma~\ref{lem:local:opt:CPI} only requires access to an $\epsilon$-optimal solution.
% \begin{theorem}[Local performance guarantee]
% If $M=\mc O(1/\epsilon^2)$ for some $\epsilon>0$ and if the stepsizes are found by linesearch, that is, if $$\alpha_m\in\argmax_{\alpha\in[0,1]} \{\Value{\pi}{P}(s_0)\ : \ P=P^{(m)}+\alpha (P_\epsilon-P^{(m)})\}$$ for every $m=0,\ldots,M-1,$ then Algorithm~\ref{alg:CPI} outputs an $\epsilon$-stationary point of problem~\eqref{expr:DMDP} in the sense that the Frank-Wolfe gap of the last iterate is at most~$\epsilon$.
% \end{theorem}
% \begin{proof}
% The claim follows from a straightforward adaptation of.
% \end{proof}

\begin{proof}[Proof of Lemma~\ref{lem:local:opt:CPI}]
Note that if Algorithm~\ref{alg:CPI} does \textit{not} terminate in iteration $m$, then $\mb G_m >\epsilon$, and hence $\Value{\pi}{P^{(m+1)}}(s_0)\ge \Value{\pi}{P^{(m)}}(s_0) +\epsilon^2(1-\gamma)^4/(8 \gamma^2)$ by Lemma~\ref{lem:PE:policy:improve}.
% \begin{align*}
%     \Value{\pi}{P^{(m+1)}}(s_0)\ge \Value{\pi}{P^{(m)}}(s_0) +\frac{\epsilon^2(1-\gamma)^4}{8 \gamma^2}.
% \end{align*}
As $c(s,a)\in[0,1]$, we have $0\le\Value{\pi}{P}(s_0)\le \sum_{t=0}^\infty \gamma^t= 1/(1-\gamma)$ for every $P\in\mc P$. The above per-iteration improvement can thus only persist for at most $8\gamma^2/(\epsilon^2(1-\gamma)^5)$ iterations. If Algorithm~\ref{alg:CPI} terminates in iteration~$m$, however, then $\mb G_m\le\epsilon,$ and thus we have
\begin{align*}
&\max_{P\in\mc P} \langle \nabla_P \Value{\pi}{P^{(m)}}(s_0),P-P^{(m)}\rangle\le \langle \nabla_P \Value{\pi}{P^{(m)}}(s_0),P_\epsilon-P^{(m)}\rangle + \epsilon
= \mb G_m+\epsilon\le 2\epsilon.
\end{align*}
Hence, the claim follows.
\end{proof}

We are now ready to establish the convergence behavior of Algorithm~\ref{alg:CPI}.

% {\color{blue}
\begin{proof}[Proof of Theorem~\ref{thm:CPI:PE}]
Let $P_{\mathsf{s}}^\star$ be any maximizer of the robust policy evaluation problem~\eqref{expr:DMDP} when $\mc P$ is replaced with the smallest $s$-rectangular uncertainty set $\mc P_{\mathsf s}$ that contains $\mc P$. 
% In addition, introduce the distribution mismatch coefficient $C=\max_{P,P'\in\mc P_{\mathsf{s}},s\in\mc S}|\freqS{\pi}{P'}(s|s_0)/\freqS{\pi}{P}(s|s_0)|$, and note that $C$ is well-defined thanks to Assumption~\ref{assume:irreducible}.
As $\mc P\subseteq \mc P_{\mathsf{s}},$ we have
\begin{align*}
\ValRob{\pi}(s_0)-\!\Value{\pi}{\widehat P}(s_0)
\!=\!\max_{P\in\mc P}\! \Value{\pi}{P}(s_0)-\!\Value{\pi}{\widehat P}(s_0)
\le \max_{P_{\mathsf{s}}\in\mc P_{\mathsf s}}\Value{\pi}{P_{\mathsf{s}}}(s_0)-\!\Value{\pi}{\widehat P}(s_0)
=\Value{\pi}{P_{\mathsf{s}}^\star}(s_0)-\!\Value{\pi}{\widehat P}(s_0).
\end{align*}
This in turn implies that
\begin{align*}
\ValRob{\pi}\!(\!s_0\!)\!-\!\Value{\pi}{\widehat P}\!(\!s_0\!)\!
&\le \Value{\pi}{P_{\mathsf{s}}^\star}(s_0)-\Value{\pi}{\widehat P}(s_0)
\\&=\frac{1}{1-\gamma}\sum_{s\in\mc S,a\in\mc A} \freqS{\pi}{P_{\mathsf{s}}^\star}(s|s_0)\pi(a|s)  \sum_{s'\in\mc S} (P_{\mathsf{s}}^\star(s'|s,a) -\widehat P(s'|s,a))\Qsp{\pi}{\widehat P}(s,a,s')
\\&= \frac{1}{1-\gamma}\sum_{s\in\mc S,a\in\mc A} \freqS{\pi}{P_{\mathsf{s}}^\star}(s|s_0)\pi(a|s)  \sum_{s'\in\mc S} P_{\mathsf{s}}^\star(s'|s,a) \A{\pi}{\widehat P}(s,a,s')
\\&\le \frac{1}{1-\gamma}\sum_{s\in\mc S} \freqS{\pi}{P_{\mathsf{s}}^\star}(s|s_0)\max_{P_{\mathsf{s}}\in\mc P_{\mathsf{s}}}\sum_{a\in\mc A}\pi(a|s)  \sum_{s'\in\mc S} P_{\mathsf{s}}(s'|s,a) \A{\pi}{\widehat P}(s,a,s')
\\& \leq \frac{1}{1-\gamma} \delta_d(P^\star_{\mathsf s}, \widehat P)\sum_{s\in\mc S} \freqS{\pi}{\widehat P}(s|s_0)\max_{P_{\mathsf{s}}\in\mc P_{\mathsf{s}}}\sum_{a\in\mc A}\pi(a|s)  \sum_{s'\in\mc S} P_{\mathsf{s}}(s'|s,a) \A{\pi}{\widehat P}(s,a,s')
\\&=\frac{1}{1-\gamma} \delta_d(P^\star_{\mathsf s}, \widehat P)\max_{P_{\mathsf{s}}\in\mc P_{\mathsf{s}}}\sum_{s\in\mc S} \freqS{\pi}{\widehat P}(s|s_0)\sum_{a\in\mc A}\pi(a|s)  \sum_{s'\in\mc S} P_{\mathsf{s}}(s'|s,a) \A{\pi}{\widehat P}(s,a,s')
\\&= \delta_d(P^\star_{\mathsf s}, \widehat P) \max_{P_{\mathsf{s}}\in\mc P_{\mathsf{s}}}\langle \nabla_P \Value{\pi}{\widehat P}(s_0), P_{\mathsf{s}}-\widehat P\rangle,
\end{align*}
where the first two equalities exploit Lemma~\ref{lem:pdl:Q} and Lemma~\ref{lemma:advantage:Qsp}, respectively.
The third inequality follows from the definition of the distribution mismatch coefficient $\delta_d(\!P^\star_{\mathsf s}, \widehat P)$ and from H\"older's inequality, which applies because $\sum_{s'\in\mc S}\!\widehat P(s'|s,a)\A{\pi}{\widehat P}(\!s,\!a,\!s')$ $=0$ and hence $\max_{P_{\mathsf{s}}\in\mc P_{\mathsf{s}}}\sum_{a\in\mc A}\pi(a|s) \sum_{s'\in\mc S}\allowbreak P(s'|s,a) \A{\pi}{\widehat P}(s,a,s')\ge0$ for all $ s\in\mc S$. The third equality exploits the $s$-rectangularity of $\mc P_{\mathsf{s}}$, and the last equality follows from a variant of Proposition~\ref{prop:deg:nonrect} where~$\mc P$ is replaced by~$\mc P_{\mathsf{s}}$. 
Hence, we find
\begin{align*}
\ValRob{\pi}(s_0)-\Value{\pi}{\widehat P}(s_0)
& \le \delta_d(P^\star_{\mathsf s}, \widehat P) \left(\max_{P_{\mathsf{s}}\in\mc P_{\mathsf{s}}}\langle \nabla_P \Value{\pi}{\widehat P}(s_0), P_{\mathsf{s}}-\widehat P\rangle -\max_{P\in\mc P} \langle \nabla_P \Value{\pi}{\widehat P}(s_0), P-\widehat P\rangle
\right.
\\&\qquad\qquad\qquad\left.+\max_{P\in\mc P} \langle \nabla_P \Value{\pi}{\widehat P}(s_0), P-\widehat P\rangle\right)\\
& \leq \delta_d(P^\star_{\mathsf s}, \widehat P)(\delta_{\mc P}(\widehat P)+2\epsilon) \leq \delta_d(\delta_{\mc P}+2\epsilon),
\end{align*}
where the second inequality holds thanks to the definition of $\delta_{\mc P}(\widehat P)$ and Lemma~\ref{lem:local:opt:CPI}. The claim finally follows because $\delta_d(P^\star_{\mathsf s}, \widehat P)$ and $\delta_{\mc P}(\widehat P)$ are trivially bounded above by~$\delta_d$ and~$\delta_{\mc P},$ respectively, which are independent of the output~$\widehat P$ of Algorithm~\ref{alg:CPI}.
\end{proof}

\vspace{-5pt}
\section{Robust Policy Improvement} \label{sec:PL}
% \input{text/2-PL.tex}

% \ML{TO THINK: Does dynamic programming principle hold for the policy improvement problem assuming that nature precommits?}

We now develop an actor-critic algorithm to solve the robust policy improvement problem~\eqref{def:policy:learning:objective} for a fixed initial state~$s_0$ to global optimality; see Algorithm~\ref{alg:PG-min-oracle}.
% We are now ready to outline the proposed actor-critic method for solving the policy improvement problem~\eqref{def:policy:learning:objective} 
In each iteration $k\in\mb Z_+$, Algorithm~\ref{alg:PG-min-oracle} first computes an $\epsilon$-optimal solution $\Pk$ of the robust policy evaluation problem~\eqref{expr:DMDP} associated with the current policy~$\pik$ (\textit{critic}) and
% For this inner loop critic subroutine, we assume an access to an approximate robust policy evaluation oracle. One example of such an oracle is Algorithm~\ref{alg:PLD} from Section~\ref{ssect:PLD}. 
then applies a projected gradient step to find a new policy $\pikp$ that locally improves the value function associated with the current transition kernel $\Pk$ (\textit{actor}). The critic's subproblem could be addressed with Algorithm~\ref{alg:PLD}, for example, which outputs an $\epsilon$-optimal solution of the robust policy evaluation problem with high probability. The actor's subproblem consists in projecting a vector onto the probability simplex~$\Pi,$ which can be done efficiently~\cite{wang2013projection}.

\begin{algorithm}[h!] 
  \caption{Actor-critic algorithm for solving the robust policy improvement problem~\eqref{def:policy:learning:objective}}
  \label{alg:PG-min-oracle}
\begin{algorithmic}[1]
\REQUIRE Iteration number $K\in\mb N$, stepsize $\eta>0,$ tolerance $\epsilon>0$
\STATE Initialize  $\pi^{(0)}(a|s)=1/A\ \forall s\in\mc S,a\in\mc A,$ and set $ k\leftarrow 0$
\WHILE{$k\le K-1$}
% \STATE Set $\tauk=\min_{a\in\mc A} \pik(a|s)$
  \STATE \textit{Critic}: Find $\Pk\in\mc P$ such that
  $\Value{\pik}{\Pk}(s_0)\ge \Value{\pik}{\star}(s_0)-\epsilon$
  \label{alg:PG-min-oracle:line:PE}
  \STATE \textit{Actor}: $\pikp=\mathrm{Proj}_{\Pi} \left(\pik+\eta\nabla_{\pi} \Value{\pik}{\Pk}(s_0)\right)$
  \STATE $k \leftarrow k+1$
\ENDWHILE
\end{algorithmic}
\end{algorithm}
The following assumption is essential for the main results of this section.
\begin{assumption}[Irreducibility]\label{assume:irreducible:pi}
The Markov chain $\{S_t\}_{t=0}^\infty$ is irreducible for any  $P\in\mc P$  and $\pi\in\Pi$.
\end{assumption}
% {\color{blue} 
The sole purpose of Assumption~\ref{assume:irreducible:pi} is to ensure that the distribution mismatch coefficient $C=\max_{\pi,\pi'\in\Pi,s\in\mc S,P\in\mc P}\freqS{\pi}{P}(s|s_0)/\freqS{\pi'}{P}(s|s_0)$ is finite and strictly positive. This can be shown by using a similar reasoning as in Remark~\ref{remark:distribution:mismatch}.
Instead of requiring irreducibility, one could also simply require that all components of the initial state distribution $\rho\in\Delta(\mathcal S)$ are strictly positive. However, this would imply that the initial state is random, which contradicts the standard assumption that states are observed.
% }
% , 
% where $\bar\pi(P)\in\Pi$ is a minimizer of $\min_{\pi\in\Pi}\Value{\pi}{P}(s_0) $ for $P\in\mc P$.

Recall now from Remark~\ref{rem:V-formula} and the surrounding discussion that~$\Value{\pi}{P}$ constitutes a rational function of~$\pi$ that is defined on a neighborhood of~$\Pi$. The following lemma establishes several desirable properties this value function. In the remainder of this section, we frequently use the constants $L=\sqrt{A}/(1-\gamma)^2$ and $\ell=2\gamma A/(1-\gamma)^3$.
\begin{lemma}[Properties of the value function]\label{lemma:extend:V}
Suppose that Assumption~\ref{assume:irreducible:pi} holds. Then, for every $\delta>0,$ there exists an open neighborhood $\Pi_\delta$ of $\Pi$ such that any point in~$\Pi_\delta$ has a (Frobenius) distance of at most~$\delta$ from some point in~$\Pi$, and the value function $\Value{\pi}{P}(s_0)$ satisfies the following conditions for every $P\in\mc P$.
\begin{enumerate}[label = (\roman*)]
\item \label{item:smooth:extend:V} $\Value{\pi}{P}(s_0)$ is $(L+\delta)$-Lipschitz continuous and $(\ell+\delta)$-smooth in $\pi$ on $\Pi_\delta$.
\item \label{eq:gd:fix:P}
% \begin{equation*}\label{gd:fixed:P}
$C^{-1}( V_{\pi}^{P}(s_0)-\min_{\pi'\in\Pi}V_{\pi'}^{P}(s_0)) -\delta\leq \max _{\pi'\in\Pi_\delta}\langle \pi-\pi',\nabla_\pi V_{\pi}^{P}(s_0)\rangle$ for all $\pi\in\Pi_\delta.$
  % $(C^{-1}-\delta)\left( V_{\pi}^{P}(s_0)-\min_{\pi'\in\Pi}V_{\pi'}^{P}(s_0)\right) -\delta\leq \max _{\pi'\in\Pi_\delta}\langle \pi-\pi',\nabla_\pi V_{\pi}^{P}(s_0)\rangle$ for every $\pi\in\Pi_\delta.$
% \end{equation}
\end{enumerate}
\end{lemma}
\begin{proof}
By~\cite[Lemma~3.1]{wang2022convergence}, $\Value{\pi}{P}(s_0)$ is $L$-Lipschitz continuous and $\ell$-smooth on~$\Pi$. In addition, $C^{-1}(\Value{\pi}{P}(s_0)-\min_{\pi'\in\Pi}\Value{\pi'}{P}(s_0)) \leq \max _{\pi'\in\Pi}\langle \pi-\pi',\nabla_\pi \Value{\pi}{P}(s_0)\rangle $ for all $\pi\in\Pi$ thanks to~\cite[Lemma 4.1]{agarwal2021theory}.  As $\Value{\pi}{P}(s_0)$ is continuous and rational in~$\pi$ and~$P$ on a neighborhood of $\Pi\times\mc P$ and as~$\Pi$ is compact, Berge's maximum theorem~\cite[pp.~115-116]{berge1997topological} implies that $\max _{\pi'\in\Pi}\langle \pi-\pi',\nabla_\pi \Value{\pi}{P}(s_0)\rangle$ is continuous in~$\pi$ on a neighborhood of~$\Pi$.
% The claim thus follows from \cite[Remark 1]{daskalakis2020independent}.
% We can in addition extend $V_{\pi}^{P}(s_0)$ in the way characterized by Lemma to a $\delta$-neighborhood~$\Pi_\delta$ of the convex compact set~$\Pi$ for any $\delta>0$ such that $\Value{\pi}{P}(s_0)$ is $(L+\delta)$-Lipschitz continuous and $(\ell+\delta)$-smooth in $\pi$ on $\Pi_\delta$ for every $P\in\mc P$ thanks to% Such an extension is guaranteed to exist for any $\delta>0$ according to
The claim then follows because both~$\Pi$ and~$\mc P$ are compact.
% Then,~\cite[Remark 1]{daskalakis2020independent} indicates the existence of an extension $V_{\pi}^{P}(s_0)$ to~$\Pi_\delta$ such that \ref{item:smooth:extend:V} and \ref{eq:gd:fix:P} are satisfied. Thus, the claim follows.
\end{proof}
Throughout the rest of this section we use $\Phi(\pi)$ as a shorthand for the worst-case value function $V_\pi^\star(s_0)=\max_{P\in\mc P}V_{\pi}^{P}(s_0)$, which is defined for all $\pi\in\Pi_\delta$. This helps us to avoid clutter.
% define the primal function $\Phi(\pi)\!=\!\Value{\pi}{\star}\!(s_0)\!=\!\max_{\!P\in\mc P}\!\Value{\pi}{P}\!(s_0)\!$ 
% as the objective function of the robust policy improvement problem~\eqref{def:policy:learning:objective} on the extended domain $\Pi_\delta$. 
We henceforth refer to~$\Phi$ as the primal function.
In addition, we let $\pistar\in\argmin_{\pi\in\Pi}\Phi(\pi)$ be an optimal solution of the policy improvement problem~\eqref{def:policy:learning:objective}.  
% We also identify the policy space~$\Pi=\Delta(\mc A)^S$ as a subset of~$\mb R^{SA}.$
The primal function~$\Phi$ generically fails to be differentiable. It is thus useful to approximate~$\Phi$ by its Moreau envelope $\Phi_\lambda: \mb R^{S\times A}\to\mb R$ parametrized by~$\lambda > 0$, which is defined through $\Phi_\lambda(\pi)=\min_{\pi'\in\Pi} \Phi(\pi')+\|\pi'-\pi\|^2_{\mathbf F}/(2\lambda).$ 
The following lemma establishes useful properties of the primal function~$\Phi$ and its Moreau envelope $\Phi_\lambda$. 

\begin{lemma}[Properties of the primal function]\label{lemma:primal:moreau}
% There exists $\bar\delta >0$ such that  $V_{\pi}^{P}(s_0)$ can be continuously extended to the $\delta$-neighborhood~$\Pi_\delta$ of the convex compact set~$\Pi$, and set $\Phi(\pi)=\max_{P\in\mc P}V_{\pi}^{P}(s_0)$ for all $\pi\in\Pi_\delta.$  
The following hold.
\begin{enumerate}[label = (\roman*)]
    % \item \label{primal:item:value:smooth} 
    \item \label{primal:item:primal:weak:convex} 
    % % The primal function 
    $\Phi(\pi)$ is $(\ell+\delta)$-weakly convex and $(L+\delta)$-Lipschitz continuous on~$\Pi_\delta$.
    % \item \label{primal:item:moreau:gradient}
    % For any $\lambda>0$, we have $\nabla \Phi_\lambda(\pi)=\lambda^{-1}(\pi-\argmin_{\pi'\in\Pi}(\Phi(\pi')+\|\pi-\pi'\|^2/(2\lambda))).$
    \item \label{primal:item:small:subgradient}
    If $0<\lambda< 1/(\ell+\delta)$, then $\Phi_{\lambda}(\pi)$ is convex and differentiable. If additionally $\|\nabla\Phi_{\lambda}(\pi)\|_{\mathbf F}\le \epsilon$ for some $\pi\in\Pi_\delta$, then there exists $\hat \pi\in\Pi_\delta$ such that $\|\pi-\hat\pi\|_{\mathbf F}\le \epsilon\lambda$ and $\min_{v\in\partial \Phi(\hat\pi)}\|v\|_{\mathbf F}\le \epsilon.$
\end{enumerate}
\end{lemma}
% \ML{how does everything work with a biased stochastic oracle?}
\begin{proof}
% We first show Assertion~\ref{primal:item:value:smooth}. If $\delta=0$,~\cite[Lemma 3.1]{wang2022convergence} indicates that $\Value{\pi}{P}(s_0)$ is $L$-Lipschitz continuous and $\ell$-smooth in $\pi$ on $\Pi$. When $\delta>0,$ the previous observation in combination with \cite[Remark 1]{daskalakis2020independent} asserts that $\Value{\pi}{P}(s_0)$ is $(L+\delta)$-Lipschitz continuous and $(\ell+\delta)$-smooth in $\pi$ on $\Pi_\delta$. This completes the proof for Assertion~\ref{primal:item:value:smooth}. 
As for Assertion~\ref{primal:item:primal:weak:convex}, note first that $\Phi(\pi)$ is $(L+\delta)$-Lipschitz continuous on~$\Pi_\delta$ thanks to~\cite[Lemma 4.3]{lin2022nonasymptotic}, which applies because of the $(L+\delta)$-Lipschitz continuity of~$\Value{\pi}{P}(s_0)$ established in Lemma~\ref{lemma:extend:V}\ref{item:smooth:extend:V}.
% . To see this, take $P^\star(\pi)\in\argmax_{P\in\mc P}\Value{\pi}{P}(s_0)$  and $P^\star(\pi')\in\argmax_{P\in\mc P}\Value{\pi'}{P}(s_0)$. Then, $\Phi(\pi)-\Phi(\pi')=\max_{P\in\mc P}\Value{\pi}{P}(s_0)-\max_{P\in\mc P}\Value{\pi'}{P}(s_0)= \Value{\pi}{P^\star(\pi)}(s_0)-\Value{\pi'}{P^\star(\pi')}(s_0)\le $ 
Similarly, \cite[Lemma~3.3]{thekumparampil2019efficient} implies that~$\Phi(\pi)$ inherits $(\ell+\delta)$-weak convexity from the $(\ell+\delta)$-smoothness of $\Value{\pi}{P}(s_0)$ established in Lemma~\ref{lemma:extend:V}\ref{item:smooth:extend:V}. Assertion~\ref{primal:item:small:subgradient} then holds because of \cite[Section~2.2]{davis2019stochastic}. %due to the $(\ell+\delta)$-weak convexity of~$\Phi(\pi)$. 
We include a short proof to keep this paper self-contained. For ease of notation we set $f_{\pi}(\pi')=\Phi(\pi')+\|\pi'-\pi\|^2_{\mathbf F}/(2\lambda)$. Note first that $f_{\pi}(\pi')$ is strongly convex in~$\pi'$ because $\Phi(\pi')$ is $(\ell+\delta)$-weakly convex and because $0<\lambda<1/(\ell+\delta)$. Danskin's theorem~\cite[Proposition~B.25]{bertsekas2016nonlinear} thus implies that, for any $\pi\in\Pi_\delta$,  $\Phi_\lambda(\pi)=\min_{\pi'\in\Pi} f_\pi(\pi')$ is convex and differentiable with $\nabla \Phi_\lambda(\pi)= \nabla_\pi f_{\pi}(\hat\pi)=(\pi-\hat{\pi})/\lambda$, where $\hat{\pi}$ is the unique minimizer of~$f_{\pi}(\pi')$ across all~$\pi'\in\Pi$. This implies that if  $\|\nabla\Phi_{\lambda}(\pi)\|_{\mathbf F}\le \epsilon$, then $\|\hat{\pi}-\pi\|_{\mathbf F} \le\epsilon\lambda$.
It suffices to show that $\hat{\pi}=\argmin_{\pi'\in\Pi} f(\pi')$ satisfies $\min _{\xi \in \partial \Phi(\hat{\pi})}\|\xi\|_{\mathbf F} \leq \epsilon$ and $\|\pi-\hat{\pi}\|_{\mathbf F} \leq \epsilon\lambda$. Therefore, the optimal solution of $\min_{\pi'\in\Pi} f(\pi')$ is unique, and Danskin's theorem~\cite[Proposition~B.25]{bertsekas2016nonlinear} implies that $\nabla \Phi_\lambda(\pi)=$ $\nabla_\pi f_\pi(\hat\pi)=(\pi-\hat{\pi})/\lambda$. Taking the norm of the previous equality yields $\|\hat{\pi}-\pi\|_{\mathbf F}=\|\nabla \Phi_\lambda(\pi)\|_{\mathbf F} \lambda\le\epsilon\lambda$, where the inequality follows from the premise that $\|\nabla \Phi_\lambda(\pi)\|_{\mathbf F}\le \epsilon$. 
On the other hand, the optimality of $\hat{\pi}$ implies that $0\in \partial_{\pi'} f_\pi(\pi')|_{\pi'=\hat\pi}$, which is equivalent to $(\pi-\hat\pi)/\lambda\in\partial \Phi(\hat\pi)$. Hence, it follows that $\min _{v \in \partial \Phi(\hat{\pi})}\|v\|_{\mathbf F} \leq \|(\pi-\hat\pi)/\lambda\|_{\mathbf F}\le \epsilon$.
\end{proof}

Lemma~\ref{lemma:primal:moreau}\ref{primal:item:small:subgradient} asserts that
% connects the approximate stationarity of Moreau envelope $\Phi_\lambda$ of the $\ell$-weakly convex function $\Phi$ to its own approximate stationarity. Namely, 
if $0<\lambda<1/\ell$, then
the $\lambda\epsilon$-neighborhood of any approximate stationary point of the Moreau envelope $\Phi_\lambda$ contains an approximate stationary point of~$\Phi$. Thus, approximate stationary points of $\Phi$ can be found by searching for approximate stationary points of $\Phi_\lambda.$ 

% 2\sqrt{2S/K}/L
\begin{lemma}[Stationarity guarantee{~\cite[Theorem~3.3]{wang2022convergence}}]\label{lemma:gdmax:converge} 
% If $\epsilon=L\sqrt{2S/K}/2$ and $\eta=\sqrt{2S/K}/L$, then 
The iterates $\!\{\pik\}_{k=0}^{K-1}\!$ of Algorithm~\ref{alg:PG-min-oracle} satisfy \[\sum_{k=0}^{K-1}\!\left\|\nabla \Phi_{1/(2\ell)}(\pik)\right\|_{\mathbf F}^2\leq \sqrt{\frac{4 \ell S}{\eta}+2 K \eta \ell L^2+4 \ell \epsilon K}.
% 6\ell L \sqrt{2SK}.
\]
\end{lemma}
Lemma~\ref{lemma:gdmax:converge} guarantees that 
% {\color{blue}
if $\epsilon=L\sqrt{2S/K}/2$ and $\eta=\sqrt{2S/K}/L$, then
% } 
the iterates $\{\pik\}_{k=0}^{K}$ generated by Algorithm~\ref{alg:PG-min-oracle} satisfy 
% \vspace{-5pt}
\[\!\min_{k=0,\ldots,K-1}\!\left\|\nabla \Phi_{1/(2\ell)}(\pik)\!\right\|_{\mathbf F}\!\!=\!\left(\min_{k=0,\ldots,K-1}\!\left\|\nabla \Phi_{1/(2\ell)}(\pik)\!\right\|^2_{\mathbf F}\right)^{\frac{1}{2}}\!\!\le\! \big(6\ell L \sqrt{2SK}\big)^{\frac{1}{2}}\!\!=\!\mc O\big(\!K^{-\frac{1}{4}}\!\big).\]

The following lemma, which is inspired by~\cite[Lemma 12]{daskalakis2020independent},
 establishes a fundamental inequality that can be used to convert an approximate stationary point of the Moreau envelope $\Phi_{1/(2\ell)}$ to an approximate minimizer of $\Phi.$
\begin{lemma}[Gradient dominance property of $\Phi$]\label{lemma:gd:primal}
If Assumption~\ref{assume:irreducible:pi} holds, then we have  $\Phi(\pi)-\Phi\left(\pi^\star\right) \leq (C\sqrt{2S}+L/(2\ell)) \|\nabla \Phi_{1 / (2 \ell)}(\pi)\|_{\mathbf F}$ for all $\pi\in\Pi$.
\end{lemma}
\begin{proof}[Proof of Lemma~\ref{lemma:gd:primal}]
% By Lemma~\ref{gd:fixed:P} and~\cite[Remark 1]{daskalakis2020independent}, there exists $\delta>0$ 
% there exists a closed convex set $\Pi_\delta$ such that $\Pi\subseteq \mathrm{int}\Pi_\delta$, so that any point in $\Pi_\delta$ is at most distance $\delta$ from a point in $\Pi,$ 
% such that $V_{\pi}^{P}(s_0)$ can be continuously extended to the $\delta$-neighborhood~$\Pi_\delta$ of the convex compact set~$\Pi,$ ${V}_\pi^P(s_0)$ is $(\ell+\delta)$-smooth and $(L+\delta)$-Lipschitz on~$\Pi_\delta$, and 
% \begin{equation}\label{eq:gd:fix:P}
%   (C^{-1}-\delta)\left(\min_{\pi'\in\Pi}V_{\pi'}^{P}(s_0)- V_{\pi}^{P}(s_0)\right) -\delta\leq \max _{\pi'\in\Pi_\delta}\langle \pi-\pi',\nabla_\pi V_{\pi}^{P}(s_0)\rangle   
% \end{equation}
% holds for any $\pi\in\Pi_\delta.$
% We can also extend the primal function $\Phi$ to~$\Pi_\delta$ by setting $\Phi(\pi)=\max_{P\in\mc P}V_{\pi}^{P}(s_0)$ for all $\pi\in\Pi_\delta.$ 
Choose any $\delta>0$. By Lemma~\ref{lemma:primal:moreau}\ref{primal:item:primal:weak:convex},~$\Phi$ is $(\ell+\delta)$-weakly convex on~$\Pi_\delta$. Theorem~25.5 in~\cite{rockafellar1970convex} then implies that the set of points at which $\Phi$ is differentiable is dense in $\mathrm{int}\Pi_\delta$ and hence in $\Pi_\delta$. We
% Then, for each $t,$ we have from the previous reasoning that
first prove that the claimed inequality holds {\em approximately} for any point $\pi\in\Pi_\delta$ at which $\Phi$ is differentiable.  
% Since the subgradient of the convex and differentiable function $\tilde\Phi$ is a singleton, 
In this case the subdifferential $\partial \Phi(\pi)=$ $\{\nabla \Phi(\pi)\}$ is a singleton, and a generalization of Danskin's theorem (Theorem~\ref{thm:danskin}) implies that $\nabla\Phi(\pi)=\nabla_\pi \Value{\pi}{P^\star}(s_0)$ for any $P^\star\in\argmax_{P\in\mc P}\Value{\pi}{P}(s_0)$. As Lemma~\ref{lemma:extend:V}\ref{eq:gd:fix:P} holds in particular for~$P=P^\star$, we have
\vspace{-5pt}
\begin{equation}\label{eq:KL:phi}
\begin{split}
\max_{\pi'\in\Pi_\delta}\langle \pi-\pi', \nabla_\pi  \Value{\pi}{P^\star}(s_0)\rangle
&\geq C^{-1}\left(\Value{\pi}{P^\star}(s_0)-\min_{\pi'\in\Pi}\Value{\pi'}{P^\star}(s_0)\right) -\delta 
\\&\ge  C^{-1}(\Phi(\pi)-\Phi(\pistar))-\delta,
\end{split}
\vspace{-5pt}
\end{equation}
where the second inequality holds because
% Since $\Value{\pi'}{P^\star}(s_0)\le\max_{P\in\mc P} \Value{\pi'}{P}(s_0)$ for every $\pi',$
% we have
$
\min_{\pi\in\Pi}\Value{\pi}{P^\star}(s_0)\le\min_{\pi\in\Pi}\max_{P\in\mc P}\Value{\pi}{P}(s_0)$ $=\Phi(\pistar).
$
% Substituting the above inequality into~\eqref{eq:KL:phi} yields
% \begin{align*}
% \max_{\bar{\pi}\in\Pi}\langle \bar{\pi}-\pi, \nabla_\pi \Value{\pi}{P^\star}(s_0)\rangle \geq C (\Phi(\pi)-\Phi(\pistar)).
% \end{align*}
The Cauchy-Schwarz inequality then allows us to conclude that
\vspace{-5pt}
\begin{align*}
\frac{\sqrt{2S}+2\delta}{\sqrt{2S}+2\delta}\|\nabla\Phi(\pi)\|_{\mathbf F}
&\ge \frac{1}{\sqrt{2S}+2\delta} \max_{\pi'\in\Pi_\delta,\|\pi'-\pi\|_{\mathbf F}\le \sqrt{2S}+2\delta}\langle \pi-\pi', \nabla\Phi(\pi)\rangle
\\&= \frac{1}{\sqrt{2S}+2\delta} \max_{\pi'\in\Pi_\delta}\langle \pi-\pi', \nabla_\pi  \Value{\pi}{P^\star}(s_0)\rangle
\\&\ge \frac{1}{C(\sqrt{2S}+2\delta)}(\Phi(\pi)-\Phi(\pistar))-\frac{\delta}{\sqrt{2S}+2\delta}.
\end{align*}
The equality in the above expression holds because there exist $\pi'_\delta,\pi_\delta\in\Pi$ with $\|\pi-\pi_\delta\|_{\mathbf F}\le~\delta$ and $\|\pi'-\pi'_\delta\|_{\mathbf F}\le~\delta$ and because $\|\pi'_\delta-\pi_\delta\|_{\mathbf F}\le\sqrt{2S}$ thanks to Lemma~\ref{frob:bound:pi}. This implies that $\|\pi'-\pi\|_{\mathbf F}\le\|\pi'-\pi'_\delta\|_{\mathbf F}+\|\pi'_\delta-\pi_\delta\|_{\mathbf F}+\|\pi_\delta-\pi\|_{\mathbf F}\le\sqrt{2S}+2\delta$. The second inequality follows from~\eqref{eq:KL:phi}.
Next, set $\epsilon\! =\!\|\nabla \Phi_{1/(2\ell+2\delta)}(\pi)\|_{\mathbf F} $. By Lemma~\ref{lemma:primal:moreau}\ref{primal:item:small:subgradient}, there is $\hat\pi\in\Pi_\delta$ such that $\|\pi-\hat\pi\|_{\mathbf F}\le \epsilon/(2\ell+2\delta)$ and $\min_{v\in\partial\Phi(\hat\pi)}\|v\|_{\mathbf F}\le\epsilon.$ Theorem~\ref{thm:danskin} thus implies that there exists $\hat P\in\argmax_{P\in\mc P}\Value{\hat\pi}{P}(s_0)$ with $\|\nabla_{\pi}\Value{\hat\pi}{\hat P}(s_0)\|_{\mathbf F}\le \epsilon$. We then find
% \vspace{-5pt}
\begin{equation}\label{eq:hat:pi:ineq}
  \begin{split}
      \epsilon\ge \|\nabla_{\pi}\Value{\hat\pi}{\hat P}(s_0)\|_{\mathbf F}&\ge 
    \frac{1}{\sqrt{2S}+2\delta}\max_{\pi'\in\Pi_\delta}\langle \hat\pi-\pi', \nabla_\pi  \Value{\hat\pi}{\hat P}(s_0)\rangle
    \\&\geq \frac{1}{C(\sqrt{2S}+2\delta)}(\Value{\hat\pi}{\hat P}(s_0)-\min_{\pi\in\Pi}\Value{\pi}{\hat P}(s_0)) -\frac{\delta}{\sqrt{2S}+2\delta} 
    \\&\ge  \frac{1}{C(\sqrt{2S}+2\delta)}(\Phi(\hat\pi)-\Phi(\pistar))-\frac{\delta}{\sqrt{2S}+2\delta},
  \end{split}  
\end{equation}
where the second inequality follows from the Cauchy-Schwarz inequality and our earlier insight that $\|\hat\pi-\pi'\|_{\mathbf F}\le \sqrt{2S}+2\delta$ for all $\hat\pi,\pi'\in\Pi_\delta$, the third inequality follows from Lemma~\ref{lemma:extend:V}\ref{eq:gd:fix:P}, and the fourth inequality holds because $\Value{\hat\pi}{\hat P}(s_0)=\Phi(\hat\pi)$ and
$\min_{\pi\in\Pi}\Value{\pi}{\hat P}(s_0)\le\min_{\pi\in\Pi}\max_{P\in\mc P}\Value{\pi}{P}(s_0)=\Phi(\pistar).
$
% \begin{align*}
% \Phi(\hat\pi)-\Phi(\pistar)\le \epsilon C\sqrt{2S}.
% \end{align*}
The above reasoning implies that
\begin{equation*}
\begin{split}
    \Phi(\pi)-\Phi(\pistar)&=\Phi(\hat\pi)-\Phi(\pistar)+\Phi(\pi)-\Phi(\hat\pi) \\&\le C(\epsilon(\sqrt{2S}+2\delta)+\delta)+ (L+\delta)\|\hat\pi-\pi\|_{\mathbf F} \\&\le  \left(C(\sqrt{2S}+2\delta)+\frac{L+\delta}{2(\ell+\delta)}\right)\epsilon+ C
    \delta, 
\end{split}
\end{equation*}
where the first inequality follows from~\eqref{eq:hat:pi:ineq} and Lemma~\ref{lemma:primal:moreau}\ref{primal:item:primal:weak:convex},
and the second inequality holds because $\|\hat\pi-\pi\|_{\mathbf F}\le \epsilon/(2\ell+2\delta)$. As $\epsilon\! =\!\|\nabla \Phi_{1/(2\ell+2\delta)}(\pi)\|_{\mathbf F}$, we thus have 
\begin{align}\label{ineq:phi:delta:nondiff}
 \Phi(\pi)-\Phi(\pi^\star)\le  
 \left(C(\sqrt{2S}+2\delta)+\frac{L+\delta}{2(\ell+\delta)}\right)\|\nabla \Phi_{1/(2\ell+2\delta)}(\pi)\|_{\mathbf F}+C\delta.
\end{align}
Hence, if $\delta$ is small, the claimed gradient dominance condition holds {\em approximately} at any point $\pi\in\Pi_\delta$ where the primal function~$\Phi$ is differentiable.

% due to Lemma~\ref{lemma:primal:moreau}\ref{primal:item:small:subgradient}. 
% As $\epsilon\! =\!\|\nabla \Phi_{1/(2\ell+2\delta)}(\pi)\|_{\mathbf F} $ and the inequality~\eqref{eq:phi:delta} holds for all $\delta>0$, the claim thus holds for any $\pi\in\Pi$ at which~$\Phi$ is differentiable.

Consider now an arbitrary $\pi\in\Pi$ irrespective of whether or not $\Phi$ is differentiable at~$\pi$. Let $\{\pi_t\}_{t=0}^\infty$ be a sequence in an open neighborhood of $\Pi$ converging to~$\pi$ such that $\Phi$ is differentiable at~$\pi_t$ and $\pi_t\in\Pi_{\delta_t}$ with $\delta_t=1/t$ for every $t\in\mathbb N$. From the inequality~\eqref{ineq:phi:delta:nondiff} established in the first part of the proof we know that 
\begin{align*}
 \Phi(\pi_t)-\Phi(\pi^\star)\le  
 \left(C(\sqrt{2S}+2\delta_t)+\frac{L+\delta_t}{2(\ell+\delta_t)}\right)\|\nabla \Phi_{1/(2\ell+2\delta_t)}(\pi_t)\|_{\mathbf F}+C\delta_t\quad \forall t\in\mathbb N.
\end{align*}
The claim then follows because $\pi_t$ converges to $\pi$ and $\delta_t$ converges to $0$, while~$\Phi$ as well as the gradient $\nabla \Phi_{1/(2\ell+2\delta)}$ of its Moreau envelope are continuous at~$\pi$.
\end{proof}
With all these preparatory results at hand, we are now ready to characterize the convergence behavior of Algorithm~\ref{alg:PG-min-oracle}.

\begin{theorem}[Convergence of Algorithm~\ref{alg:PG-min-oracle}]\label{thm:PO:convergence}
% Let $L=\sqrt{A}/(1-\gamma)^2$ and $\ell=2\gamma A/(1-\gamma)^3$ and suppose that
If Assumption~\ref{assume:irreducible:pi} holds, $\epsilon=L\sqrt{2S/K}/2$ and $\eta=\sqrt{2S/K}/L$, then the iterates $\{\pik\}_{k=0}^{K-1}$ of Algorithm~\ref{alg:PG-min-oracle} satisfy
  \begin{align*}
    \frac{1}{K}\sum_{k=0}^{K-1}\left( \Value{\pik}{\star}(s_0)-\min_{\pi\in\Pi}\Value{\pi}{\star}(s_0)\right)\le  \frac{(72S)^{1/4}(C\sqrt{2\ell L S}+L\sqrt{L/\ell}/2)}{K^{1/4}}. %2\sqrt{2S/K}/L
  \end{align*}
\end{theorem}

\begin{proof}[Proof of Theorem~\ref{thm:PO:convergence}]
We have 
\begin{align*}
\frac{1}{K}\sum_{k=0}^{K-1}\left( \Value{\pik}{\star}(s_0)-\min_{\pi\in\Pi}\Value{\pi}{\star}(s_0)\right)
&=\frac{1}{K}\sum_{k=0}^{K-1}\left( \Phi(\pik)-\min_{\pi\in\Pi}\Phi(\pi)\right)
\\
&\le \frac{(C\sqrt{2S}+L/(2\ell))}{K}\sum_{k=0}^{K-1}\|\nabla \Phi_{1 / (2 \ell)}(\pik)\|_{\mathbf F}
\\&\le\frac{(C\sqrt{2S}+L/(2\ell))}{\sqrt{K}} \sqrt{\sum_{k=0}^{K-1}\|\nabla \Phi_{1 / (2 \ell)}(\pik)\|_{\mathbf F}^2}
% \\&{\color{blue}\le \frac{(C\sqrt{2S}+L/(2\ell))}{\sqrt{K}} \sqrt{\frac{4 \ell S}{\eta}+2 K \eta \ell L^2+4 \ell \epsilon K}}
\\&\le \frac{(C\sqrt{2S}+L/(2\ell))(72S)^{1/4}(\ell L)^{1/2}}{K^{1/4}},
\end{align*}
where the equality exploits the definition of the primal function~$\Phi$, while the three inequalities follow from Lemma~\ref{lemma:gd:primal}, Jensen's inequality and Lemma~\ref{lemma:gdmax:converge} with $\epsilon=L\sqrt{2S/K}/2$ and $\eta=\sqrt{2S/K}/L$, respectively.
\end{proof}

% {\color{blue} 
Theorem~\ref{thm:PO:convergence} implies that an $\varepsilon$-optimal solution of the robust policy improvement problem~\eqref{def:policy:learning:objective} can be computed in $K=\mc O(1/\varepsilon^4)$ iterations provided that the robust policy evaluation oracle is guaranteed to output an $\epsilon$-optimal solutions with $\epsilon=\mc O(\varepsilon^2)$. Arbitrarily accurate solutions of the robust policy improvement problem~\eqref{def:policy:learning:objective} can thus only be obtained if the error~$\epsilon$ of the robust policy evaluation oracle can be made arbitrarily small. If the robust policy evaluation oracle has a fixed accuracy~$\epsilon$, however, then the robust policy improvement problem~\eqref{def:policy:learning:objective} can only be solved up to a certain accuracy. The following corollary of Theorem~\ref{thm:PO:convergence} formalizes this statement.

%If $\epsilon$ is not set as a function of total iteration number~$K,$ then one can derive the following corollary with a global non-diminishing policy evaluation error.
\begin{corollary}[Performance of Algorithm~\ref{alg:PG-min-oracle} under constant policy evaluation error] \label{cor:performance:of:algo}
Suppose that Assumption~\ref{assume:irreducible:pi} holds and that the robust policy evaluation oracle (called in step 3 of Algorithm~\ref{alg:PG-min-oracle}) is solved to a constant error $\epsilon\geq 0$. If the learning rate is $\eta=\sqrt{2S/K}/L$, then the iterates $\{\pik\}_{k=0}^{K-1}$ of Algorithm~\ref{alg:PG-min-oracle} satisfy
\begin{equation*}
    \frac{1}{K}\sum_{k=0}^{K-1}\left( \Value{\pik}{\star}(s_0)-\min_{\pi\in\Pi}\Value{\pi}{\star}(s_0)\right)\le  C_1 K^{-1/4} + C_2 \sqrt{\epsilon},
\end{equation*}
where $C_1=(32S)^{1/4}(C\sqrt{2S}+L/(2\ell))\sqrt{\ell L}$ and $C_2=(C\sqrt{2S}+L/(2\ell)) \sqrt{4\ell}$.
\end{corollary}
% {\color{red}[Can we find the constants $C_1$ and $C_2$?]}
\begin{proof}[Proof of Corollary~\ref{cor:performance:of:algo}]
By using a similar reasoning as in the proof of Theorem 4.5, we obtain
\begin{align*}
        &\frac{1}{K}\sum_{k=0}^{K-1}\left( \Value{\pik}{\star}(s_0)-\min_{\pi\in\Pi}\Value{\pi}{\star}(s_0)\right) \\
        &\qquad\le\frac{(C\sqrt{2S}+L/(2\ell))}{\sqrt{K}} \sqrt{\sum_{k=0}^{K-1}\|\nabla \Phi_{1 / (2 \ell)}(\pik)\|_{\mathbf F}^2}
\\&\qquad\le \frac{(C\sqrt{2S}+L/(2\ell))}{\sqrt{K}} \sqrt{\frac{4 \ell S}{\eta}+2 K \eta \ell L^2+4 \ell \epsilon K}
\\&\qquad\le \frac{(C\sqrt{2S}+L/(2\ell))}{\sqrt{K}} \left(\sqrt{\frac{4 \ell S}{\eta} + 2 K \eta \ell L^2}+ \sqrt{4 \ell \epsilon K}\right) 
\\&\qquad= \frac{(C\sqrt{2S}+L/(2\ell))}{\sqrt{K}} \left(\sqrt{4\sqrt{2KS}\ell L}+ \sqrt{4 \ell \epsilon K}\right) 
\\&\qquad= \frac{(32S)^{1/4}(C\sqrt{2S}+L/(2\ell))\sqrt{\ell L}}{K^{1/4}}  + (C\sqrt{2S}+L/(2\ell)) \sqrt{4\ell \epsilon} .
\end{align*}
The first inequality follows from the definition of~$\Phi$, from Lemma~\ref{lemma:gd:primal} and from Jensen's inequality. The second inequality uses Lemma~\ref{lemma:gdmax:converge}, and the third inequality holds because $\sqrt{x+y}< \sqrt{x}+\sqrt{y}$ for all $x,y>0.$ The first equality, finally, follows from the choice of the learning rate $\eta$.
\end{proof}
% }
A similar global convergence result for a projected gradient descent algorithm with access to an approximate robust policy evaluation oracle was established in~\cite{wang2022convergence}. However, no robust policy evaluation oracle for general non-rectangular uncertainty sets is described, and its accuracy is required to increase geometrically with the number of iterations of the algorithm. 
% {\color{red}The convergence proof in~\cite{wang2022convergence} also relies on the implicit assumption that the set of worst-case transition kernels for any given policy is finite, which would be difficult to check in practice. In contrast, Theorem~\ref{thm:PO:convergence} does not rely on such an assumption.}
% {\color{blue}
Also, the proof strategy leading up to the gradient dominance condition is methodologically different from ours. In \cite{wang2022convergence} any subgradient of the primal function must be expressed as a convex combination of gradients of the value function evaluated at a finite number of worst-case kernels. In contrast, we approximate the subgradients at non-differentiable points by sequences of gradients at nearby differentiable points. This alternative proof technique is arguably more transparent and more directly reveals the gradient dominance property of the Moreau envelope of the primal function.
% }

% \newpage
\vspace{-5pt}
\section{Numerical Experiments}
\label{sect:exp}
% \input{text/3-Exp.tex}

%%%%%%%%%%%%%%
We assess the performance of the proposed algorithms on standard test problems: A stochastic GridWorld problem~\cite{sutton2018reinforcement}, randomly generated Garnet MDPs~\cite{archibald1995generation}, and a machine replacement problem~\cite{delage2010percentile}. 
% We benchmark against state-of-the-art algorithms on those instances.
Sections~\ref{ssec:rect:non:amb:sets} and~\ref{ssec:rec:amb} focus on robust policy evaluation. Section~\ref{ssec:rect:non:amb:sets} first compares the solution qualities of the projected Langevin dynamics algorithm (PLD, Algorithm~\ref{alg:PLD}) and the Frank-Wolfe algorithm (FW, Algorithm~\ref{alg:CPI}) in the context of a GridWorld problem. %We show that the outputs of both methods match for rectangular uncertainty sets but that PLD may significantly outperform CPI when the uncertainty set fails to be rectangular. 
Section~\ref{ssec:rec:amb} uses Garnet MDPs to assess the runtime performance of FW against that of the state-of-the-art projected gradient descent algorithm for robust policy evaluation described in~\cite{wang2022convergence}. Section~\ref{ssec:non:rec:amb}, finally, focuses on a machine replacement problem and compares the actor-critic algorithm (ACA, Algorithm~\ref{alg:PG-min-oracle}) against the only existing method for robust policy improvement with non-rectangular uncertainty sets described in~\cite{wiesemann2013robust}. 
%in the context of a standard machine replacement problem with a non-rectangular uncertainty set. Note that~\cite{wiesemann2013robust} is the  and the only approach that can handle non-rectangular uncertainty sets. This comparison is performed on the uncertainty set induced by maximum likelihood estimation.
All experiments are implemented in Python, and are run on an Intel~i7-10700 CPU~(2.9GHz) computer with 16
GB~RAM.
%%%%%%%%%%%
\vspace{-5pt}

\subsection{Stochastic GridWorld: Rectangular and Non-Rectangular Uncertainty Sets}
\label{ssec:rect:non:amb:sets}
The purpose of the first experiment is to show that PLD outputs the same policy value as FW when the uncertainty set is rectangular but may output a higher policy value than FW otherwise. Our experiment is based on a stylized GridWorld problem, which is widely studied in reinforcement learning~\cite{sutton2018reinforcement}. Specifically, the state space $\mc S$ comprises the $25$ cells of a $5\times 5$ grid, and the action space~$\mc A$ comprises the~$4$ directions ``up,'' ``down,'' ``left,'' and ``right.'' An agent moves across the grid with the aim to reach the Goal State in cell~$1$ (in the top left corner) while avoiding the Bad State in cell~$25$ (in the bottom right corner). If the agent resides in cell $s\in\mc S$ and selects action $a\in\mc A$, then she moves to cell~$s'\in\mc S$ with probability $P(s'|s,a)$. The agent incurs a cost of~$10$ in the Bad State, a cost of~$0$ in the Goal State, and a cost of~$0.2$ in any other state. The initial state~$s_0$ is assumed to follow the uniform distribution on~$\mc S$, which we denote as~$\rho$, and the discount factor is set to~$\gamma=0.9$. We also assume that the agent's knowledge is captured by an uncertainty set~$\mc P$, which is defined as some neighborhood of a reference transition kernel~$P_{\rm ref}$. In the following we define~$\mc S(s)\subseteq \mc 
S$ as the set of all cells adjacent to~$s$. We set $P_{\rm ref}(s'|s,a)=0.7$ if~$s'$ is the cell adjacent to~$s$ in direction~$a$, $P_{\rm ref}(s'|s,a)=0.1$ if~$s'$ is any other cell adjacent to~$s$, $P_{\rm ref}(s'|s,a)=1-\sum_{s''\in \mc S(s)}P_{\rm ref}(s''|s,a)$ if $s'=s$, and $P_{\rm ref}(s'|s,a)=0$ otherwise. If there is no cell adjacent to~$s$ in direction~$a$, then we set $P_{\rm ref}(s'|s,a)=0.1$ if~$s'$ is any cell adjacent to~$s$, $P_{\rm ref}(s'|s,a)=1-\sum_{s''\in \mc S(s)}P_{\rm ref}(s''|s,a)$ if $s'=s$, and $P_{\rm ref}(s'|s,a)=0$ otherwise. 
%In the case of~$s$ having less than~$4$ neighbors, we distribute the remaining probabilities to the self-transition. 
Our goal is to compute the worst-case net present cost $\Value{\pi}{\star}(\rho)=\max_{P\in\mc P} \sum_{s_0\in\mc S}\rho(s_0) \Value{\pi}{P}(s_0)$ of the policy $\pi\in\Pi$ that selects actions randomly from the uniform distribution on~$\mc A$ irrespective of the current state. 

Gradient-based methods such as PLD or FW can be used to compute $\Value{\pi}{\star}(\rho)$ even if the initial state is random. In this case, however, nature's policy gradients of the form $\nabla_P\Value{\pi}{P}(s_0)$ must be replaced with $\sum_{s_0\in\mc S}\rho(s_0)\,\nabla_P\Value{\pi}{P}(s_0)$. Throughout the first experiment we employ PLD with Gibbs parameter~$\beta=160$, stepsize~$\eta=0.8$, initial iterate $\xi^{(0)}=\xi_{\rm ref}$ corresponding to the nominal transition kernel~$P_{\rm ref}$ and~$M=100$ iterations. In addition, we use FW with tolerance~$\epsilon=10^{-2}$, stepsizes $\{\alpha_m\}_{m=0}^{\infty}$ chosen as in Theorem~\ref{thm:CPI:PE} and initial iterate $P^{(0)}=P_{\rm ref}$. We also work with variants of the PLD and FW algorithms that output the best iterates found during execution. We first assume that~$\mc P$ constitutes an $(s,a)$-rectangular uncertainty set of the form
\[
    \mc P=\left\{P\in\Delta(\mc S)^{S\times A}\;:\; \left\|P(\cdot|s,a)-P_{\rm ref}(\cdot|s,a) \right\|_2\le r \ \forall s\in\mc S, \, a\in\mc A \right\}
\]
with size parameter~$r\ge 0$. Note that if $P\in\mc P$ and~$r>0$, then $P(s'|s,a)$ can be strictly positive even if~$s'$ is not adjacent to~$s$. Figure~\ref{fig:PE:rectangular} shows the worst-case policy values output by PLD averaged over $20$ independent simulation runs and compares them against the deterministic values output by FW. We highlight that the standard deviations of the values output by PLD range from~ $6.50\times 10^{-5}$ to~$0.12$ and are therefore practically negligible. As expected from Theorems~\ref{thm:convergence:PLD} and~\ref{thm:CPI:PE}, we observe that the two algorithms are consistent. That is, if the uncertainty set is rectangular, then both PLD and FW succeed in solving the robust policy evaluation problem to global optimality. Figure~\ref{fig:PE:traj:rectangular} visualizes the policy values associated with the iterates $\xi^{(m)}$ of a single simulation run of PLD, illustrating the exploratory nature of the algorithm. Specifically, we see that
%in finding an approximate global maximum as the iterations~$m$ increase and that the iteration values begin to 
for large~$m$ the policy values oscillate around a constant level.

\begin{figure}[!htbp]
    \centering
    \begin{subfigure}{0.457\textwidth}
    \centering
    \begin{adjustbox}{width=0.8\linewidth}
    % This file was created with tikzplotlib v0.10.1.
\begin{tikzpicture}
\definecolor{crimson2143940}{RGB}{214,39,40}
\definecolor{darkgray176}{RGB}{176,176,176}
\definecolor{darkorange25512714}{RGB}{255,127,14}
\definecolor{forestgreen4416044}{RGB}{44,160,44}
\definecolor{steelblue31119180}{RGB}{31,119,180}
\begin{axis}[
log basis x={10},
tick align=outside,
tick pos=left,
x grid style={darkgray176},
xlabel={size parameter $r$},
xmin=0.0009, xmax=1.1,
xmode=log,
xtick style={color=black},
xtick={1e-05,0.0001,0.001,0.01,0.1,1,10},
xticklabels={
  % \(\displaystyle {10^{-7}}\),
  % \(\displaystyle {10^{-6}}\),
  \(\displaystyle {10^{-5}}\),
  \(\displaystyle {10^{-4}}\),
  \(\displaystyle {10^{-3}}\),
  \(\displaystyle {10^{-2}}\),
  \(\displaystyle {10^{-1}}\),
  \(\displaystyle {10^{0}}\),
  \(\displaystyle {10^{1}}\)
},
y grid style={darkgray176},
ylabel={$V_{\pi}^{\star}(\rho)$},
ymax=35, ymin=5,
ytick style={color=black},
legend style={
    at={(0.3,0.8)}, % Adjust the position of the legend
    anchor=south east, % Set the anchor to the bottom right
    legend cell align=left,
    draw=none, % Remove the legend border
    fill=none, % Remove the legend background
    font=\small % Set the font size for the legend
  }
]

\path [fill=blue, fill opacity=0.2]
(axis cs:0.001,5.865727)
--(axis cs:0.01,6.092396909)
--(axis cs:0.1,8.182523892)
--(axis cs:1,32.9423076)
--(axis cs:1,32.55518104)
--(axis cs:0.1,8.158394643)
--(axis cs:0.01,6.090051184)
--(axis cs:0.001,5.865513)
--cycle;

\addplot [semithick, steelblue31119180, mark=*, mark size=3, mark options={solid}]
table {%
0.001 5.865619905
0.01 5.865727
0.1 8.170459267
1 32.74874432
};\addlegendentry{PLD}

\addplot [semithick, crimson2143940 ,mark=asterisk, mark size=3, mark options={solid}]
table {%
0.001 5.859258919371496
0.01 6.118144033570185
0.1 8.50248329582078
1 33.83825996286552
};\addlegendentry{CPI}
\end{axis}

\end{tikzpicture}
    \end{adjustbox}           
    \caption{Robust policy value output by PLD (blue) and FW (red) for an $(s,a)$-rectan\-gular uncertainty set of varying size~$r$. 
    % Shaded area shows the empirical $90\%$ confidence region
    }
    \label{fig:PE:rectangular}
    \end{subfigure}
    % \hfill
    \centering
    \raisebox{1pt}[0pt][0pt]{
    \begin{subfigure}{0.45\textwidth}
   \begin{adjustbox}{width=0.8\linewidth}
    % This file was created with tikzplotlib v0.10.1.
\begin{tikzpicture}

\definecolor{darkgray176}{RGB}{176,176,176}
\definecolor{steelblue31119180}{RGB}{31,119,180}
\begin{axis}[
log basis x={10},
tick align=outside,
tick pos=left,
x grid style={darkgray176},
xlabel={iteration counter $m$},
xmin=0.9, xmax=150,
xmode=log,
xtick style={color=black},
xtick={0.01,0.1,1,10,100,1000,10000,100000},
xticklabels={
  \(\displaystyle {10^{-2}}\),
  \(\displaystyle {10^{-1}}\),
  \(\displaystyle {10^{0}}\),
  \(\displaystyle {10^{1}}\),
  \(\displaystyle {10^{2}}\),
  \(\displaystyle {10^{3}}\),
  \(\displaystyle {10^{4}}\),
  \(\displaystyle {10^{5}}\)
},
y grid style={darkgray176},
ylabel={$V_{\pi}^{P^{\xi^{(m)}}}(\rho)$},
ymax=35, ymin=5,
ytick style={color=black},
legend style={
    at={(0.9,0.02)}, % Adjust the position of the legend
    anchor=south east, % Set the anchor to the bottom right
    legend cell align=left,
    draw=none, % Remove the legend border
    fill=none, % Remove the legend background
    font=\small % Set the font size for the legend
  }
]

\addplot [semithick, steelblue31119180]
table {
0 5.84
1 5.84
2 19.4122154410229
3 26.6715023820442
4 30.1513021553255
5 30.9813120670705
6 31.5492133845308
7 31.8697023058361
8 31.9369994905341
9 31.9543879467354
10 32.0995675748827
11 32.002638629924
12 32.0902449275542
13 32.1019990037716
14 32.1560973651982
15 32.1782300200018
16 32.0908090132001
17 32.0583832279242
18 32.2752400693617
19 32.2839645507922
20 32.4283130142114
21 32.453849069596
22 32.3420047472245
23 32.4164726680961
24 32.4815809719849
25 32.5127027261467
26 32.3144882095683
27 32.4055503987357
28 32.4261493375826
29 32.3502993249654
30 32.4751937322006
31 32.4422826817547
32 32.3554906599682
33 32.3073697019279
34 32.5014265931861
35 32.1962113727954
36 32.4737081709225
37 32.2494569164903
38 32.5071723019385
39 32.2017559348698
40 32.4290717155663
41 32.3119500997628
42 32.4184851385649
43 32.1501661474638
44 32.2837929569973
45 32.4017309237501
46 32.4078260537918
47 32.3706697790202
48 32.4180685191362
49 32.5589687857205
50 32.3656994829759
51 32.5381508031091
52 32.5352149809753
53 32.4866142738038
54 32.4089036264262
55 32.1664012649808
56 32.1365866204281
57 32.3854949249439
58 32.324363552725
59 32.1676594884134
60 32.4737028432549
61 32.1687922013788
62 32.2842653964102
63 32.2322712728331
64 32.4065098099054
65 32.2605342912729
66 32.3273789571945
67 32.2396363009171
68 32.3521514408562
69 32.4381363845207
70 32.4466221960369
71 32.4654016835005
72 32.3721907485192
73 32.4755908781986
74 32.4090719923418
75 32.4168568603514
76 32.3468811308774
77 32.312659753504
78 32.3319701177595
79 32.3487425447622
80 32.2156343984812
81 32.3333718381061
82 32.1795769539094
83 32.233341730276
84 32.3769871271865
85 32.3276349409199
86 32.4638467060958
87 32.3640954278788
88 32.2086464712105
89 32.1621347945714
90 32.4292128663404
91 32.3450473444064
92 32.3874570839556
93 32.4703730658411
94 32.4831122362043
95 32.2795948570199
96 32.546352968222
97 32.2411097639444
98 32.3244384858257
99 32.3789156061737
100 32.535383437112
101 32.5195185439209
};\addlegendentry{PLD} 
\end{axis}

\end{tikzpicture}
    \end{adjustbox}
    \caption{Trajectory of policy values computed by PLD for an $(s,a)$-rectangular uncertainty set of fixed size~$r=10$. 
    }
    \label{fig:PE:traj:rectangular}
    \end{subfigure}}
    \vspace{-15pt}
    \caption{Comparison of PLD (Algorithm~\ref{alg:PLD}) against FW (Algorithm~\ref{alg:CPI}) on a stochastic GridWorld problem with an $(s,a)$-rectangular uncertainty set.}
    \label{fig:rect}
    \vspace{-20pt}
\end{figure}
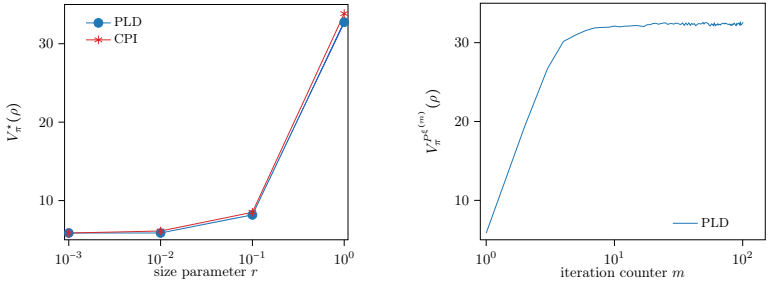

Next, we assume that $\mc P$ constitutes a non-rectangular ambiguity set of the form 
\[
    \mc P=\left\{P^\xi \;:\;(\xi-\xi_{\rm ref})^\top H (\xi-\xi_{\rm ref})\leq r \right\}
\]
with size parameter~$r\ge 0$ and Hessian matrix $H=\mathrm{diag}(1,2,\ldots,(S-1)SA)$. As shown in Appendix~\ref{example:MLE:detail}, ellipsoidal uncertainty sets of this type naturally emerge when maximum likelihood estimation is used to construct statistically optimal confidence regions for~$\xi$. Figure~\ref{fig:PE:nonrectangular} shows the worst-case policy values output by PLD (averaged over 20 independent simulation runs) and FW. The standard deviations of the values output by PLD range from~$3.57\times 10^{-2}$ to~$7.75\times 10^{-2}$ and are thus again negligible. We observe that for $r\lesssim 1$ PLD reports higher worst-case policy values than FW. This suggests that the deterministic FW method may get trapped in local maxima, while the randomized PLD method manages to escape local maxima. For $r\gtrsim 1$ the outputs of PLD and FW match. This is to be expected from Theorem~\ref{thm:CPI:PE} because the uncertainty set~$\mc P$ converges to the $(s,a)$-rectangular product simplex $\Delta(\mc S)^{S\times A}$---and thus becomes increasingly rectangular---as~$r$ grows.  %eventually all possible transition kernels are considered, decreasing the degree of non-rectangularity, which is consistent . 
%becomes more and more conservative as more and more possible transition kernels are taken into account in the uncertainty set. In addition, the shaded region represents the~$90\%$ empirical confidence interval for the mean of the output of PLD. Moreover, it shows that CPI (Algorithm~\ref{alg:CPI}), while being deterministic, finds a local optimum comparing to PLD, which reports a better value upon termination.
Figure~\ref{fig:PE:nonrectangular} visualizes the policy values associated with the iterates $\xi^{(m)}$ of a single simulation run of PLD. 
%also shows that the larger the radius $\alpha$ is, the close the estimations of two distinct algorithms are. Figure~\ref{fig:PE:traj:non:rectangular} displays the algorithm's iteration behavior that the policy values oscillate around a constant value as the iteration number increases.
We remark that PLD can outperform FW by up to 80\% on $2\times 2$ GridWorld problems (not shown).

Table~\ref{tab:runtime:nonrect} shows the runtimes of PLD and FW for non-rectangular uncertainty sets of different sizes. Despite the suboptimal theoretical convergence rate, PLD is empirically faster than FW while producing more accurate solutions for robust policy evaluation problems with non-rectangular uncertainty sets.
% \vspace{-20pt}

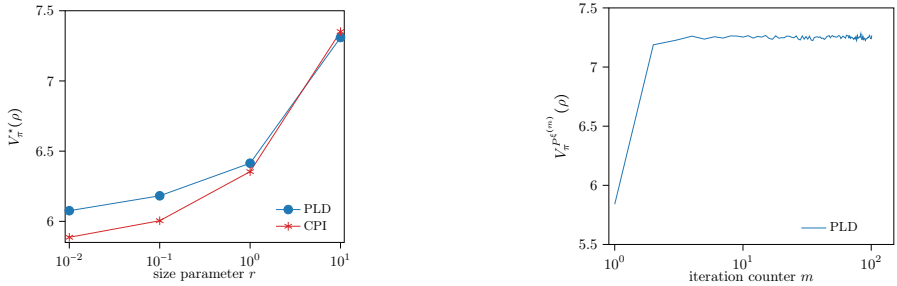
\begin{figure}[!htb]
    \centering
    \begin{subfigure}{0.45\textwidth}
    \begin{adjustbox}{width=0.8\linewidth}
    % This file was created with tikzplotlib v0.10.1.
\begin{tikzpicture}

\definecolor{darkgray176}{RGB}{176,176,176}
\definecolor{darkorange25512714}{RGB}{255,127,14}
\definecolor{steelblue31119180}{RGB}{31,119,180}
\definecolor{crimson2143940}{RGB}{214,39,40}

\begin{axis}[
log basis x={10},
tick align=outside,
tick pos=left,
x grid style={darkgray176},
xlabel={size parameter $r$},
xmin=0.009, xmax=11,
xmode=log,
xtick style={color=black},
xtick={0.001,0.01,0.1,1,10,100,1000},
xticklabels={
  \(\displaystyle {10^{-3}}\),
  \(\displaystyle {10^{-2}}\),
  \(\displaystyle {10^{-1}}\),
  \(\displaystyle {10^{0}}\),
  \(\displaystyle {10^{1}}\),
  \(\displaystyle {10^{2}}\),
  \(\displaystyle {10^{3}}\)
},
y grid style={darkgray176},
ylabel={$V_{\pi}^{\star}(\rho)$},
ymax=7.5, ymin=5.85,
ytick style={color=black},
legend style={
    at={(0.98,0.02)}, % Adjust the position of the legend
    anchor=south east, % Set the anchor to the bottom right
    legend cell align=left,
    draw=none, % Remove the legend border
    fill=none, % Remove the legend background
    font=\small % Set the font size for the legend
  }
]

\addplot [semithick, steelblue31119180, mark=*, mark size=3, mark options={solid}]
table {%
0.01 6.075820608
0.1 6.181987842
1 6.414008407
10 7.30935177
};\addlegendentry{PLD}
\addplot [semithick, crimson2143940, mark=asterisk, mark size=3, mark options={solid}]
table {%
0.01 5.886482828593631
0.1 6.004422387914141
1 6.354287936546278
10 7.351553536405394
};\addlegendentry{CPI}

\end{axis}

\end{tikzpicture}
    \end{adjustbox}
    \caption{Robust policy value output by PLD (blue) and FW (red) for a non-rectan\-gular uncertainty set of varying size~$r$. 
    }
        \label{fig:PE:nonrectangular}
    \end{subfigure}
    \hfill
    \begin{subfigure}{0.45\textwidth}
\begin{adjustbox}{width=0.8\linewidth}
    % This file was created with tikzplotlib v0.10.1.
\begin{tikzpicture}

\definecolor{darkgray176}{RGB}{176,176,176}
\definecolor{steelblue31119180}{RGB}{31,119,180}

\begin{axis}[
log basis x={10},
tick align=outside,
tick pos=left,
x grid style={darkgray176},
xlabel={iteration counter $m$},
xmin=0.9, xmax=150,
xmode=log,
xtick style={color=black},
xtick={0.01,0.1,1,10,100,1000,10000,100000},
xticklabels={
  \(\displaystyle {10^{-2}}\),
  \(\displaystyle {10^{-1}}\),
  \(\displaystyle {10^{0}}\),
  \(\displaystyle {10^{1}}\),
  \(\displaystyle {10^{2}}\),
  \(\displaystyle {10^{3}}\),
  \(\displaystyle {10^{4}}\),
  \(\displaystyle {10^{5}}\)
},
y grid style={darkgray176},
ylabel={$V_{\pi}^{P^{\xi^{(m)}}}(\rho)$},
ymax=7.5, ymin=5.5,
ytick style={color=black},
legend style={
    at={(0.9,0.02)}, % Adjust the position of the legend
    anchor=south east, % Set the anchor to the bottom right
    legend cell align=left,
    draw=none, % Remove the legend border
    fill=none, % Remove the legend background
    font=\small % Set the font size for the legend
  }
]

\addplot [semithick, steelblue31119180]
table {%
0 5.84
1 5.84
2 7.18799690669511
3 7.22670006735493
4 7.26112968149114
5 7.23624650566163
6 7.25621807813013
7 7.24505283480948
8 7.26335135812127
9 7.26167250683314
10 7.25345425171263
11 7.26519384017686
12 7.24495048164226
13 7.25857867995663
14 7.25336986771952
15 7.26683294848794
16 7.24180702147265
17 7.23474058150847
18 7.2647910152456
19 7.24400381789143
20 7.24784446709072
21 7.24809858336005
22 7.26178394879536
23 7.26070383763911
24 7.25112428522319
25 7.24816915247125
26 7.26360584672995
27 7.23585375450882
28 7.22942550460396
29 7.26302870888556
30 7.23598788718521
31 7.25754175664301
32 7.23629372498193
33 7.26111513770515
34 7.22864261781731
35 7.22393997614655
36 7.25005959216838
37 7.25382536754609
38 7.25596996116991
39 7.25019105034177
40 7.25491969132221
41 7.24210649264588
42 7.24538496420518
43 7.23347310835593
44 7.24207735630196
45 7.25668521819925
46 7.26176935315116
47 7.26590809808308
48 7.24743715221391
49 7.26122335101552
50 7.26585724891097
51 7.25536438913167
52 7.25750564505949
53 7.25343306522989
54 7.25278231373456
55 7.26076979435528
56 7.26616694379142
57 7.2603806002866
58 7.24141093156492
59 7.24975211424692
60 7.26063089281825
61 7.25597912094784
62 7.251506193992
63 7.2440355450529
64 7.26871722379416
65 7.2650849929016
66 7.26602061390203
67 7.24789288552314
68 7.26347738469304
69 7.24291058527511
70 7.25310702772016
71 7.24985434764289
72 7.24203178871523
73 7.23462261070694
74 7.24750524328415
75 7.23709093613041
76 7.25587043863701
77 7.22642353374481
78 7.26639708760609
79 7.24333388230151
80 7.24793295472362
81 7.25547440808309
82 7.2529257923983
83 7.28131928493449
84 7.23617804433528
85 7.23308353204612
86 7.26057176018274
87 7.23408430753446
88 7.24481037166406
89 7.23580804197908
90 7.22420815213505
91 7.24078178281555
92 7.23921283411915
93 7.24133903490446
94 7.25125299773285
95 7.24914285082326
96 7.25899800221344
97 7.26764529823436
98 7.25100906857305
99 7.24088443804652
100 7.23947782110064
101 7.26817668595079
};\addlegendentry{PLD} 
\end{axis}

\end{tikzpicture}
    \end{adjustbox}    
    \caption{Trajectory of policy values computed by PLD for a non-rectangular uncertainty set of fixed size~$r=10$.
    }
    \label{fig:PE:traj:non:rectangular}
    \end{subfigure}
    \vspace{-15pt}
    \caption{Comparison of PLD (Algorithm~\ref{alg:PLD}) against FW (Algorithm~\ref{alg:CPI}) on a stochastic GridWorld problem with a non-rectangular uncertainty set.}
    \label{fig:non-rect}
    % \vspace{-20pt}
\end{figure}
% \vspace{-20pt}
\begin{table}[htb!]
\small
    \centering
    \vspace{-7pt}
    \begin{tabular}{|l|r|r|r|r|}
\hline
\!\textbf{Runtime [s]} \! &  $r=0.01\quad$ &  $r=0.1\quad\;\;$  &  $r=1\quad\;\;$  & $r=10\quad\;\;$\\
\hline
PLD (Algorithm~\ref{alg:PLD}) &  $357.56\,(6.08)$ & $310.05\,(6.20)$ &  $428.57\,(6.73)$ & $370.87\,(91.08)$\\
\hline
FW (Algorithm~\ref{alg:CPI}) & $499.48\quad\;\;$ & $850.60\quad\;\;$ & $948.65\quad\;\;$ & $1{,}950.04\quad\;\;$\\
\hline
\end{tabular}
    \caption{Runtimes of PLD and FW  for non-rectangular uncertainty sets. For PLD we report both means and standard deviations (in parenthesis) over~$20$ simulation runs.}
    \label{tab:runtime:nonrect}
    \vspace{-15pt}
\end{table}
%%%%%%%%%%%%%%%%%%%%%%
\vspace{-5pt}

\subsection{Garnet MDPs: Rectangular Uncertainty Sets}
\label{ssec:rec:amb}

%%%%%%%%%%%
The purpose of the second experiment is to show that FW may solve robust policy evaluation problems with $s$-rectangular uncertainty sets faster than the state-of-the-art method for this problem class developed in~\cite{wang2022convergence}. We use the Generalized Average Reward Non-stationary Environment Test-bench (Garnet)~\cite{archibald1995generation,bhatnagar2009natural} to generate random reference transition kernels~$P_{\text{ref}}$ with a prescribed number of states and actions and with a prescribed branching parameter~$b\in[0,1]$. By definition, $b$ determines the proportion of states that are reachable from any given state-action pair in one single transition. We set the branching parameter to~$b=1$ and the discount factor to~$\gamma=0.6$, and we generate the cost~$c(s,a)$ corresponding to any $a\in\mc A$ and $s\in\mc S$ randomly from the uniform distribution on~$[0,1]$. The initial state~$s_0$ follows the uniform distribution over~$\mc S$. In addition, we fix a policy~$\pi\in\Pi$ defined through~$\pi(a|s)=v(s,a)/\sum_{a'\in\mc A} v(s,a')$, where $v(s,a)$ is sampled uniformly from~$\{1,\ldots,10\}$ for every~$s\in\mc S$ and~$a\in\mc A$.
%Here we set an extremely small~$\gamma$ to ensure a faster theoretical convergence of policy gradient algorithm under comparison. This allows us to concentrate on assessing the scalability of the algorithms with respect to the sizes of state and action spaces.
%
%We extend each Garnet MDP instance to an instance of the robust policy evaluation problem $\Value{\pi}{\star}(\rho)=\max_{P\in\mc P} \sum_{s_0\in\mc S}\rho(s_0) \Value{\pi}{P}(s_0)$ associated with a prescribed policy $\pi\in\Pi$ and a prescribed 
Finally, we assume that~$\mc P$ constitutes an $s$-rectangular uncertainty set of the form
% \vspace{-5pt}
\[
    \mc P=\left\{P\in\Delta(\mc S)^{S\times A}\;:\; \| P(\cdot|s,\cdot)-P_{\rm ref}(\cdot|s,\cdot)\|_{1} \leq 5 \ \forall s\in\mc S \right\}.
    % \vspace{-5pt}
\]
We solve the resulting instances of the robust policy evaluation problem~\eqref{expr:DMDP} with FW and with~\cite[Algorithm~2]{wang2022convergence}, a state-of-the-art projected gradient descent method.

In the second experiment we seek $\delta$-optimal solutions of~\eqref{expr:DMDP} for~$\delta=0.01$. To this end, we use FW with tolerance~$\epsilon=\delta S/(1-\gamma)$, stepsizes~$\{\alpha_m\}_{m=1}^\infty$ chosen as in Theorem~\ref{thm:CPI:PE} and initial iterate~$P^{(0)}=P_{\text{ref}}$. In addition, we use~\cite[Algorithm~2]{wang2022convergence} with initial iterate~$P^{(0)}=P_{\text{ref}}$ and stepsize~$(1-\gamma)^3/(2\gamma S^2)$ as suggested by~\cite[Theorem~4.4]{wang2022convergence}. The above estimates of the iteration complexity and the distribution mismatch coefficient imply that we would have to run~\cite[Algorithm~2]{wang2022convergence} over~$32\gamma S^5 A/(\delta^2(1-\gamma)^6)$ iterations in order to guarantee that it outputs a $\delta$-optimal solution. Unfortunately, this is impractical. For example, already our smallest test problem with only~$S=100$ states would require more than~$10^{18}$ iterations. We thus use the inequality
% \vspace{-5pt}
\[
    |\Value{\pi}{P^{(m+1)}}(\rho)-\Value{\pi}{P^{(m)}}(\rho)|\le 2\times 10^{-5}
    % \vspace{-5pt}
\]
as a heuristic termination criterion. Even though it has
no theoretical justification, this criterion ensures that~\cite[Algorithm~2]{wang2022convergence} terminates within a reasonable amount of time and outputs similar value estimates as FW with a maximum difference of~$2.5\%$.

The direction-finding subproblems of FW as well as the projection subproblems of~\cite[Algorithm~2]{wang2022convergence} are solved with GUROBI. To faithfully assess algorithmic efficiency, we record the solver times for these most time-consuming subroutines. For all other processes we record the wall-clock time. Table~\ref{tab:runtime} reports the overall runtimes of FW and~\cite[Algorithm~2]{wang2022convergence} (based on the authors' code available from GitHub\footnote{\url{https://github.com/JerrisonWang/ICML-DRPG}}) averaged over 20 random instances with $A=10$ actions and increasing numbers of states.

As expected from the analysis of the iteration complexities, FW is significantly faster than~\cite[Algorithm~2]{wang2022convergence} on instances with large state spaces. The value estimates of both algorithms differ at most $2.5\%$, with FW outputting a more accurate solution. %The advantage may be attributed to the projection-free nature of our method. 

\begin{table}[htb!]
\small
    \centering
    \begin{tabular}{|l|c|c|c|c|}
\hline
% ,\,A=10
\textbf{Runtime [s]}  &  $S=100$  &  $S=200$  &  $S=300$  &  $S=400$\\
\hline
\cite[Algorithm 2]{wang2022convergence} &  $51.14$ & $426.83$ &  $1{,}887.47$ & $5{,}328.41 $\\
\hline
FW (Algorithm~\ref{alg:CPI}) & $28.39$ & $209.18$ & $696.60$ & $1{,}120.26 $\\
\hline
\end{tabular}
    \caption{Runtimes of the projected gradient descent algorithm developed in~\cite{wang2022convergence} and FW on Garnet MDP instances with $s$-rectangular uncertainty sets with $A=10$.}
    \label{tab:runtime}
\end{table}

% \vspace{-10pt}
\subsection{Machine Replacement: Non-Rectangular Uncertainty Sets}
\label{ssec:non:rec:amb}
The purpose of the third experiment is to assess the out-of-sample performance of different data-driven policies for MDPs with unknown transition kernels. Our experiment is based on a now standard machine replacement problem described in~\cite{delage2010percentile,wiesemann2013robust}. The goal is to find a repair strategy for a machine whose condition is described by eight ``operative'' states $1, \ldots, 8$ and two ``repair'' states R1 and R2. The available actions are ``do nothing'' or ``repair.'' The states~8, R1 and~R2 incur a cost of~20, 2 and~10 per time period, respectively, whereas no cost is incurred in the other states. The discount factor is set to~$\gamma=0.8$, and the initial state~$s_0$ follows the uniform distribution on~$\mc S$. In addition, we define the transition kernel~$P^0$ as in~\cite[Section~6]{wiesemann2013robust}. The optimal value of the resulting (non-robust) policy improvement problem then amounts to~$5.98$. %\min_{\pi\in\Pi}V^{P^0}_\pi(\rho)=5.98$. 

In the following we assume that~$P^0$ is unknown but falls within a known structural uncertainty set~$\mc P^0$. We specifically assume that some of the $2\times10^2$ transition probabilities are known to vanish such that~$\mc P^0 =\{P^\xi:\xi\in\Xi^0\}$, where $P^\xi$ is an affine function, and $\Xi^0=[0,1]^{25}$ is a hypercube of dimension~$25$. The components of $\xi$ represent different entries of the transition kernel that are neither known to vanish nor determined by the normalization conditions $\sum_{s'\in\mc S} P^0(s'|s,a)=1$ for all $s\in\mc S$ and $a\in\mc A$. Sometimes we will additionally assume that certain transition probabilities are known to be equal, in which case~$\Xi^0=[0,1]^5$ reduces to a hypercube of dimension~$5$. Full details about these structural assumptions are provided in~\cite[Section~6]{wiesemann2013robust}.

In addition to structural information, there is statistical information about~$P^0$, that is, $P^0$ is indirectly observable through a history of~$n$ states and actions generated under a known policy~$\pi^0$. We assume that $\pi^0$ chooses the actions ``do nothing'' and ``repair'' in each operative state $1, \ldots, 7$ with probabilities $0.8$ and $0.2$, respectively. In the states $8$ and R2, $\pi^0$ always chooses the action ``repair'', and in state R1, $\pi^0$ always chooses the action ``do nothing.'' In the following we use~$\xi^n$ to denote the maximum likelihood estimator for the parameter~$\xi^0\in\Xi^0$ that generates the unknown true transition kernel $P^0=P^{\xi^0}$. Following \cite[Section~5]{wiesemann2013robust}, one can use the observation history of length~$n$ to construct an ellipsoidal confidence region~$\Xi^n\subseteq\Xi^0$ centered at~$\xi^n$ that contains $\xi^0$ with probability at least $1-\alpha$ for any prescribed $\alpha\in[0,1]$. It is then natural to construct an uncertainty set $\mc P^n =\{P^\xi:\xi\in\Xi^n\}$ that amalgamates all structural and statistical information about~$P^0$ and is guaranteed to contain the data-generating kernel~$P^0$ with probability $1-\alpha$. A related but simpler recipe for constructing uncertainty sets using maximum likelihood estimation is sketched in Appendix~\ref{example:MLE:detail} for illustrative purposes. Full details are provided in \cite[Section~5]{wiesemann2013robust}. 

The uncertainty set~$\mc P^n$ is non-rectangular, and thus the corresponding robust policy improvement problem is hard. A sequential convex optimization procedure that solves a decision rule approximation of the robust policy improvement problem is described in \cite[Algorithm~4.1]{wiesemann2013robust}. To our best knowledge, this is the only existing method for addressing robust MDPs with non-rectangular uncertainty sets. Replacing~$\mc P^n$ with its $s$-rectangular or even its $(s,a)$-rectangular hull leads to a simpler robust policy improvement problem that can be solved exactly and efficiently via dynamic programming. However, the resulting optimal policy is dominated by the policy output by \cite[Algorithm~4.1]{wiesemann2013robust} in that it generates up to $30\%$ or even $60\%$ higher out-of-sample net present costs, respectively, see \cite[Table~3]{wiesemann2013robust}.

% {\color{blue} It is shown in~ that using this non-rectangular uncertainty set can reduce the  by up to~$30\%$ and~$60\%$ of the $s$- and $(s,a)$-rectangular projections of~$\mc P$ when $\xi^0$ is of dimension~$5$. 
% This observation motivates us to focus on the out-of-sample performance of Algorithms for solving robust MDPs with non-rectangular uncertainty sets.
% }

Unlike \cite[Algorithm~4.1]{wiesemann2013robust}, ACA (Algorithm~\ref{alg:PG-min-oracle}) uses no decision rule approximation and computes near-optimal solutions to the robust policy improvement problem of any prescribed accuracy (see Theorem~\ref{thm:PO:convergence}). We will now show numerically that the near-optimal policies found by ACA dominate the approximately optimal policies found by \cite[Algorithm~4.1]{wiesemann2013robust} in terms of out-of-sample net present cost under~$P^0$. Throughout the experiment we employ ACA with iteration number $K=100$ and stepsize $\eta=0.05$. The critic's subproblem computes near-optimal solutions to the robust policy evaluation problem by using PLD (Algorithm~\ref{alg:PLD}) with initial iterate~$\xi^{(0)}=\xi^n$, Gibbs parameter $\beta=450$, stepsize $\eta=0.07$ and iteration number $M=50$. We work with a variant of PLD that outputs the best iterate found during execution. 

Tables~\ref{tab:xi:5:improvement} and~\ref{tab:xi:25:improvement} compare the out-of-sample costs of the policies found by ACA and \cite[Algorithm~4.1]{wiesemann2013robust} under the assumption of full ($\xi\in\mathbb R^5$) and partial ($\xi\in\mathbb R^{25}$) structural information, respectively, as a function of the length~$n$ of the observation history and the coverage probability $1-\alpha$ of the uncertainty set. The out-of-sample costs corresponding to~\cite[Algorithm~4.1]{wiesemann2013robust} in Table~\ref{tab:xi:5:improvement} are directly borrowed from~\cite[Table~3]{wiesemann2013robust}. Conversely, the out-of-sample costs corresponding to~\cite[Algorithm~4.1]{wiesemann2013robust} in Table~\ref{tab:xi:25:improvement} are computed using the authors' source code in C++ (private communication).

Table~\ref{tab:xi:5:improvement} shows that when the transition kernel has only $5$ degrees of freedom, both policies generate an out-of-sample cost close to the optimal value $5.98$ of the classical policy improvement problem under the unknown true transition kernel~$P^0$. Moreover, the out-of-sample costs of the two policies differ at most by~$0.5\%$. These observations are not surprising because kernels with only $5$~degrees of freedom are easy to learn and because the uncertainty set~$\mc P^n$ is small already for small sample sizes~$n$. In this case, the decision rule approximation underlying \cite[Algorithm~4.1]{wiesemann2013robust} is highly accurate. Algorithm~\ref{alg:PG-min-oracle}, which is designed for uncertainty sets of arbitrary size and solves the critic's subproblem with a randomized PLD scheme, slightly outperforms the benchmark method only for the smallest sample sizes considered.

% This similarity arises due to the incorporation of additional structural knowledge, which significantly constrains the degrees of freedom within the confidence region. In this case, the method in~\cite{wiesemann2013robust} designed for small uncertainty sets demonstrates comparable performance with Algorithm~\ref{alg:PG-min-oracle}, which is designed for uncertainty sets of arbitrary size. 
% We observe a lower policy value of Algorithm~\ref{alg:PG-min-oracle} than \cite{wiesemann2013robust} in scenarios with small sample sizes, and a higher policy value when the sample size is large, with the relative difference no more than $0.5\%$. This trend could stem from the fact that Algorithm~\ref{alg:PG-min-oracle}, being globally optimal, performs better with larger uncertainty sets. In contrast, \cite{wiesemann2013robust} is tailored for very small uncertainty sets, thereby potentially offering a slight advantage in such contexts.

\begin{table}[htb!]
\small
\vspace{-3pt}
    \centering
    \begin{tabular}{|c|c|c|c|c|}
        \hline \diagbox[height=18pt]{\quad$n$}{$1-\alpha$} & $80\%$ & $90\%$ & $95\%$ & $99\%$ \\\hline
        $\phantom{1{,}}500$ & $6.02 \,(6.04)$ & $6.02\, (6.04)$ & $6.02 \,(6.04)$ & $6.02 \,(6.06)$\\\hline
        $1{,}000$ & $6.03 \,(6.02)$ & $6.04 \,(6.02)$ & $6.04 \,(6.02)$ & $6.00 \,(6.02)$ \\\hline
        $2{,}500$ & $6.03\, (6.01)$ & $6.03 \,(6.00)$ & $6.02\, (6.00)$ & $6.02 \,(6.01)$\\\hline
        $5{,}000$ & $6.01 \,(5.99)$ & $6.03 \,(5.99)$ & $6.02 \,(5.99)$ & $6.03 \,(5.99)$ \\\hline
    \end{tabular}
    \caption{
    Out-of-sample costs of the policies found by ACA and \cite[Algorithm~4.1]{wiesemann2013robust} (in parenthesis) under full structural information (kernel with $5$ degrees of freedom). 
    }
    \label{tab:xi:5:improvement}
    \vspace{-20pt}
\end{table}

Table~\ref{tab:xi:25:improvement} shows that when the transition kernel has $25$ degrees of freedom, then Algorithm~\ref{alg:PG-min-oracle} outperforms~\cite[Algorithm~4.1]{wiesemann2013robust} uniformly across all values of~$n$ and~$1-\alpha$. The advantage is most significant when the uncertainty set is large ({\em i.e.}, for $n\leq 1{,}000$).
%For smaller uncertainty sets (sample size $2{,}500$ and $5{,}000$), the out-of-sample policy values are relatively close, while Algorithm~\ref{alg:PG-min-oracle} generally offers $5\%-10\%$ advantage over~\cite{wiesemann2013robust}. 
We also highlight that the average wall-clock time for solving all problem instances with Algorithm~\ref{alg:PG-min-oracle} amounts to~$179.16$ seconds. The average solver time consumed by~\cite[Algorithm~4.1]{wiesemann2013robust}, on the other hand, amounts to~$351.24$ seconds.

\begin{table}[htb!]
\small
\vspace{-5pt}
    \centering
    \begin{tabular}{|c|c|c|c|c|}
        \hline 
        \diagbox[height=18pt]{\quad$n$}{$1-\alpha$}& $80\%$ & $90\%$ & $95\%$ & $99\%$ \\\hline
        $\phantom{1{,}}500$ & $8.34 \,(15.72)$ & $8.40\, (14.24)$ & $6.48 \,(13.44)$ & $7.41\, (19.29)$\\\hline
        $1{,}000$ & $6.57\, (8.45)$ & $6.27\, (9.79)$ & $6.96\, (10.60)$ & $6.77\, (10.02)$ \\\hline
        $2{,}500$ & $6.26\, (6.55)$ & $6.08\, (6.84)$ & $6.36 \,(6.82)$ & $6.20\, (8.47)$\\\hline
        $5{,}000$ & $6.23 \,(6.64)$ & $6.49 \,(6.53)$ & $6.29 \,(6.50)$ & $6.24\, (6.54)$ \\\hline
    \end{tabular}
    \caption{Out-of-sample costs of the policies found by ACA and \cite[Algorithm~4.1]{wiesemann2013robust} (in parenthesis) under partial structural information (kernel with $25$ degrees of freedom). 
    }
    \label{tab:xi:25:improvement}
        \vspace{-20pt}
\end{table}

\vspace{-5pt}
\section*{Acknowledgements}
This work was supported as a part of the NCCR Automation, a National Center of Competence in Research, funded by the Swiss National Science Foundation (grant number 51NF40\_225155).
The authors are indebted to George Lan and Yan Li for helpful comments on an earlier version of this paper, and to Ilyas Fatkhullin for useful discussions.
% \newpage
% \vspace{-5pt}

\appendix

\section{Construction of Uncertainty Sets via Maximum Likelihood Estimation}\label{example:MLE:detail}
We now review a standard procedure for constructing an uncertainty set for the transition kernel of an MDP as described in~\cite[Section~5]{wiesemann2013robust}. This uncertainty set is statistically optimal in a precise sense but fails to be rectangular.

% {\color{blue} Assume that there is a known solid parameter set~$\Xi^0\subseteq \mb R^q$ such that the true unknown transition kernel~$P^0$ is an affine transformation of some parameter~$\xi^0$ that lies within the interior of~$\Xi^0.$
% }
Assume for ease of exposition that it is possible to move from any state of the MDP to any other state in one single transition, that is, all entries of the unknown transition kernel are strictly positive. The uncertainty set can thus be expressed as the image of a solid parameter set~$\Xi\subseteq \mb R^{q}$ of dimension $q=SA(S-1)$ under an affine function~$P^\xi$. Specifically, there exists a bijection $g: \mc S\times \mc A\times (\mc S\backslash \{S\})\rightarrow \{1,\ldots,q\}$, and any such bijection can be used to construct a valid function~$P^\xi$ defined through $P^{\xi}(s' | s, a) = \xi_{g(s, a, s')}$ for all $s \in \mathcal{S}$, $a \in \mathcal{A}$ and $s' \in \mathcal{S}\backslash \{S\}$, and $P^{\xi}(S | s, a) = 1 - \sum_{s' =1}^{S-1} \xi_{g(s, a, s')}$ for all $s \in \mathcal{S}$ and $a \in \mathcal{A}$. The largest imaginably uncertainty set $\mc P_0=\Delta(\mc S)^{S\times A}$ of all possible transition kernels can then be expressed as the image of the parameter set
\begin{equation*}
    \Xi^0 = \left\{\xi \in \mb R^q_+: \sum_{s' =1}^{S-1} \xi_{g\left(s, a, s'\right)} \leq 1 \ \forall s\in \mc S,\, a\in \mc A\right\}
\end{equation*}
under $P^\xi$. In the following we assume that the decision maker has access to a state-action observation history $(s_0,a_0,\ldots,s_{n-1},a_{n-1})\in(\mc S\times\mc A)^n$ of the MDP generated under some known policy $\pi^0\in\Pi$ and the unknown true transition 
kernel~$P^{\xi^0}$ encoded by~$\xi^0\in\Xi^0$. 
% kernel~$P^{\xi^0}$ encoded by~$\xi^0\in\Xi^0$. 
The log-likelihood of observing this history under any $\xi \in \Xi$ is given by
\[
    \ell_n(\xi)=\sum_{t=0}^{n-2} \log [P^{\xi}(s_{t+1} | s_t, a_t)]+\zeta, \quad \text {where}\quad \zeta=\log [\rho(s_0)]+\sum_{t=0}^{n-1} \log [\pi^0(a_t | s_t)]
\]
is an irrelevant constant independent of~$\xi$, and $\rho$ represents the initial state distribution. One can show that the maximum likelihood estimator~$\xi^n$ that maximizes $\ell_n(\xi)$ over~$\Xi^0$ corresponds to the kernel of empirical transition probabilities \cite[Remark~6]{wiesemann2013robust}. This means that~$P^{\xi^n}(s'|s,a)$ coincides with number of observed transitions from~$(s,a)$ to~$s'$, normalized by the length~$n$ of the observation history. One can use the maximum likelihood estimator~$\xi^n$ as well as the log-likelihood function~$\ell_n(\xi)$ to construct a confidence set~$\{\xi\in\Xi^0: \ell_n(\xi)\geq \ell_n(\xi^n)-\delta \}$ for~$\xi^0.$
% ~$\xi^0$. 
Indeed, this set contains~$\xi^0$ with probability $1-\alpha$ asymptotically for large $n$ if~$\delta$ is set to one half of the $(1-\alpha)$-quantile of the chi-squared distribution with $S-1$ degrees of freedom \cite[Theorem~5]{wiesemann2013robust}. This statistical guarantee persists if we approximate the log-likelihood function $\ell_n$ by its second-order Taylor expansion 
\[
    \varphi_n(\xi) =\ell_n(\xi^n)-\frac{1}{2}\left(\xi-\xi^n\right)^{\top}\left[\nabla_{\xi}^2 \ell_n\left(\xi^n\right)\right]\left(\xi-\xi^n\right).
\]
One can show that, as~$n$ grows, the scaled Hessian matrix $\nabla_{\xi}^2 \ell_n(\xi^n)/n$ converges in probability to the Fisher information matrix, which we denote as~$I(\xi^0)$ \cite[Section~2]{billingsley1961statistical}. In addition, the scaled estimation error~$\sqrt{n}(\xi^n-\xi^0)$ converges in distribution to the normal distribution with mean~$0$ and covariance matrix~$I(\xi^0)^{-1}$ \cite[Theorem~2.2]{billingsley1961statistical}.
A generalization of the classical Cramér-Rao inequality ensures that the covariance matrix of any unbiased estimator for~$\xi^0$ is bounded below by~$I(\xi^0)^{-1}/n$ in Loewner order asymptotically for large~$n$ %\cite[Theorem~1]{bai1975efficient}
\cite[Remark~7.9]{liese2008statistical}. In conjunction, these findings suggest that $\Xi^n=\{\xi\in\Xi^0: \varphi_n(\xi)\geq \ell_n(\xi^n)-\delta \}$ constitutes the smallest possible $(1-\alpha)$-confidence set for $\xi^0$ asymptotically for large~$n$. The uncertainty set $\mc P=\{P^{\xi}: \xi\in\Xi^n\}$ therefore enjoys a statistical efficiency property. However, it fails to be rectangular~\cite[pp.~173]{wiesemann2013robust}.

\vspace{-5pt}
\section{Auxiliary Lemmas}
The following results will be used throughout the main text. Their proofs are elementary and thus omitted.
% We include short proofs to keep this paper self-contained.
\begin{lemma}[Relations between value functions{~\cite[Section 3.5]{sutton2018reinforcement}}]\label{lem:recursion:PE}For any $\pi\in\Pi$ and $P\in\mc P$ we have
\begin{enumerate}[label = (\roman*)]
    \item \label{item:valueinQ} $ \Value{\pi}{P}(s)=\sum_{a\in\mc A} \pi(a|s)\Q{\pi}{P}(s, a)$ for all $ s\in\mc S$,
    \item \label{item:QinG}$ \Q{\pi}{P}(s, a)=c(s,a)+\gamma\sum_{s'\in\mc S} P(s'|s,a)\Value{\pi}{P}(s') = \sum_{s'\in\mc S} P(s'|s,a) \Qsp{\pi}{P}(s,a,s')$ for all $s\in\mc S$ and $a\in\mc A$,
    \item \label{item:GinQ}$ \Qsp{\pi}{P} (s, a,s')=c(s,a)+ \gamma \sum_{a'\in\mc A} \pi(a'|s') \Q{\pi}{P}(s',a')=c(s,a)+\gamma \Value{\pi}{P}(s')$ for all $ s,s'\in\mc S$ and $a\in\mc A.  $
\end{enumerate}
\end{lemma}

\begin{proof}[Proof of Lemma~\ref{lem:recursion:PE}]
    %The claims follow from~\cite[Section 3.5]{sutton2018reinforcement}. 
    As for Assertion~\ref{item:valueinQ}, we have 
    \begin{align*}
        \Value{\pi}{P} (s)&=\mb E^P_\pi\left[\sum_{t=0}^{\infty} \gamma^t c(S_t, A_t) \mid S_0=s\right]
        \\&=\sum_{a\in\mc A}\pi(a|s) \mb E^P_\pi\left[\sum_{t=0}^{\infty} \gamma^t c(S_t, A_t) \mid S_0=s,A_0=a\right]=\sum_{a\in\mc A}\pi(a|s)\Q{\pi}{P}(s,a),
    \end{align*}
where the second equality follows from the law of total expectation and~\eqref{def:pi:pi}, and the last equality follows from the definition of $\Q{\pi}{P}(s,a)$.
Next, we prove Assertion~\ref{item:GinQ}, which will help us to prove Assertion~\ref{item:QinG}. By the definition of~$\Qsp{\pi}{P}(s,a,s'),$ we have
    \begin{align*}
        \Qsp{\pi}{P} (s, a,s')&=\mb E^P_\pi\left[\sum_{t=0}^{\infty} \gamma^t c(S_t, A_t) \mid S_0=s, A_0=a,S_1=s'\right]
        \\&=\mb E^P_\pi\left[c(s,a)+\sum_{t=1}^{\infty} \gamma^t c(S_t, A_t) \mid S_1=s'\right]
        \\&=c(s,a)+\gamma \sum_{a'\in\mc A}\mb E^P_\pi\left[\sum_{t=0}^{\infty} \gamma^t c(S_t, A_t) \mid S_0=s', A_0=a'\right]\pi(a'|s')
        \\&=c(s,a)+ \gamma \sum_{a'\in\mc A} \pi(a'|s') \Q{\pi}{P}(s',a')
        \\&=c(s,a)+ \gamma \Value{\pi}{P}(s'),
    \end{align*}
where the second equality holds because $\{(S_t,A_t)\}_{t=1}^\infty$ is a Markov chain and because~$A_t$ is independent of this Markov chain conditional on~$S_t$ under~$\mb P^P_\pi$. The third equality follows from law of total expectation and~\eqref{def:pi:pi} together with an index shift $t\leftarrow t+1$, the fourth equality follows from the definition of $\Q{\pi}{P}(s,a)$, and the last equality follows from Assertion~\ref{item:valueinQ}.
As for Assertion~\ref{item:QinG}, finally, we have
    \begin{align*}
        \Q{\pi}{P}(s, a)&=\mb E^P_\pi\left[\sum_{t=0}^{\infty} \gamma^t c(S_t, A_t) \mid S_0=s,A_0=a\right]
        \\&=\sum_{s'\in\mc S} P(s'|s,a)\mb E^P_\pi\left[\sum_{t=0}^{\infty} \gamma^t c(S_t, A_t) \mid S_0=s,A_0=a,S_1=s'\right]
        \\&= \sum_{s'\in\mc S} P(s'|s,a) \Qsp{\pi}{P}(s,a,s')
        \\&=c(s,a)+ \gamma \sum_{s'\in\mc S} P(s'|s,a) \Value{\pi}{P}(s'),
    \end{align*}
    where the second equality follows from the law of total expectation and~\eqref{def:Q:pi}, the third equality follows from the definition of $\Qsp{\pi}{P}(s,a,s')$, and the fourth equality holds thanks to Assertion~\ref{item:GinQ}.
\end{proof}

\begin{lemma}[Relation between the advantage function and the action-next state value function]\label{lemma:advantage:Qsp}
For any $P,P'\in\mc P$ and $\pi\in\Pi,$ we have
\begin{align*}
    \sum_{s'\in\mc S} (P(s'|s,a)-P'(s'|s,a)) \Qsp{\pi}{P'}(s,a,s')= \sum_{s'\in\mc S} P(s'|s,a) \A{\pi}{P'}(s,a,s') \quad \forall s\in\mc S, a\in\mc A.
\end{align*}
\end{lemma}
% \begin{proof}
% We have
% \begin{align*}
%   & \sum_{s'\in\mc S} (P(s'|s,a)-P'(s'|s,a)) \Qsp{\pi}{P'}(s,a,s')
%     \\&=\sum_{s'\in\mc S} (P(s'|s,a)-P'(s'|s,a)) (\Qsp{\pi}{P'}(s, a,s')-\Q{\pi}{P'}(s, a)+\Q{\pi}{P'}(s, a))
%  \\&=\sum_{s'\in\mc S} P(s'|s,a) (\Qsp{\pi}{P'}(s, a,s')-\Q{\pi}{P'}(s, a)) -\sum_{s'\in\mc S}P'(s'|s,a)(\Qsp{\pi}{P'}(s, a,s')-\Q{\pi}{P'}(s, a))
%  \\&\quad+\sum_{s'\in\mc S}(P(s'|s,a)-P'(s'|s,a))\Q{\pi}{P'}(s, a)
%  \\&=\sum_{s'\in\mc S} P(s'|s,a) (\Qsp{\pi}{P'}(s, a,s')-\Q{\pi}{P'}(s, a))
%  =\sum_{s'\in\mc S} P(s'|s,a) \A{\pi}{P'}(s,a,s'),
% \end{align*}
% where the third equality follows from Lemma~\ref{lem:recursion:PE}\ref{item:QinG} and the fact that $\sum_{s'\in\mc S}P(s'|s,a)=\sum_{s'\in\mc S}P'(s'|s,a)=1$, and the last equality follows from the definition of $\A{\pi}{P'}$.
% \end{proof}

% \begin{lemma}[Properties of the value function]\label{lemma:value:smooth}
    
% \end{lemma}
\begin{lemma}[Frobenius distance between policies]\label{frob:bound:pi}
We have $\|\pi'-\pi\|_{\mathbf F}\le \sqrt{2S}$ for any $\pi,\pi'\in\Pi.$
\end{lemma}
% \begin{proof}
% For every $s\in\mc S$ we have \[\|\pi'(\cdot|s)-\pi(\cdot|s)\|_2^2=\|\pi'(\cdot|s)\|_2^2+\|\pi(\cdot|s)\|_2^2- 2\langle \pi'(\cdot|s), \pi(\cdot|s)\rangle\le 2,\] where the inequality holds because $\pi(\cdot|s),\pi'(\cdot|s)\in\Delta(\mc A)$, which implies that $\|\pi(\cdot|s)\|_2^2\le 1$, $\|\pi'(\cdot|s)\|_2^2\le 1$ and $\langle \pi'(\cdot|s), \pi(\cdot|s)\rangle\ge 0$. 
% By the definition of the Frobenius norm we then have $\|\pi'-\pi\|_{\mathbf F}^2=\sum_{s\in\mc S} \|\pi'(\cdot|s)-\pi(\cdot|s)\|_2^2\le 2S.$ Thus, the claim follows.
% \end{proof}
Inspired by~\cite[Lemma 3]{thekumparampil2019efficient}, we present a generalization of Danskin's theorem for optimization problems with a smooth but not necessarily convex objective functions.
\begin{theorem}[Danskin's theorem]\label{thm:danskin}
Let $\mathcal X\subseteq\mathbb R^n$ be an open convex set, $\mathcal{Y} \subseteq \mathbb{R}^m$ 
an arbitrary compact set and $f: \mathcal X \times \mc Y \rightarrow \mathbb{R}$ a continuous function such that $f(x, y)$ is $\ell$-smooth in~$x$ for each $y \in \mc Y$ and some $\ell\ge0$. In addition, suppose that $\nabla_x f(x, y)$ is continuous in~$y$ for each $x\in\mathcal X$. Then, the optimal value function $\Phi(x)=\max _{y \in \mc Y} f(x, y)$ is~$\ell$-weakly convex, and its subdifferential is given by
\[
\partial \Phi(x)=\operatorname{conv}\left\{\nabla_x f(x, y^\star) \mid y^\star \in \underset{y \in \mc Y}{\arg \max }\, f(x, y)\right\}.
\]
\end{theorem}
\bibliographystyle{siamplain}
\bibliography{M163125_references.bib}% \newpage

\end{document}